\documentclass{article}

\usepackage{arxiv}

\usepackage[utf8]{inputenc} % allow utf-8 input
\usepackage[T1]{fontenc}    % use 8-bit T1 fonts
\usepackage{hyperref}       % hyperlinks
\usepackage{url}            % simple URL typesetting
\usepackage{booktabs}       % professional-quality tables
\usepackage{amsfonts}       % blackboard math symbols
\usepackage{nicefrac}       % compact symbols for 1/2, etc.
\usepackage{microtype}      % microtypography
\usepackage{lipsum}		% Can be removed after putting your text content
\usepackage{graphicx}
\usepackage{doi}
\usepackage{framed}

\newcommand{\defEnv}[2]{%
    \expandafter\newcommand\csname #1\endcsname{#2}%
}

\newcommand{\useDefEnv}[2]{%
    \begin{#1}%
    \label{#2}%
    \csname #2\endcsname%
    \end{#1}%
}

\newcommand{\autoEnv}[3]{%
    \expandafter\newcommand\csname #2\endcsname{#3}%
    \begin{#1}%
    \label{#2}%
    #3%
    \end{#1}%
}

\newcommand{\refEnv}[1]{%
    \textbf{\cref{#1}. }\textit{\csname #1\endcsname}%
}

\title{CaCuTe: Casual Cubic-Model Technique for Faster Optimization}

\date{} 					% Or removing it

\author{
    \hspace{1mm}Nazarii Tupitsa
    \\
	Department of Machine Learning\\
	MBZUAI\\
}

\renewcommand{\shorttitle}{CaCuTe: Casual Cubic-Model Technique}

\hypersetup{
pdftitle={A template for the arxiv style},
pdfsubject={q-bio.NC, q-bio.QM},
pdfauthor={David S.~Hippocampus, Elias D.~Striatum},
pdfkeywords={First keyword, Second keyword, More},
}

\usepackage{microtype}
\usepackage{graphicx}
\usepackage{subfigure}
\usepackage{booktabs} % for professional tables

\usepackage{hyperref}

\usepackage{algorithm}
\usepackage{algorithmicx}
\usepackage[noend]{algpseudocode}

\usepackage{amsmath}
\usepackage{amssymb}
\usepackage{mathtools}
\usepackage{amsthm}

\usepackage[capitalize,noabbrev]{cleveref}
\Crefname{assumption}{Assumption}{Assumptions}

\theoremstyle{plain}
\newtheorem{theorem}{Theorem}[section]

\newtheorem{lemma}[theorem]{Lemma}
\newtheorem{corollary}[theorem]{Corollary}
\theoremstyle{definition}

\newtheorem{assumption}[theorem]{Assumption}
\newtheorem{example}[theorem]{Example}
\theoremstyle{remark}
\newtheorem{remark}[theorem]{Remark}

\usepackage{enumitem}

\usepackage[colorinlistoftodos,bordercolor=orange,backgroundcolor=orange!20,linecolor=orange,textsize=normalsize]{todonotes}

\usepackage{hyperref}
\usepackage{multirow}
\usepackage{colortbl}
\definecolor{niceblue}{rgb}{0.0,0.19,0.56}
\usepackage{url}
\usepackage[utf8]{inputenc} % allow utf-8 input
\usepackage[T1]{fontenc}    % use 8-bit T1 fonts

\usepackage{url}            % simple URL typesetting
\usepackage{booktabs}       % professional-quality tables
\usepackage{amsfonts}       % blackboard math symbols
\usepackage{nicefrac}       % compact symbols for 1/2, etc.
\usepackage{microtype}      % microtypography
\usepackage{bbm}            % for indicators 1
\usepackage{bm}             % bold symbols

\usepackage{eqparbox}
\newcommand{\eqdef}{\stackrel{\text{def}}{=}}

\newcommand{\algname}[1]{{\sf  #1}\xspace}

\DeclareMathOperator*{\argmin}{arg\!min}

\newcommand{\E}{\mathbb{E}}

\newcommand{\R}{\mathbb{R}}

\newcommand{\cO}{{\cal O}}

\newcommand{\mA}{{\bf A}}

\newcommand{\mI}{{\bf I}}

\newcommand\swapifbranches[3]{#1{#3}{#2}}
\makeatletter
\MHInternalSyntaxOn
\patchcmd{\DeclarePairedDelimiter}{\@ifstar}{\swapifbranches\@ifstar}{}{}
\MHInternalSyntaxOff
\makeatother

\DeclarePairedDelimiterX{\inp}[2]{\langle}{\rangle}{#1, #2}
\DeclarePairedDelimiterX{\abs}[1]{\lvert}{\rvert}{#1}
\DeclarePairedDelimiterX{\roundup}[1]{\lceil}{\rceil}{#1}
\DeclarePairedDelimiterX{\norm}[1]{\lVert}{\rVert}{#1}
\DeclarePairedDelimiterX{\cbr}[1]{\{}{\}}{#1} % curly bracket
\DeclarePairedDelimiterX{\rbr}[1]{(}{)}{#1} % round bracket
\DeclarePairedDelimiterX{\sbr}[1]{[}{]}{#1} % 

\DeclarePairedDelimiterXPP{\PositivePart}[1]{}{\lbrack}{\rbrack}{_+}{#1}

\DeclarePairedDelimiterX{\ClosedClosedInterval}[2]{\lbrack}{\rbrack}{#1, #2}
\DeclarePairedDelimiterX{\OpenOpenInterval}[2]{\lparen}{\rparen}{#1, #2}
\DeclarePairedDelimiterX{\ClosedOpenInterval}[2]{\lbrack}{\rparen}{#1, #2}
\DeclarePairedDelimiterXPP{\DualNorm}[1]{}{\lVert}{\rVert}{_*}{#1}
\DeclarePairedDelimiterX{\OpenClosedInterval}[2]{\lparen}{\rbrack}{#1, #2}

\DeclarePairedDelimiterX{\SetBuilder}[2]{\lbrace}{\rbrace}{#1 : #2}
\DeclarePairedDelimiterX{\DualPairing}[2]{\langle}{\rangle}{#1, #2}

\newcommand{\f}{f}

\newcommand{\g}{g}
\newcommand{\x}{x}

\usepackage{xspace}

\usepackage{enumitem}

\begin{document}

\maketitle

% \begin{abstract}
%     CaCuTe (see a cute) is a technique based on establish a local $\mathcal{O}(k^{-2})$ rate for the gradient update $x^{k+1}=x^k-\nabla f(x^k)/\sqrt{H\|\nabla f(x^k)\|}$ under $2H$ Hessian--Lipschitz assumption. The regime-detection relies on Hessian--vector products without forming or factorizing Hessians.
    
%     Incorporating this certificate into cubic-regularized Newton (CRN) and into an accelerated variant switches per iterate between the cubic step and the gradient step, that preserve CRN's global guarantees. Same applied for Adaptive Newton method led to a method with the least wall-clock time.

%     Studying the technique for first order optimization lead to a monotone adaptive parameter free method, which inherits the local $\mathcal{O}(k^{-2})$ rate. We also accommodate several  smoothness relaxation beyond classical gradient--Lipschitznes allowing to state tighter bound including global $\mathcal{O}(k^{-2})$ rate.

%     Finally we extend the techniquee to stochastic optimization.
% \end{abstract}

\begin{abstract}
% We \emph{see acute} decrease! A local $\mathcal{O}(k^{-2})$ rate for the gradient update $x^{k+1}=x^k-\nabla f(x^k)/\sqrt{H\|\nabla f(x^k)\|}$ under a $2H$-Hessian--Lipschitz assumption is established.
We establish a local $\mathcal{O}(k^{-2})$ rate for the gradient update $x^{k+1}=x^k-\nabla f(x^k)/\sqrt{H\|\nabla f(x^k)\|}$ under a $2H$-Hessian--Lipschitz assumption. Regime detection relies on Hessian--vector products, avoiding Hessian formation or factorization.
Incorporating this certificate into cubic-regularized Newton (CRN) and an accelerated variant enables per-iterate switching between the cubic and gradient steps while preserving CRN’s global guarantees. The technique achieves the lowest wall-clock time among compared baselines in our experiments.
In the first-order setting, the technique yields a monotone, adaptive, parameter-free method that inherits the local $\mathcal{O}(k^{-2})$ rate. Despite backtracking, the method shows superior wall-clock performance. Additionally, we cover smoothness relaxations beyond classical gradient--Lipschitzness, enabling tighter bounds, including global $\mathcal{O}(k^{-2})$ rates. 
Finally, we generalize the technique to the stochastic setting.
\end{abstract}

\section{Introduction}

% \subsection{Background}
In this work, we are interested in solving the unconstrained minimization problem $
% \begin{equation}
  \min_{x\in \R^d}  f(x) \label{eq:min_f} $
% \end{equation}
where $f\colon\R^d\to \R$ is a twice-differentiable function with Lipschitz Hessian.

\subsection{Motivation}
% Let us discuss a very simple connection between the cubic regularized Newton (\algname{CRN}) method~\cite{griewank1981modification, nesterov2006cubic} and regularized Newton method~\cite{mishchenko2023regularized}.

The cubic regularized Newton (\algname{CRN}) method~\cite{griewank1981modification} naturally follows from the Lipschitz--Hessian assumption \cite{nesterov2006cubic}. Its update can be written implicitly as
\begin{equation*}
    \label{cubic-opt}
    H\|x^{k+1}-x^k\| (x^{k+1} - x^k) = - \nabla f(x^k) - \nabla^2 f(x^k) (x^{k+1} - x^k)
\end{equation*}
where $H>0$ is a constant, $\mI$ is the identity matrix.
When the quadratic term is negligible, i.e.
\begin{equation*}
    H\|x^{k+1}-x^k\|^3 \gg \inp{\nabla^2 f(x^k)(x^{k+1}-x^k)}{x^{k+1}-x^k}
\end{equation*}
an approximate update formula reads as
\begin{equation*}
    H\|x^{k+1}-x^k\|^3 \approx -\inp{\nabla f(x^k)} {x^{k+1}-x^k}
    \quad \text{or} \quad 
    x^{k+1} \approx x^k - \frac{\nabla f(x^k)}{\sqrt{H\|\nabla f(x^k)\|}}
\end{equation*}
% Equivalently, the explicit approximation
% \begin{equation*}
%     x^{k+1} \approx x^k - \frac{\nabla f(x^k)}{\sqrt{H\|\nabla f(x^k)\|}}
% \end{equation*}
% Similarly, regularized Newton method update
% \begin{equation}
%     x^{k+1} = x^k - (\nabla^2 f(x^k) + \sqrt{H_k\|\nabla f(x^k)\|} \mI)^{-1}\nabla f(x^k)
% \end{equation}

% transforms to
% \begin{equation}
%     x^{k+1} \approx x^k - \rbr*{\sqrt{H\|\nabla f(x^k)\|} \mI}^{-1}\nabla f(x^k)
% \end{equation}
% when
% \begin{equation}
%     \sqrt{H\|\nabla f(x^k)\|} \gg \norm{\nabla^2 f(x^k)}
% \end{equation}
% \todo{inaccurate}
and raises a natural question:
\begin{gather*}
	\textbf{\textit{Can we provably avoid Hessian evaluation and matrix inversion while retaining an $\cO(k^{-2})$ rate?}}
\end{gather*}
This affirmative result shows that the convergence is achievable with a first-order method and naturally leads to the next question.
\begin{gather*}
	\textbf{\textit{How to utilize Hessian smoothness for the first order methods?}}
\end{gather*}
This question is central, since exact second-order steps are often infeasible in modern high-dimensional problems even when Hessian smoothness holds. Our approach leverages smoothness without forming or inverting Hessians. We also extend the framework to stochastic optimization, yielding practical algorithms for large-scale applications.

\subsection{Related works}
\textbf{Cubic Regularization.} Beyond advances in Quasi-Newton methods~\cite{scheinberg2016practical,ghanbari2018proximal,rodomanov2021greedy,rodomanov2021new,jin2022sharpened,berahas2022quasi,kamzolov2023cubic,jin2024exact,scieur2024adaptive,wang2024global}, global convergence guarantees in the convex enhances by cubic regularization~\cite{kamzolov2023cubic,scieur2024adaptive,wang2024global}. However, those methods result in implicit (explicit for \cite{mishchenko2023regularized}) update formula requiring additional line search in each iteration, involving matrix inversions. Our approach enables occasional skipping of these costly operations. We also obtain a result similar to \cite{agafonov2025simple}, which establishes global $\cO(k^{-1})$ rate while having $\cO(k^{-2})$ at the first iterates.

\textbf{Adaptive Gradient Methods.} Gradient descent is attractive for its simplicity and low per-iteration cost, the only hyperparameter is the step size: classical Lipschitz-based analyses pick a constant step to ensure monotone decrease, but this requires the (unknown) Lipschitz constant and is typically overly conservative in practice \cite{malitsky20adaptive}. Backtracking line search \cite{nocedal2006numerical}, which gives theoretical convergence guarantee, is a parameter-free adaptive alternative to constant step sizes. However its try-increase nature can be avoided while keeping adaptivity without any extra computations under convexity and Lipschitz gradients \cite{malitsky20adaptive,malitsky2024adaptive}. Alternative adaptive methods with Polyak-stepsize~\cite{polyak1987introduction}, requires a knowledge of a value of the objective's minimum.

\textbf{Smoothness.} Smoothness relaxations recently draw significant attention recently. One of such assumptions is \textit{$(L_0,L_1)$-smoothness} originally introduced by \cite{zhang2020why} for twice differentiable functions, and extended to the class of differentiable but not necessarily twice differentiable functions \cite{zhang2020improved,chen2023generalized}. We introduce the assumption that the Hessian eigenvalue in the gradient direction grow linearly with the \textit{square root of the optimality gap} instead \textit{gradient norm}. Another closely related work \cite{mishkin2024directional}  also study directional smoothness, but based on a quadratic model, in contrast to ours cubic one. Adaptive methods \cite{defazioLearningRateFreeLearningDAdaptation2023,mishchenko2024prodigy,ivgi2023dog,Khaled2023DoWGUA} avoid smoothness assumption in their analysis in the convex setting as well as any line search.

\textbf{HVP Methods.} To the best of our knowledge Hessian vector products are used only for Hessian approximation \cite{jahani2022doubly,yao2021adahessian,yao2018hessian, sen2023foplahd} i.e. based on \cite{bekas2007estimator}. Except the closest work \cite{smee2025first}, which studies similar idea for non-convex deterministic setup. Also, a similar step size appears in \cite{thomsen2024complexity}, which studies quadratic functions.

\subsection{Our Contribution}
\begin{itemize}[leftmargin=*]
    \item \textbf{Certificate for Cubic-Model Decrease.} We introduces an Hessian–vector product (HVP\footnote{The cost of computing Hessian–vector products is of the same order as that of gradients \cite{agarwal2016second}.})-based certificate to decide when the gradient step ($x^{k+1} = x^k - {\nabla f(x^k)}/{\sqrt{H_k\|\nabla f(x^k)\|}}$, where $H_k$ is a local estimate of the Hessian–Lipschitz constant) decrease (corresponds coincides aligns) with the cubic-model i.e. \algname{CRN} step decrease. Cubic-model decrease holds under \textit{vanishing directional curvature} (see \eqref{eq:accclass}) and yields $\cO(k^{-2})$ global rate for i.e logistic loss function.
    % \todo{logistic loss example}
    % maybe plus below
    % \item \textbf{global k-2 rate vaniching curvature} text
    
    \item \textbf{\underline{Ca}sual \underline{Cu}bic \underline{N}ewton (\algname{CaCuN}).} The certificate enables provable obviating the second order work (Hessian computation and factorization) for \algname{CRN}~\cite{nesterov2006cubic} type methods retaining global $\cO(k^{-2})$ rate.
    
    \item \textbf{\underline{Acc}elerated \underline{Ca}sual \underline{Cu}bic \underline{N}ewton (\algname{AccCaCuN}).} Same argument applies for accelerated \algname{CRN} retaining \algname{AccCaCuN}'s global $\cO(k^{-3})$ rate. The method  obviates the second order work under\textit{ vanishing directional curvature}.
    
    % \item \textbf{\underline{Ca}sual \underline{Cu}bic \underline{AdaN} (\algname{CaCuAdaN} (\algname{CaCuAdaN}).} text~\cite{mishchenko2023regularized}

    \item \textbf{\underline{Ca}sual \underline{Cu}bic \underline{Ad}aptive \underline{G}radient \underline{D}escent (\algname{CaCuAdGD}).} The central contribution of our work addresses the infeasibility of directly applying second-order methods. Our technique exploits Hessian smoothness through backtracking on \textit{directional curvature} (see \eqref{eq:relaxed}), which requires only Hessian–vector products. When the cubic-model decrease does not apply, the method falls back to the standard quadratic-model decrease. The method is monotone, parameter-free and shows superior wall-clock performance in experiments compared to, e.g., \algname{AdGD}~\cite{malitsky20adaptive}.

    \item \textbf{Gradient-Directional Smoothness Relaxation.} We propose an assumption allowing the Hessian eigenvalue in the gradient direction to grow linearly with the square root of the optimality gap. It enables tighter convergence rates and, in particular, a global cubic-model decrease rate, e.g., for logistic regression under separability.
    
    \item \textbf{\underline{Ca}sual \underline{Cu}bic \underline{S}tochastic \underline{G}radient \underline{D}escent (\algname{CaCuSGD}).} Under standard assumptions, having the mini-batch gradient and the batched Hessian–gradient estimator, our analysis recovers either the stochastic cubic–Newton decrease~\cite{chayti2023unified} or the standard stochastic decrease under $L$-smoothness.
    
\end{itemize}

\section{Notation and Main Assumptions}\label{sec:nanda}
Our theory is based on the following assumption about second-order smoothness, which is also the key tool in proving the convergence of cubic Newton~\cite{nesterov2006cubic}.
\begin{assumption}\label{as:hessian_smooth}
    There exists a constant $H>0$ such that for any $x, y\in \R^d$
    \begin{align}
        &f(y)
        \le f(x) + \inp{\nabla f(x)}{y-x} + \frac{1}{2} \inp{\nabla^2 f(x)(y-x)}{y-x} + \frac{H}{3}\|y-x\|^3, \label{eq:taylor}
        \\
        &\|\nabla f(y) - \nabla f(x) - \nabla^2 f(x)(y-x)\|\le H\|x-y\|^2. \label{eq:grad_dif_hess}
    \end{align}
    Both inequalities hold if the Hessian of $f$ is $2H$-Lipschitz, i.e. $\|\nabla^2 f(x)-\nabla^2 f(y)\|\le 2H\|x-y\|$ for all $\x,y\in\R^d$.
\end{assumption}
The proof is given in~\cite[Lemma~1]{nesterov2006cubic}.

\begin{assumption}[Convexity]\label{as:convexity}
    Function $f: \R^d \to \R$ is convex with, i.e. $\forall x,y \in \R^d$ it holds
    \begin{equation}
        f(y) \geq f(x) + \langle \nabla f(x), y - x \rangle.\label{eq:convexity}
    \end{equation}
\end{assumption}

\begin{assumption}[Bounded level set]\label{as:bounded-level-set}
Let $x^0\in\R^d$ and $f:\R^d\to\R$ be continuous. The sublevel set
\[
    \mathcal{L}_0:=\{x\in\R^d:\ f(x)\le f(x^0)\}
\]
is bounded, i.e.,
\[
    D:=\sup_{x\in\mathcal{L}_0} \|x-x^\star\|<\infty,
    \qquad x^\star \in \arg\min f.
\]
\end{assumption}

\section{Descent Analysis}\label{sec:core}
Since $f(x)$ has $2H$-Lipschitz--Hessian, there exists $H_k \leq H$ s.t.
 \begin{equation*}\textstyle
    f(x^+) \leq 
    f(x^k)
    +\inp{\nabla f(x^k)} {x^+ - x^k} 
    + \frac{1}{2} \inp{x^+ - x^k}{\nabla^2 f(x^k) (x^+ - x^k)} 
    +  \frac{H_k}{3} \norm{x^+ - x^k}^3.
\end{equation*}

By setting $x^+ = x^k - \frac{\nabla f(x^k)}{\sqrt{M_k\norm{\nabla f(x^k)}}}$ for $M_k \ge H_k$ and $\alpha \in (0, 1)$ we then obtain
% \begin{equation}\label{eq:bt}
%     f\rbr*{x^k - \frac{\nabla f(x^k)}{\sqrt{H_k\norm{\nabla f(x^k)}}}} \leq f(x^k) + \frac{1}{2} \frac{\inp{\nabla f(x^k)}{\nabla^2 f(x^k) \nabla f(x^k)}}{H_k \norm{\nabla f(x^k)}} -  \frac2{3\sqrt{H_k}} \norm{\nabla f(x^k)}^{3/2}.
% \end{equation}

 \begin{multline}\textstyle
    \label{eq:per-condition}
    f\rbr*{x^k - \frac{\nabla f(x^k)}{\sqrt{M_k\norm{\nabla f(x^k)}}}} 
    % \\
    \leq 
    % f(x^k) + \frac{1}{2} \sbr*{ \frac{\inp{\nabla f(x^k)}{\nabla^2 f(x^k) \nabla f(x^k)}}{M_k \norm{\nabla f(x^k)}} 
    % % -  \frac{4\alpha\norm{\nabla f(x^k)}^{3/2}}{3\sqrt{M_k}}}
    % % -  \frac{2(1-\alpha) \norm{\nabla f(x^k)}^{3/2}}{3\sqrt{M_k}}
    % % \\
    % =
    f(x^k) -  \frac{2(1-\alpha)}{3\sqrt{M_k}} \norm{\nabla f(x^k)}^{3/2}
    \\
    + \frac{1}{2M_k \norm{\nabla f(x^k)}} \sbr*{\inp{\nabla f(x^k)}{\nabla^2 f(x^k) \nabla f(x^k)} -  \frac{4\alpha\sqrt{M_k}}3 \norm{\nabla f(x^k)}^{5/2}}.
    % \\
    % =
    % f(x^k) + \frac{\norm{\nabla f(x^k)}}{2M_k} \sbr*{\frac{1}{\norm{\nabla f(x^k)}^2}\inp{\nabla f(x^k)}{\nabla^2 f(x^k) \nabla f(x^k)} -  \sqrt{M_k} \norm{\nabla f(x^k)}^{1/2}}-  \frac{1}{6\sqrt{M_k}} \norm{\nabla f(x^k)}^{3/2}.
\end{multline}

% The latter means one can chose $M_k \leq H$ s.t.
% \begin{equation}\label{eq:condition}
%     \inp*{\frac{\nabla f(x^k)}{\norm{\nabla f(x^k)}}}{\nabla^2 f(x^k) \frac{\nabla f(x^k)}{\norm{\nabla f(x^k)}}} \leq \sqrt{M_k\norm{\nabla f(x^k)}}
% \end{equation}
The latter means that there exists
\begin{equation}\label{eq:certificate}\textstyle
     M_k = \max{\rbr*{H_k, \frac{9\inp*{\nabla f(x^k)}{\nabla^2 f(x^k) \nabla f(x^k)}^2 }{16\alpha^2\norm{\nabla f(x^k)}^5}}} 
\end{equation}
s.t.
\begin{equation}\label{eq:condition}\textstyle
    f\rbr*{x^k - \frac{\nabla f(x^k)}{\sqrt{M_k\norm{\nabla f(x^k)}}}}  \leq f(x^k) -  \frac{2(1-\alpha)}{3\sqrt{M_k}} \norm{\nabla f(x^k)}^{3/2}.
\end{equation}

Inequality~\eqref{eq:certificate} employs the Hessian–gradient product can be computed at (asymptotically) gradient cost, without forming the Hessian. Together with inequality~\eqref{eq:condition}, one can estimate local, directional Hessian--Lipschitz constant, which we do via first-order backtracking.

\emph{Case I (cubic-model decrease):} If $M_k=H_k$, then the above gives
\begin{equation} \label{eq:cubic-model decrease}\textstyle
    f(x^{k+1}) \le  f(x^k)  -  \frac{2(1-\alpha)}{3\sqrt{H_k}} \|\nabla f(x^k)\|^{3/2},
\end{equation}
which yields the usual cubic-model decrease and the $\mathcal{O}(k^{-2})$ regime under \cref{as:convexity,as:bounded-level-set}.

\emph{Case II (quadratic-model decrease):} If $M_k = \frac{9\inp*{\nabla f(x^k)}{\nabla^2 f(x^k) \nabla f(x^k)}^2}{16\alpha^2\norm{\nabla f(x^k)}^5}$ then
\begin{equation}\label{eq:quadratic-model decrease}\textstyle
    f\rbr*{x^k - \frac{\nabla f(x^k)}{\sqrt{M_k\norm{\nabla f(x^k)}}}}  \leq f(x^k) -  \frac{8(1-\alpha)\alpha}{9L_k} \norm{\nabla f(x^k)}^{2},
\end{equation}
where $L_k = \frac{\inp*{\nabla f(x^k)}{\nabla^2 f(x^k) \nabla f(x^k)}}{\norm{\nabla f(x^k)}^2}$, it corresponds to the classical quadratic-model decrease and yields the $\mathcal{O}(k^{-1})$ regime under \cref{as:convexity,as:bounded-level-set} and an additional smoothness assumption. Note that gradient $L$-smoothness is not an issue since $\nabla^2 f$ is is $2H$-smooth and, by monotonicity under \cref{as:bounded-level-set}, the iterates remain in a compact set $\mathcal L_0$, and $L=\sup_{x\in\mathcal L_0}\|\nabla^2 f(x)\|<\infty$. \cref{sec:fo} consider the standard $L$-smoothness and several alternatives.

\section{Regularized Newton Methods}\label{sec:cacun}
Inequality~\eqref{eq:condition} immediately certificates for skipping the second order step when $M_k \le H$. But, we propose a slightly different rule which leans more towards performing \algname{CRN} step, preserving the constants of the original results in~\cite{nesterov2006cubic,nesterov2008accelerating}.

\subsection{Casual Cubic Newton}
\begin{algorithm}[t]% \begin{algorithm}[H]
\caption{Casual Cubic Newton (\algname{CaCuN})}
\label{alg:cacun}
\begin{algorithmic}[1]
    \State \textbf{Input:} $x^0 \in \R^d$, $H > 0$
    \For{$k = 1,2,\dots$}
        \If {$f\rbr*{x^k - \frac{2\nabla f(x^k)}
            {\sqrt{3H\norm{\nabla f(x^k)}}}} \leq f(x^k) -  \frac{1}{\sqrt{2H}} \rbr*{\frac{2}{3}}^{3/2}\norm{\nabla f(x^k)}^{3/2}$}
            \State
            $ x^{k+1} = x^k - \frac{2\nabla f(x^k)}         {\sqrt{3H\norm{\nabla f(x^k)}}}$
        \Else {}
            \State
            $x^{k+1}=x^k - (\nabla^2 f(x^k)+H \|x^{k+1}-x^k\| \mI)^{-1}\nabla f(x^k)$, 
            % \State  
            % where $\lambda_k \gets 
                % \left\{
                %   \begin{array}{ll}
                %     \sqrt{H\|\nabla f(x^k)\|} & \text{ for AdaN},
                %     \\
                %     H \|x^{k+1}-x^k\| & \text{ for Cubic Newton}
                %   \end{array}
                % \right.
                % $
            \Comment{CRN step}
         \EndIf
 	\EndFor
\end{algorithmic}	
\end{algorithm}

In order to align constants with the standard cubic Newton convergence one can check inequality~\eqref{eq:condition} with $\alpha=0.5$ and $M_k = \nicefrac{3}{4}H$. If it holds then $
% \begin{equation*}
    x^{k+1} = x^k - \frac{\nabla f(x^k)}{\sqrt{M_k\norm{\nabla f(x^k)}}} $
% \end{equation*}
gives the CRN decrease result~\cite{nesterov2006cubic}
\begin{equation}\label{eq:crn decrease}
    f(x^{k+1}) \leq f(x^k) -  \frac{1}{3\sqrt{M_k}} \norm{\nabla f(x^k)}^{3/2} = f(x^k) - \frac{1}{\sqrt{2H}} \rbr*{\frac{2}{3}}^{3/2}\norm{\nabla f(x^k)}^{3/2}
\end{equation}
and pure CRN step is performed otherwise (\cref{alg:cacun}).
When the first step is the gradient step then
\begin{multline*}
    f(x^{1}) \leq f(y^0) -  \frac{\rbr*{2/3}^{3/2}}{\sqrt{2H}}  \norm{\nabla f(y^0)}^{3/2} \equiv \min_y \sbr*{f(y^0) + \inp{\nabla f(y^0)}{y-y^0} + H\|y-y^0\|^3}
        \\
        \leq \min_y \sbr*{f(y) + H\|y-y^0\|^3}
\end{multline*}
aligns with the original proof for the first iterate in~\cite[$L=2H$]{nesterov2006cubic}. Thus, 
\begin{equation*}
    f(x^k) - f_\star \leq \frac{3HD^3}{\rbr{1 + k/3}^2}
\end{equation*}
follows for \cref{alg:cacun} by \cref{as:convexity,as:bounded-level-set}.

\subsection{Accelerated Casual Cubic Newton}
The same technique applies to the accelerated \algname{CRN}. The resulting \cref{alg:acccacun} mirrors the original converge guarantees.
% \begin{algorithm}[H]
\begin{algorithm}[t]
\caption{Accelerated Casual Cubic Newton (\algname{AccCaCuN})}
\label{alg:acccacun}
\begin{algorithmic}[1]
    \State \textbf{Input:} $y^0 = x^0 \in \R^d$, $H > 0$, $M_k \le H$, i.e. $M_k = H$
    \State \textbf{Init:}
    \If {$f\rbr*{y^0 - \frac{\nabla f(y^0)}
        {\sqrt{M_0\norm{\nabla f(y^0)}}}} \leq f(y^0)-  \frac{\sqrt2}{3\sqrt{M_0}} \norm{\nabla f(y^0)}^{3/2}$}
        \State $ x^{1} = y^0 - \frac{\nabla f(y^0)}         {\sqrt{M_0\norm{\nabla f(y^0)}}}$
    \Else {}
        \State
         $x^{1}=x^0 - (\nabla^2 f(x^0)+2H \|x^{1}-x^0\| \mI)^{-1}\nabla f(x^0)$
    \EndIf
    \State $\psi_{1}(x) = \frac{N}{3}\norm{x - x_0}^3 + f(x^{1})$
    \For{$k = 1,2,\dots$}
        \State $v^k = \arg\min_{x \in \R^d} \psi_k(x)$
        \State $y^k = \frac{k}{k+3} x^k + \frac{3}{k+3} v^k$
        \If {$f\rbr*{y^k - \frac{\nabla f(y^k)}
            {\sqrt{M_k\norm{\nabla f(y^k)}}}} \leq f(y^k)-  \frac{1}{\sqrt{3M_k}} \norm{\nabla f(y^k)}^{3/2}$} 
            \State $ x^{k+1} = y^k - \frac{\nabla f(y^k)}         {\sqrt{M_k\norm{\nabla f(y^k)}}}$ 
            \State $\psi_{k+1}(x) = \psi_k(x) + \frac{(k+1)(k+2)}{2}  \left[ f(y^{k}) + \langle \nabla f(y^{k}), x - y^{k} \rangle \right]$
        \Else {}
            \State
            $x^{k+1}=x^k - (\nabla^2 f(x^k)+2H \|x^{k+1}-x^k\| \mI)^{-1}\nabla f(x^k)$ \Comment{CRN step}
            \State $\psi_{k+1}(x) = \psi_k(x) + \frac{(k+1)(k+2)}{2} \left[ f(x^{k+1}) + \langle \nabla f(x^{k+1}), x - x^{k+1} \rangle \right]$
         \EndIf
 	\EndFor
\end{algorithmic}	
\end{algorithm}

\begin{theorem}
    Let~\cref{as:convexity,as:bounded-level-set,as:hessian_smooth} hold. Then the iterates generated by~\cref{alg:acccacun} satisfy the following
    \[
        f(x^k) - f(x^*) \leq \frac{14 \cdot 2H \|x^0 - x^*\|^3}{k(k + 1)(k + 2)}.
    \]
    % where $x^*$ is an optimal solution to the problem.
\end{theorem}

\section{Casual Cubic AdaN} \label{sec:regnewton}
Same technique is also applicable for regularized global Newton~\cite{mishchenko2023regularized}, which we refer as \algname{RegNewton}. We simply perform
\begin{equation}\label{eq regnewton step}
    x^{k+1}=x^k-\frac{\nabla f(x^k)}{\sqrt{H_k\|\nabla f(x^k)\|}}
\end{equation}
if
\begin{equation}\label{eq:regnewton-descent}
    f\left(x^k-\frac{\nabla f(x^k)}{\sqrt{H_k\|\nabla f(x^k)\|}}\right)
    \le f(x^k)-\frac{2}{3\cdot64\sqrt{H_k}}\|\nabla f(x^k)\|^{3/2},
\end{equation}
where, for the non-adaptive \algname{RegNewton}, $H_k \equiv H$ is fixed; $H$ is a known global upper bound on the Hessian–Lipschitz constant.

% where $H_k = H$ for the non-adaptive method, with known upper bound $H$.

Instead of assuming a known global $H$, we adapt a local Hessian--Lipschitz estimate $H_k$ from first-order information via \eqref{eq:per-condition} and accept the gradient step~\eqref{eq regnewton step} whenever the \algname{AdaN}~\cite{mishchenko2023regularized} descent test~\eqref{eq:regnewton-descent} holds, which is the key inequality for \algname{AdaN} convergence. The resulting scheme does not always improve wall-clock time. But we observe consistent gains with a simple two-stage scheme: run the gradient step governed by \eqref{eq:regnewton-descent} (backtracking on $H_k$) until the test first fails, then switch to \algname{AdaN}. This also “warms up’’ $H_k$ using only first-order backtracking. The resulting algorithm is listed as \cref{alg:cacuadan}.

\begin{algorithm}[H]
\caption{Casual Cubic \algname{AdaN} (\algname{CaCuAdaN)}}
\label{alg:cacuadan}
\begin{algorithmic}[1]
    \State \textbf{Input:} $x^0 \in \R^d$, $ H>0$, $H>0$, $flag = 0$
    \For{$k = 0,1,\dots$}
        \If{ \textbf{not } \textit{ flag }}
            \State $ H= H/2$
                \While{$f\rbr*{x^k - \frac{\nabla f(x^k)}{\sqrt{H\norm{\nabla f(x^k)}}}} \geq f(x^k) + \frac{1}{2} \frac{\inp{\nabla f(x^k)}{\nabla^2 f(x^k) \nabla f(x^k)}}{H \norm{\nabla f(x^k)}} -  \frac2{3\sqrt{H}} \norm{\nabla f(x^k)}^{3/2}$}
                
                \State $ H = 2  H$
                \EndWhile
                \If {$f\rbr*{x^k - \frac{\nabla f(x^k)}{\sqrt{H\norm{\nabla f(x^k)}}}} \geq f(x^k) -  \frac2{3\cdot64\sqrt{H}} \norm{\nabla f(x^k)}^{3/2}$}
                    \State \textit{flag = 1}
                    \State \textbf{break}
                \EndIf
                \State $x^{k+1} = x^k - \frac{\nabla f(x^k)}{\sqrt{H\norm{\nabla f(x^k)}}}$
        \Else
                \State $x^{k+1}$ is calculated by \algname{AdaN} step.
        \EndIf
 	\EndFor
\end{algorithmic}	
\end{algorithm}

Analogously to \algname{AdaN+}~\cite{mishchenko2023regularized}, we propose a heuristic extension that uses gradient-based backtracking to calibrate the local Hessian–Lipschitz constant. The resulting algorithm is listed as \cref{alg:cacuadanp}.
\begin{algorithm}[H]
\caption{Casual Cubic AdaN+ (\algname{CaCuAdaN+)}}
\label{alg:cacuadanp}
\begin{algorithmic}[1]
    \State \textbf{Input:} $x^0 \in \R^d$, $H>0$, 
    \For{$k = 0,1,\dots$}
        % \State  $\widehat H = \frac{9\inp*{\nabla f(x^k)}{\nabla^2 f(x^k) \nabla f(x^k)}^2 }{16\alpha^2\norm{\nabla f(x^k)}^5}$
        \State $H=H/8$
        \While{$f\rbr*{x^k - \frac{\nabla f(x^k)}{\sqrt{H\norm{\nabla f(x^k)}}}} \geq f(x^k) + \frac{1}{2} \frac{\inp{\nabla f(x^k)}{\nabla^2 f(x^k) \nabla f(x^k)}}{H \norm{\nabla f(x^k)}} -  \frac2{3\sqrt{H}} \norm{\nabla f(x^k)}^{3/2}$}
            \State
            $H = 2 H$
        \EndWhile
        \State $x^{k+1}=x^k - (\nabla^2 f(x^k)+\sqrt{H\|\nabla f(x^k)\|} \mI)^{-1}\nabla f(x^k)$, 
        % \If {$H \ge \widehat H$}
 		     % \State
        %      $ x^{k+1} = x^k - \frac{\nabla f(x^k)}         {\sqrt{ H\norm{\nabla f(x^k)}}}$
        % \Else {}
        %     \State{add backtracking}
        %     \State
        %     $x^{k+1}=x^k - (\nabla^2 f(x^k)+\sqrt{H\|\nabla f(x^k)\|} \mI)^{-1}\nabla f(x^k)$, 
        % \EndIf
        % \Comment{Accept point}
 	\EndFor
\end{algorithmic}	
\end{algorithm}

\section{First-Order Method}\label{sec:fo}
In this section we consider first-order update \(x^{k+1}=x^k-\nabla f(x^k)/\sqrt{M_k\|\nabla f(x^k)\|}\). We additionally make assumptions on gradient smoothness guaranteeing decrease when cubic regime is not feasible.

\subsection{Vanishing Directional Curvature}
Lets take a closer look at inequality \eqref{eq:per-condition}
%  \begin{multline*}
%     f\rbr*{x^k - \frac{\nabla f(x^k)}{\sqrt{M_k\norm{\nabla f(x^k)}}}} 
%     \\
%     \leq 
%     f(x^k) + \frac{1}{2M_k \norm{\nabla f(x^k)}} \sbr*{\inp{\nabla f(x^k)}{\nabla^2 f(x^k) \nabla f(x^k)} -  \sqrt{M_k} \norm{\nabla f(x^k)}^{5/2}}-  \frac{1}{6\sqrt{M_k}} \norm{\nabla f(x^k)}^{3/2},
% \end{multline*}
with $\alpha = 3/4$ for simplicity.
If the Hessian eigenvalue in the gradient direction growth is bounded as
\begin{equation}\label{eq:accclass}\textstyle
    \frac{\inp{\nabla f(x^k)}{\nabla^2 f(x^k) \nabla f(x^k)}}{\norm{\nabla f(x^k)}^2} \leq  \sqrt{C\norm{\nabla f(x^k)}}
\end{equation}
with constant $C$, then setting $M_k = M = \max{\rbr*{H, C}}$ leads to
\begin{equation*}\textstyle
    f\rbr*{x^k - \frac{\nabla f(x^k)}{\sqrt{M\norm{\nabla f(x^k)}}}} 
    \\
    \leq 
    f(x^k) -  \frac{1}{6\sqrt{M}} \norm{\nabla f(x^k)}^{3/2},
\end{equation*}
and $\cO\rbr*{\frac{\rbr*{\sqrt{MD^3} + \sqrt{f_0 - f_\star}}^2}{k^{2}}}$ convergence rate under \cref{as:convexity,as:bounded-level-set} by \cite[Lemma A.1]{nesterov2019inexact}.

While the class in \eqref{eq:accclass} is uncommon in modern practice, it is nonempty and includes several simple families with explicit constants and Lipschitz Hessians. We provide examples below.

\begin{example}\label{example:iso-cubic}
Let $f(x)=\frac{H}{3}\|x\|^3$ with $H>0$. Then $\nabla^2 f$ is $2H$-Lipschitz and \eqref{eq:accclass} holds with $C=4H$.
\end{example}

\begin{example}\label{example:sep-cubic}
Let $f(x)=\sum_{i=1}^d \frac{H_i|x_i|^3}{3}$ with $H_i>0$ and $H=\max_iH_i$.
Then $\nabla^2 f$ is $2H$-Lipschitz and \eqref{eq:accclass} holds with $C=4H$.
\end{example}

\begin{example}[Logistic function]\label{example:logistic}
Let $a\in\mathbb{R}^d\setminus\{0\}$ and $f(x)=\log\big(1+e^{-a^\top x}\big)$. Then $\nabla^2 f$ is $\frac{\|a\|^3}{6\sqrt{3}}$-Lipschitz and \eqref{eq:accclass} holds with $C=\frac{4}{27}\|a\|^3$.
\end{example}

\subsection{Casual Cubic Adaptive Gradient Descent (CaCuAdGD)}

In addition to the uncommon class in \eqref{eq:accclass}, we consider the standard smoothness setting and a relaxed variant. Using global (worst-case) constants typically yields overly conservative steps and poor practical performance. To avoid this, we introduce \algname{CaCuAdGD} (\cref{alg:cacuadgd}), a monotone, parameter-free, adaptive scheme.

\textbf{Idea.} At each iteration, \algname{CaCuAdGD} estimates the local Hessian smoothness \emph{along the gradient direction} via a backtracking procedure that relies only on Hessian–vector products (no explicit Hessian computations). This yields an \emph{automatic switch} between:
(i) a \emph{bounded-Hessian} regime, where the quadratic model driven decrease holds; and (ii) a \emph{Hessian-Lipschitz} (cubic) regime, where a cubic-model controls the remainder. 
\begin{remark}
    Since $\nabla^2 f$ is $2H$-smooth and, by monotonicity with \cref{as:bounded-level-set}, the iterates stay in the compact sublevel set $\mathcal L_0$, we have
    $L \coloneqq \sup_{x\in\mathcal L_0}\|\nabla^2 f(x)\| < \infty$,
    so $\nabla f$ is $L$-Lipschitz on $\mathcal L_0$.
\end{remark}

% \noindent\textbf{Basic properties.}
% \begin{remark}
%     The method is monotone, i.e., $f(x^{k+1})\le f(x^k)$, and under~\cref{as:bounded-level-set} its iterates remain in a compact set. On such a compact set it suffices to assume the \emph{existence} of continues third derivatives; no Hessian-Lipschitzness is required globally.
%     Similarly, since $\nabla^2 f$ is $2H$-Lipschitz, $\nabla f$ is $L$-smooth on the compact.
%     % \todo{Similarly gradient $L$-smoothness is not an issue since the Hessian exsts.}
% \end{remark}

\begin{algorithm}[H]
\caption{Casual Cubic Adaptive Gradient Descent (\algname{CaCuAdGD})}
\label{alg:cacuadgd}
\begin{algorithmic}[1]
    \State \textbf{Input:} $x^0 \in \R^d$, $H>0$, $\alpha \in (0, 1)$
    \For{$k = 0,1,\dots$}
        \State  $\widehat H = \frac{9\inp*{\nabla f(x^k)}{\nabla^2 f(x^k) \nabla f(x^k)}^2 }{16\alpha^2\norm{\nabla f(x^k)}^5}$
        \State $H = H / 16$ 
        \While{
        $f\rbr*{x^k - \frac{\nabla f(x^k)}{\sqrt{H\norm{\nabla f(x^k)}}}} \geq f(x^k) + \frac{1}{2} \frac{\inp{\nabla f(x^k)}{\nabla^2 f(x^k) \nabla f(x^k)}}{H \norm{\nabla f(x^k)}} -  \frac2{3\sqrt{H}} \norm{\nabla f(x^k)}^{3/2}$ \textbf{and}  $\widehat H < H$}
            \State
            $H = 2 H$
        \EndWhile
        \State
             $ x^{k+1} = x^k - \frac{\nabla f(x^k)}         {\sqrt{\max\cbr*{H,  \widehat{H}}\norm{\nabla f(x^k)}}}$
        
 	\EndFor
\end{algorithmic}	
\end{algorithm}

\subsubsection{Lipschitz gradient}
\begin{assumption}[$L$-smoothness]\label{as:Lsmooth}
    Function $f: \R^d \to \R$ is $L$-smooth, i.e., for all $x, y \in \R^d$ we have
    \begin{equation}
        \|\nabla^2 f(x)\| \leq L. \label{eq:Lsmooth}
    \end{equation}
\end{assumption}

Setting $\alpha=0.5$ in \eqref{eq:certificate} yields
% \begin{equation*}
$
     M_k = \max{\rbr*{H_k, \frac{9\inp*{\nabla f(x^k)}{\nabla^2 f(x^k) \nabla f(x^k)}^2 }{4\norm{\nabla f(x^k)}^5}}} 
     $
% \end{equation*}
and either
\begin{equation*}
    f(x^{k+1}) \leq f(x^k) -  \frac{1}{3\sqrt{H_k}} \norm{\nabla f(x^k)}^{3/2} \quad \text{ or } \quad f(x^{k+1}) \leq f(x^k) -  \frac{2}{9L_k} \norm{\nabla f(x^k)}^{2},
\end{equation*}
% \todo{2H due to backtracking}
% where $H_k \leq H$ or 
% \begin{equation*}%\label{eq:condition}
%     f(x^{k+1}) \leq f(x^k) -  \frac{2}{9L_k} \norm{\nabla f(x^k)}^{2},
% \end{equation*}
where $L_k = \frac{\inp*{\nabla f(x^k)}{\nabla^2 f(x^k) \nabla f(x^k)}}{\norm{\nabla f(x^k)}^2} \leq L$.

Since the optimality gap $f(x^k) - f_\star$ is monotonically decreases, iterates split into 2 sets.

\begin{lemma}\label{lem:two-regime}
    Let~\cref{as:convexity,as:bounded-level-set,as:Lsmooth,as:hessian_smooth} hold. Then the iterates generated by~\cref{alg:cacuadgd} with $\alpha=0.5$ satisfy the following
    
    \begin{eqnarray}
        &\text{ (i) }& \bigl(f(x^k)-f_\star\bigr)^{1/2}\ge \frac{3}{2}\frac{LD^2}{\sqrt{HD^3}}, \quad \text{then}\quad  f(x^{k+1}) - f_\star  \le  f(x^k) - f_\star - \frac{1}{3\sqrt{HD^3}}\bigl(f(x^k)-f_\star\bigr)^{3/2} \label{eq first stage}
        \\
        &\text{(ii) }& \bigl(f(x^k)-f_\star\bigr)^{1/2}\le \frac{3}{2}\frac{LD^2}{\sqrt{HD^3}}, \quad \text{then}\quad  f(x^{k+1}) - f_\star  \le  f(x^k) - f_\star - \frac{2}{9LD^2}\bigl(f(x^k)-f_\star\bigr)^{2}.  \label{eq second stage}
    \end{eqnarray}
    
    % Consequently, for all $k$,
    % \[
    %     f(x^{k+1}) - f_\star
    %      \le 
    %     f(x^k) - f_\star - \min\Bigl\{\tfrac{1}{3\sqrt{HD^3}}\bigl(f(x^k)-f_\star\bigr)^{3/2}, \tfrac{2}{9LD^2}\bigl(f(x^k)-f_\star\bigr)^{2}\Bigr\}.
    % \]
    % Moreover, letting $K_1 \coloneqq \min\left\{k\ge 0:\bigl(f(x^k)-f_\star\bigr)^{1/2}\le \frac{3}{2}\frac{LD^2}{\sqrt{HD^3}}\right\}$, inequality (a) holds for all $k<K_1$ and (b) holds for all $k\ge K_1$.
\end{lemma}

Inequality~\eqref{eq first stage} immediately implies
\begin{corollary}
     Let~\cref{as:convexity,as:bounded-level-set,as:Lsmooth,as:hessian_smooth} hold. Then the iterates generated by~\cref{alg:cacuadgd} with $\alpha=0.5$ satisfy for all $k$ s.t. $\bigl(f(x^k)-f_\star\bigr)^{1/2} \ge \frac{3}{2}\frac{LD^2}{\sqrt{HD^3}}$ the following
    % \begin{equation*}
    %     f(x^{k+1}) - f_\star \leq \frac{81HD^3}{k^2} \rbr*{1+\rbr*{\frac{f(x^0) - f_\star}{9HD^3}}^{1/2}}^2.
    % \end{equation*}
    \begin{equation*}
        f(x^{k+1}) - f_\star \leq \frac{\rbr*{9\sqrt{HD^3}+3\sqrt{f(x^0) - f_\star}}^2}{k^2} .
    \end{equation*}
\end{corollary}

If the desired accuracy is not reached in the accelerated regime, then the global rate is given by the following
\begin{theorem}\label{thm:lsmooth global}
    Let~\cref{as:convexity,as:bounded-level-set,as:Lsmooth,as:hessian_smooth} hold. Then the iterates generated by~\cref{alg:cacuadgd} with $\alpha=0.5$ satisfy for all $k$ s.t. $\bigl(f(x^k)-f_\star\bigr)^{1/2} \le \frac{3}{2}\frac{LD^2}{\sqrt{HD^3}}$  the following
    \begin{equation*}
        f(x^k)-f_\star
         \le 
        \frac{27}{2} \frac{LD^2}{ k+LD^2 \cdot \Phi(f(x^0)-f_\star) }, 
    \end{equation*}
    where $\Phi(\xi) \coloneqq  \frac{9}{2} \frac{1}{\xi}
       + 
      \frac{6\sqrt{H}}{L\sqrt{D}} \frac{1}{\sqrt{\xi}}$.
\end{theorem}
\begin{remark}
    In practice, there is no iterates splinting due to the algorithm's adaptive nature. That means that the algorithm performs the best possible step enabling the cubic-model decrease even for later iterates.
\end{remark}

\begin{corollary}
    Let~\cref{as:convexity,as:bounded-level-set,as:Lsmooth,as:hessian_smooth} hold. Then the total number of iterations needed to reach $f(x^K)-f_\star\le\varepsilon$ satisfies
    \[
      K  =  \cO\left(\frac{L D^{2}}{\varepsilon} +
          \frac{\rbr{\sqrt{H D^{3}}+\sqrt{ f(x^0)-f_\star }}^2}{\sqrt{\varepsilon}}
      \right).
    \]
\end{corollary}

\subsection{Relaxed Smoothness}
Similarly to \emph{$(L_0,L_1)$-smoothness} originally introduced by \cite{zhang2020why} for twice differentiable functions, we allow the norm of the Hessian of the objective to increase linearly with the growth of square root of the optimality gap. We show that this assumption holds, for example, in logistic regression under separability. Importantly, it controls \emph{only} the directional growth of the Hessian.
\begin{assumption}[Directional $L_0, L_1$-smoothness]\label{as:relaxed}
    Function $f: \R^d \to \R$ is $L$-smooth, i.e., for all $x, y \in \R^d$ we have
    \begin{equation}
        \label{eq:relaxed}
        \frac{\inp{\nabla f(x^k)}{\nabla^2 f(x^k) \nabla f(x^k)}}{\|\nabla f(x)\|^2} \leq L_0 + L_1 \rbr{f(x) - f_\star}^{1/2}. 
    \end{equation}
\end{assumption}

\begin{remark}
    \cref{as:relaxed} is a relaxed version of the global condition
    \(\|\nabla^2 f(x)\|\le L_0 + L_1 \bigl(f(x)-f_\star\bigr)^{1/2}\).
\end{remark}

\begin{example}[Logistic function, $L_0=0$]\label{ex:dir-one}
Let $a\in\R^d\setminus\{0\}$, $s=\inp{a}{x}$, and $f(x)=\log(1+e^{-s})$. Then $f_\star=0$ and for all $x\in\R^d$,
\[
\frac{\inp{\nabla f(x)}{\nabla^2 f(x) \nabla f(x)}}{\norm{\nabla f(x)}^2}
 \le \frac{2}{3\sqrt{3}} \norm{a}^2 \bigl(f(x)-f_\star\bigr)^{1/2}.
\]
\end{example}

\begin{example}[Empirical logistic, $L_0=0$ under separability]\label{ex:dir-emp}

Let there exists $v\in\R^d$ and $\gamma>0$ such that $y_i\inp{a_i}{v}\ge\gamma$ for all $i$,
then for $f(x)=\frac1n\sum_{i=1}^n \log\bigl(1+e^{-y_i\inp{a_i}{x}}\bigr)$ one has $f_\star=0$ and for all $x\in\R^d$,
\[
    \frac{\inp{\nabla f(x)}{\nabla^2 f(x) \nabla f(x)}}{\norm{\nabla f(x)}^2}
     \le \frac{2}{3\sqrt{3}} \rbr*{\frac{1}{n}\sum_{i=1}^n \norm{a_i}^4}^{1/2} \bigl(f(x)-f_\star\bigr)^{1/2}.
\]
\end{example}

The proof of the following result mainly follows the arguments above, as the iterates split into cases where either vanishing curvature or standard smoothness holds.
\begin{corollary}
    Let~\cref{as:convexity,as:bounded-level-set,as:Lsmooth,as:relaxed} hold. Then the total number of iterations needed to reach $f(x^K)-f_\star\le\varepsilon$ satisfies
    \begin{equation*}
        K = \cO\left(\frac{L_0D^2}{\varepsilon} + \frac{\rbr*{L_1 D^2 +\sqrt{HD^3} + \sqrt{f_0 - f_*}}^2}{\sqrt{\varepsilon}} \right).
    \end{equation*} 
\end{corollary}

\subsection{Comparison with \algname{AdGD}}
\algname{CaCuAdGD} uses gradient steps $x^{k+1}=x^k-\lambda_k \nabla f(x^k)$ with
\[
\lambda_k
=\min\Bigl\{ \frac{16\alpha^2}{9 L_k} , \frac{1}{\sqrt{ H_k \|\nabla f(x^k)\| }} \Bigr\},
\]
while \algname{AdGD} employs
\[
\lambda_k
=\min\Bigl\{ \frac{1}{2 L_{k-1}} , \sqrt{1+\theta_{k-1}} \lambda_{k-1} \Bigr\},
\qquad
\theta_{k-1}:=\frac{\lambda_{k-1}}{\lambda_{k-2}}.
\]

\textbf{Step size growth.} When the local smoothness improves sharply (i.e., $L_k\ll L_{k-1}$), \algname{AdGD} can increase $\lambda_k$ only multiplicatively via $\sqrt{1+\theta_{k-1}}$, which may take several iterations to reach the new scale $L_k$. In contrast, \algname{CaCuAdGD} adapts in a single step $H_k\|\nabla f(x^k)\|$ vanishes near a solution, so the step is governed by $L_k$ (fast expansion to the local scale). Far from the solution, when $H_k \|\nabla f(x^k)\|$ is large, the cubic-model decrease \eqref{eq:cubic-model decrease} is guaranteed
which corresponds to a local $O(1/k^2)$ decay.

\textbf{Choice of $\alpha$.} The parameter $\alpha\in(0,1)$ control the trade-off between the two regimes. Smaller $\alpha$ tightens $\tfrac{16\alpha^2}{9L_k}$ and triggers the quadratic regime earlier (more conservative). Whereas larger $\alpha$ loosens it, making the steps larger and lets the cubic regime act more often. At the same time, the decrease in the quadratic regime scales with $\alpha(1-\alpha)$ \eqref{eq:quadratic-model decrease}, so increasing $\alpha$ beyond $1/2$ reduces this per-step guarantee. Because the time spent in each regime is problem dependent, selecting $\alpha$ a priori is nontrivial. Our analysis allows any $\alpha\in(0,1)$; in practice, values slightly below $3/4$ (i.e. 0.7) perform consistently well on various problems, making the method effectively parameter free.
Also, in the quadratic regime, the allowed step (the constant multiplying $1/L_k$) is larger for \algname{CaCuAdGD} ($\approx 0.87$), compared with \algname{AdProxGD}~\cite{malitsky2024adaptive} ($1/\sqrt{2}$) and \algname{AdGD} ($0.5$).

\section{Stochastic Extension}\label{sec:stoch}
In this section we consider 
\[
  \min_{x\in\R^d} f(x) = \E_\xi f_\xi(x)  ,
\]
particularly, a finite sum minimization $f(x) = \tfrac{1}{n}\sum_{i=1}^n f_i(x)$
where each $f_i:\R^d\to\R$ is a per-sample loss.
Here, $\xi$ denotes a (multi)set of sampled indices $S\subseteq[n]$ (a mini-batch), so $f_\xi(x)\equiv f_S(x)=\tfrac{1}{|S|}\sum_{i\in S} f_i(x)$.

\begin{assumption}\label{as:stoch}
	% For all $\x\in \R^d$ the noise in every stochastic estimator is independent. 
 Stochastic gradient $g(x, \xi)$ is an unbiased estimator of $\nabla\f(x)$ with bounded variance, i.e., $\E_{\xi} [\g(x, \xi)] = \nabla\f(x)$ and for $\sigma_g \geq 0$
    %  \begin{equation} \label{ass:xi-unbiased}
    %     \E_{\xiv_i} [\f_i(\x, \xiv_i)] = \F_i(\x),
    % \end{equation}
    % and has bounded variance, i.e., satisfies 
    \begin{equation}  \label{eq:xi-var}
    \textstyle
        \E_{\xi} \norm*{ g(x, \xi) -  \nabla\f(x)}^2 \leq \sigma_g^2,
    \end{equation}

    stochastic Hessian $H(x, \xi)$ has bounded third moment
    \begin{equation}  \label{eq:hess-xi-var}
    % \textstyle
        \E_{\xi}\norm{H(x, \xi) - \nabla^2 f(x)}^3 \le \sigma_H^3.
    \end{equation}
    Moreover, $f_\xi$ is $L$-smooth, i.e., $\forall x \in \R^d$
	\begin{equation}\label{eq:lipschitzness}\textstyle
		\norm{H(x, \xi)} \leq L.
	\end{equation}
     Also, $f$ and $f_\xi$ $2H$-smooth Hessian , i.e., $\forall x, y \in \R^d$
    \begin{equation}\label{eq:hesslipschitzness}\textstyle
		f_{(\xi)}\rbr{x} \le f_{(\xi)}\rbr{y} + \inp*{\nabla f_{(\xi)}\rbr{y}}{x - y}  + \frac12 \inp*{\nabla^2 f_{(\xi)}\rbr{y}\rbr{x - y}}{x - y}+ \frac{H}{3}\norm*{x - y}^3. 
	\end{equation}
\end{assumption}

These assumptions are standard~\cite{gorbunov2020unified,chayti2023unified}.

\begin{remark}\label{rem:hgp}
    In the algorithm we \textit{do not} form $H(x,\xi)$ explicitly; we only require stochastic Hessian--gradient products $H(x,\xi) g(x,\xi)$. Since, to the best of our knowledge, there is no standard assumption tailored exactly to this product, we use the standard third-moment bound above. However, our analysis also suffices under the following non-standard second-moment bound:
    \[
      \E_\xi \left[\|(\nabla^2 f(x)-H(x,\xi)) g(x,\xi)\|^2\right] \le \sigma_P^2.
    \]
\end{remark}

For brevity we write $g^k \coloneqq g(x^k,\xi_k)$ and $H^k \coloneqq H(x^k,\xi_k)$. The stochastic extension is given by the following
\begin{theorem}\label{thm:cacusgd}
    Let \cref{as:stoch} holds. Assume that $\widehat L \ge L$ and $\widehat H \ge H$. Then the method (\algname{CaCuSGD})
    % \begin{equation*}
    %     x^{k+1} = x^k - \frac{g^k}{\sqrt{M_k\norm{g^k}}}
    % \end{equation*}
    \[
        x^{k+1} = x^k - \frac{g^k}{\sqrt{M_k\norm{g^k}}}, \quad M_k \coloneqq
        \begin{cases}
            \widehat L^2 {\norm{g^k}}^{-1}, & 
            \text{if}\quad {4 \inp{g^k}{H^k g^k}^{2}}{\norm{g^k}^{-5}} \ge \widehat H,\\
            % \dfrac{\widehat L}{\norm*{g^k}}, & \text{otherwise.}
            \widehat H, & 
            \text{otherwise,}
            % \widehat H, & \text{if}\quad \tfrac{4\inp*{g^k}{H^k g^k}^{2}}{\norm*{g^k}^{5}} \le H,\\
        \end{cases} \quad \text{satisfies}
    \]
    % where $M_k = \widehat H$ if $\frac{4\inp*{g^k}{H^k g^k}^2 }{\norm{g^k}^5} \le H$ and $M_k = \frac{\widehat L}{\norm{g^k}}$
    % satisfy
    \begin{align*}
        &\text{ (i) } \E f(x^{k+1}) \le \E f(x^k)
        -\frac{1}{4\widehat L}\E\|\nabla f(x^k)\|^2
        +\Big(\frac{2L}{\widehat L^2}+\frac{2L^2}{3\widehat L^3}\Big) \sigma_g^2
        +\frac{\sigma_H^3}{6 \widehat H^2},  
        &\text{if}\quad {4 \inp{g^k}{H^k g^k}^{2}}{\norm{g^k}^{-5}} \ge \widehat H,
        \\
        &\text{(ii) } \E f(x^{k+1})
        \le
        \E f(x^k)
        -\frac{1}{3\sqrt{\widehat H}} \E\sbr*{\|\nabla f(x^k)\|^{3/2}}
        +\frac{7}{4\sqrt{\widehat H}}\sigma_g^{3/2}
        +\frac{32}{3 \widehat H^2} \sigma_H^{3}, 
        &\text{otherwise}.
    \end{align*}
    
\end{theorem}

Telescoping these one-step bounds is the only remaining step to obtain the standard-form result. However, because the iterates may switch regimes and the stochastic error terms mix, we state the per-iteration bounds as above.

The two regimes mirror the stochastic decrease results for standard \algname{SGD}~\cite{gorbunov2020unified} and stochastic \algname{CRN}~\cite{chayti2023unified}. The cubic phase provide irreducible but lower-power variance floor $\sigma_g^{3/2}$, while the quadratic phase reproduces the \algname{SGD} behavior: increasing $\widehat L$ (i.e., using a smaller stepsize) reduces the variance-driven term, enabling convergence to arbitrary accuracy. In practice, the cubic regime shows a faster transient toward the noise floor.

\section{Experiments}

\textbf{Logistic regression.} Our first experiment concerns the logistic regression problem with $\ell_2$ regularization:
\begin{equation}
    \min_{x\in\R^d} \frac{1}{n}\sum_{i=1}^n \left(-b_i\log(\sigma(a_i^\top x)) - (1-b_i)\log(1-\sigma(a_i^\top x)) \right) + \frac{\ell}{2}\|x\|^2,
    \label{eq:logreg}
\end{equation}
where $\sigma\colon\R\to (0, 1)$ is the sigmoid function, $\mA=(a_{ij})\in\R^{n\times d}$ is the matrix of features, and $b_i\in\{0, 1\}$ is the label of the $i$-th sample.
We use the `covtype', `w8a' and `mushrooms' datasets, and set $\ell=0$ or $\ell=10^{-7}$ to make the problem ill-conditioned. To set $H$ (if applicable), we upper bound the Lipschitz Hessian constant as $\sup_{x\in\R^d}\|\nabla^3 f(x)\|\le \frac{1}{6\sqrt{3}}\max_{i}\|a_i\| \|\mA\|^2$. This estimate is not tight, which causes non-adaptive methods to converge very slowly. 
% , where $L=\|\mA\|^2/n$ is the Lipschitz constant of the gradient. 

\textbf{Log-sum-exp.} In our second experiment, we consider a significantly more ill-conditioned problem of minimizing
\[
    \min_{x\in\R^d} \rho\log\left(\sum_{i=1}^n \exp\left(\frac{a_i^\top x - b_i}{\rho}\right)\right) , 
\]
where $a_1,\dotsc, a_n\in\mathbb{R}^d$ are some vectors and $\rho, b_1,\dotsc, b_n$ are scalars. This objectives serves as a smooth approximation of function $\max\{a_1^\top x-b_1,\dotsc, a_n^\top x - b_n\}$, with $\rho>0$ controlling the tightness of approximation. We set $n=500$, $d=200$ and randomly generate $a_1,\dotsc, a_n$ and $b_1,\dotsc, b_n$. After that, we run our experiments for several choices of $\rho$, namely $\rho\in\{0.75, 0.5, 0.25, 0.1, 0.05\}$. The values are specified in brackets on the plots.

Our implementation is available at \url{https://github.com/nazya/CaCuTe} and builds on \url{https://github.com/konstmish/opt_methods}.

\subsection{Regularized Newton Methods (non adaptive)}

Following~\cite{mishchenko2023regularized}, we compare \algname{CaCuN} (\cref{alg:cacun}) and \algname{AccCaCuN} (\cref{alg:acccacun}) with \algname{Cubic} Newton~\cite{nesterov2006cubic}, its accelerated modification \algname{AccCubic}~\cite{nesterov2008accelerating} and global regularized Newton \algname{RegNewton}~\cite{mishchenko2023regularized}
The results are presented in \cref{fig:covtypecacun,fig:w8acacun,fig:mushroomscacun,fig:logsumexpcacun}.

\begin{figure}[H]
    \centering
    \begin{minipage}[htp]{0.24\textwidth}
        \centering
        \includegraphics[width=1\linewidth]{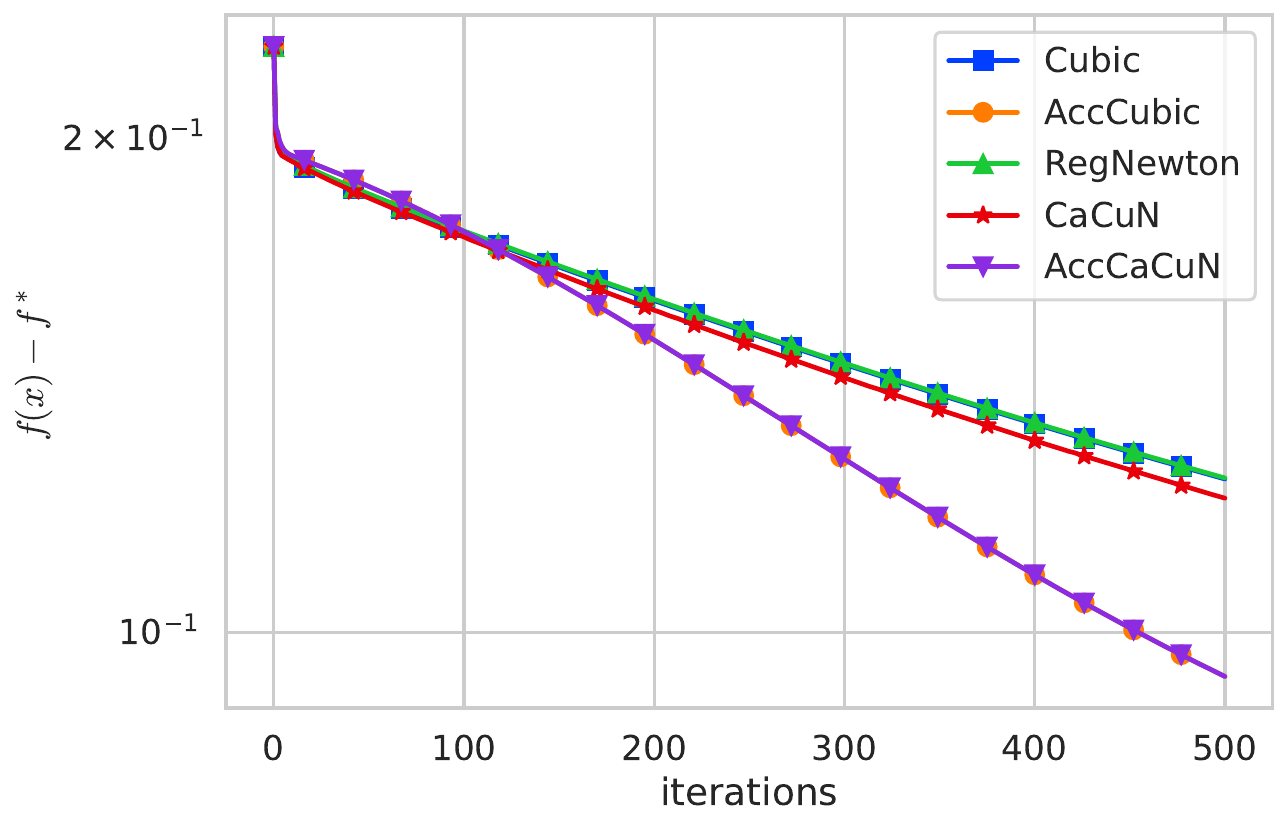}
        \label{fig:covtype-cacun-i}
        \centering
        \includegraphics[width=1\linewidth]{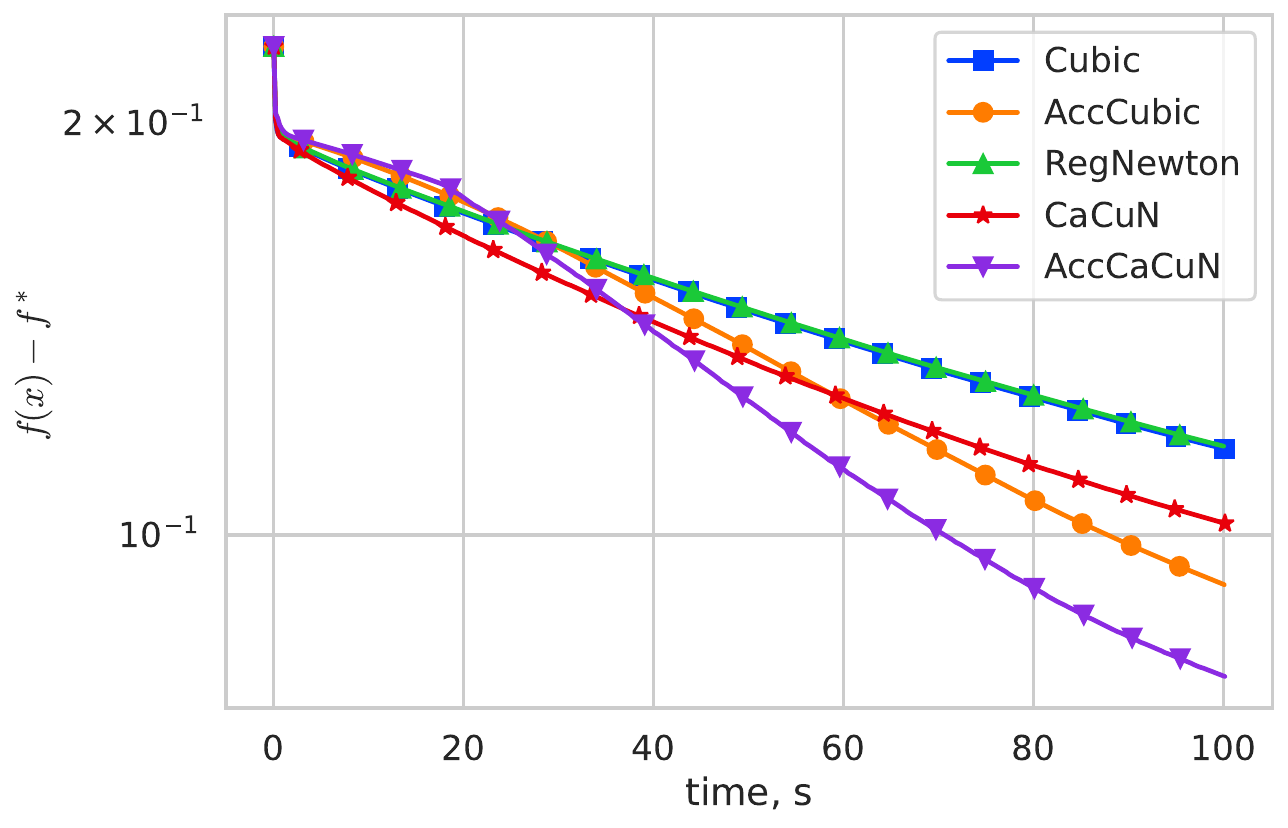}
        \caption{covtype}
        \label{fig:covtypecacun}
    \end{minipage}
    \begin{minipage}[htp]{0.228\textwidth}
        \centering
        \includegraphics[width=1\linewidth]{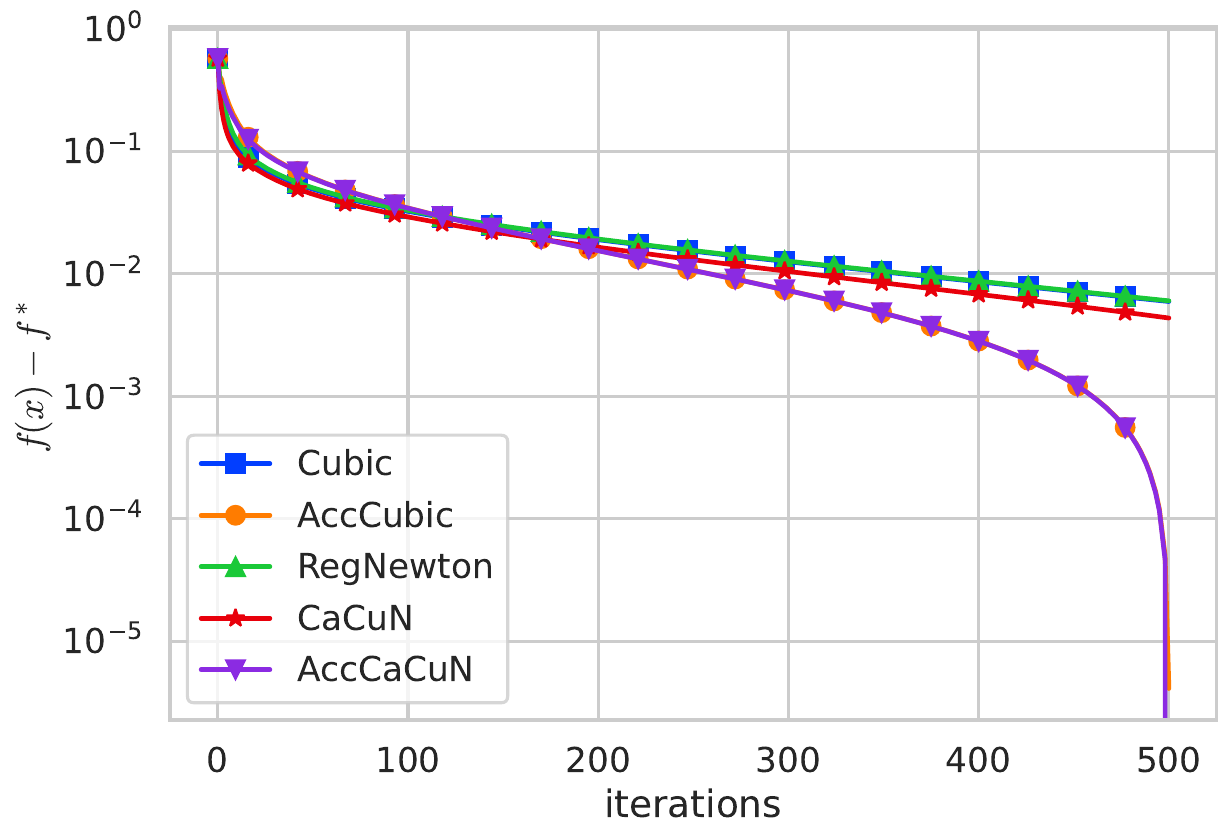}
        \centering
        \includegraphics[width=1\linewidth]{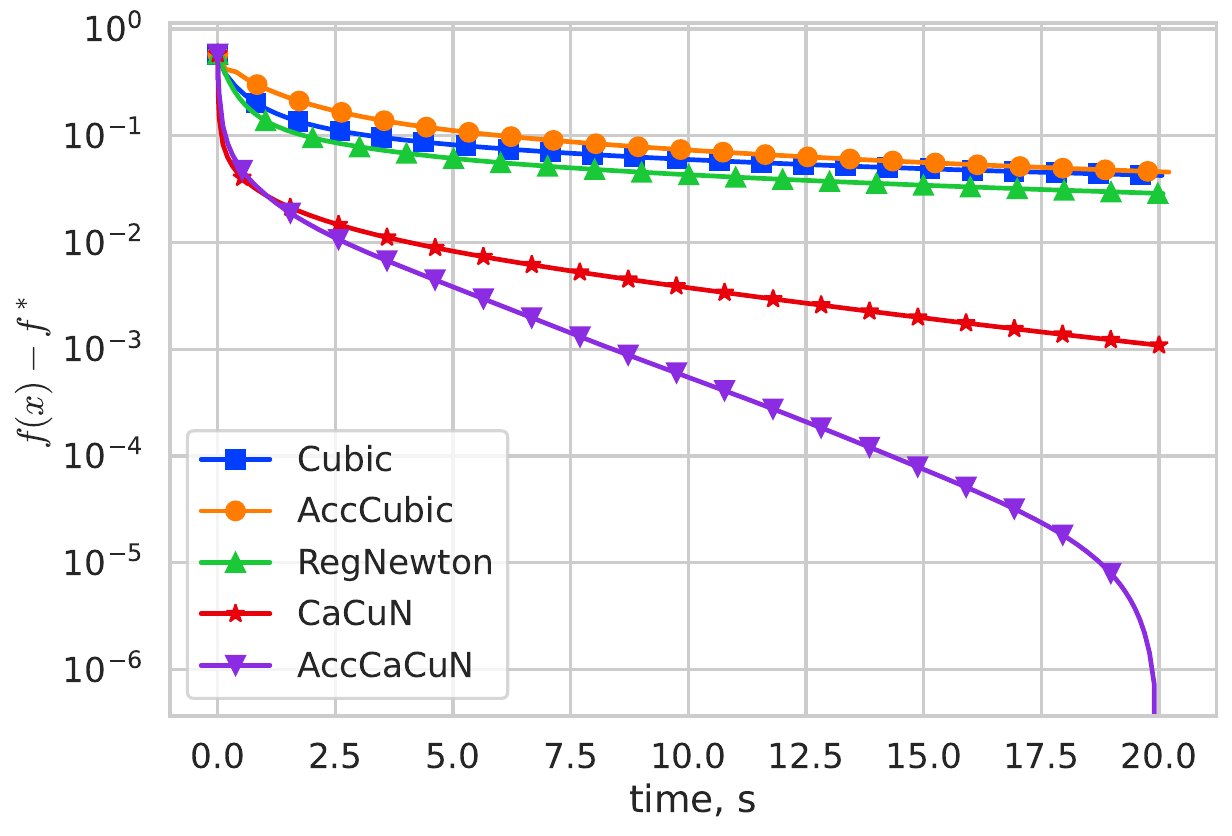}
        \caption{w8a}
        \label{fig:w8acacun}
    \end{minipage}
    \centering
    \begin{minipage}[htp]{0.23\textwidth}
        \centering
        \includegraphics[width=1\linewidth]{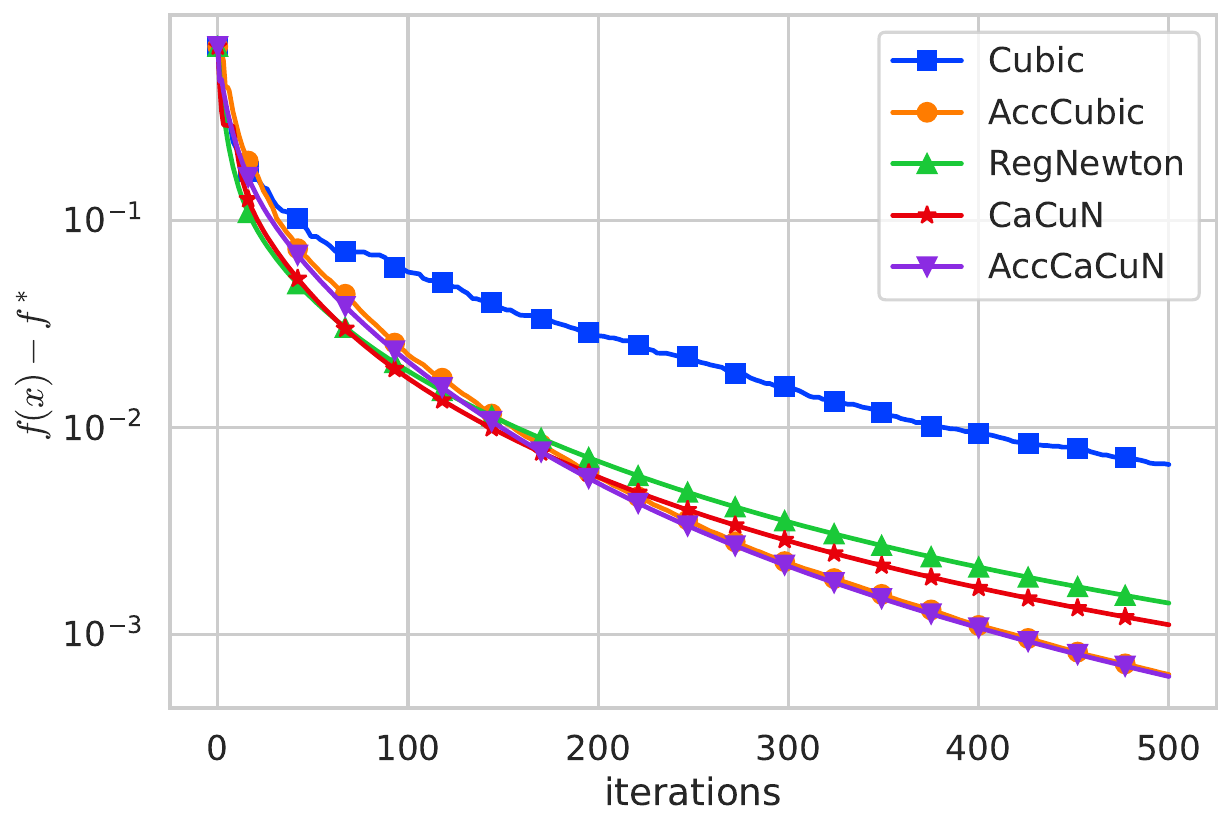}
        \centering
        \includegraphics[width=1\linewidth]{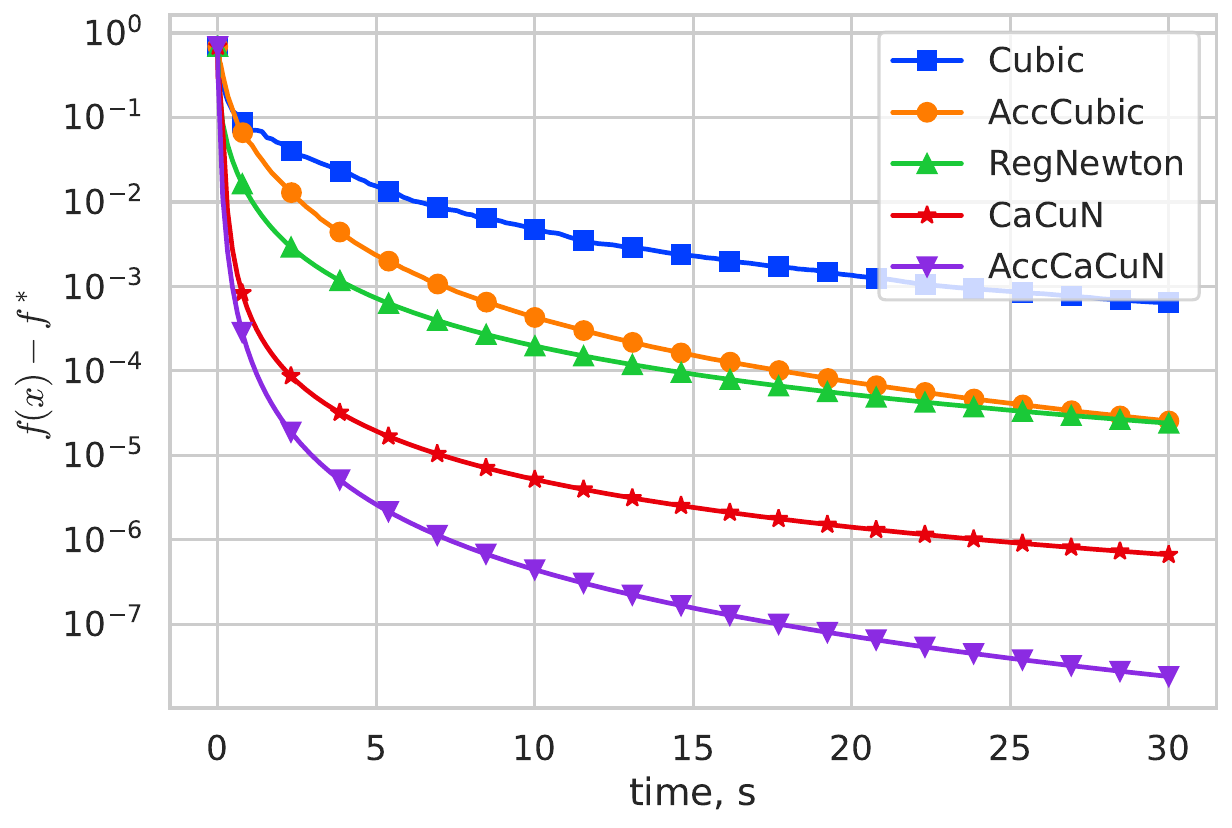}
        \caption{mushrooms}
        \label{fig:mushroomscacun}
    \end{minipage}
    \centering
    \begin{minipage}[htp]{0.228\textwidth}
        \centering
        \includegraphics[width=1\linewidth]{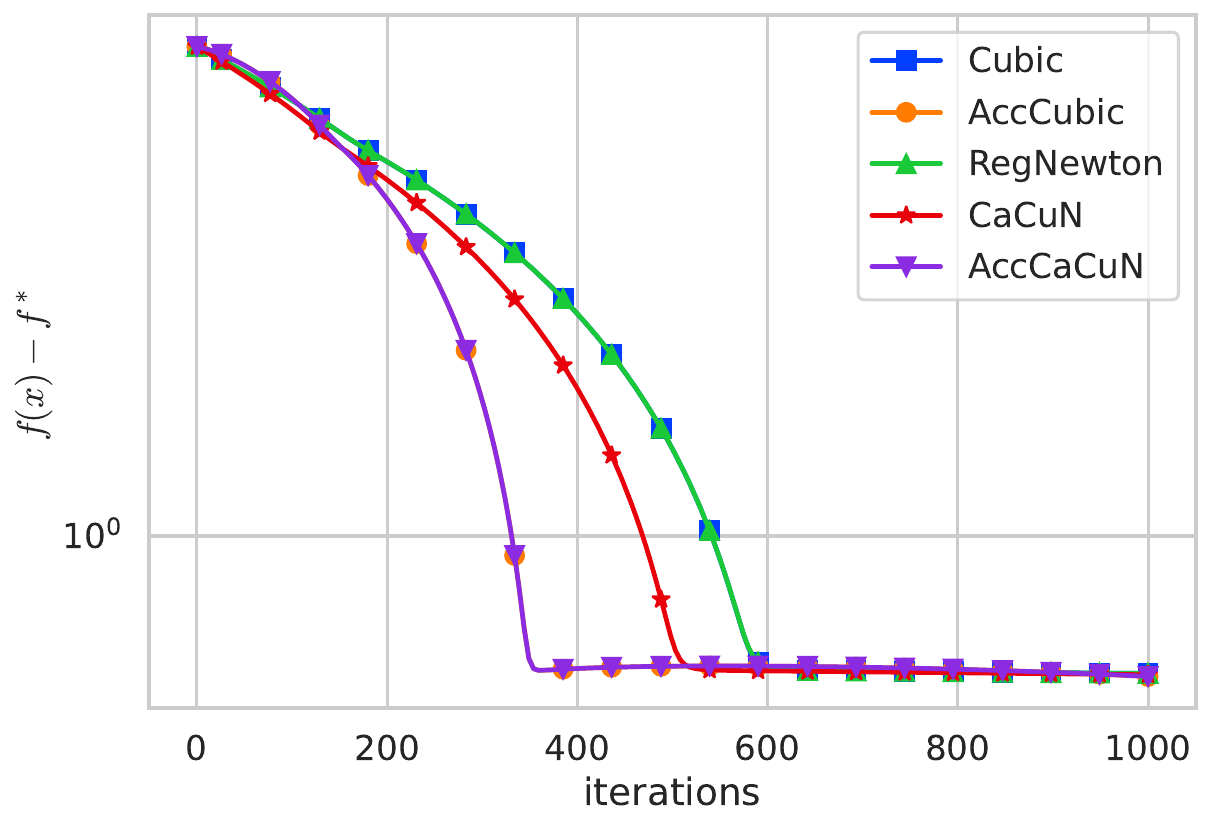}
        \centering
        \includegraphics[width=1\linewidth]{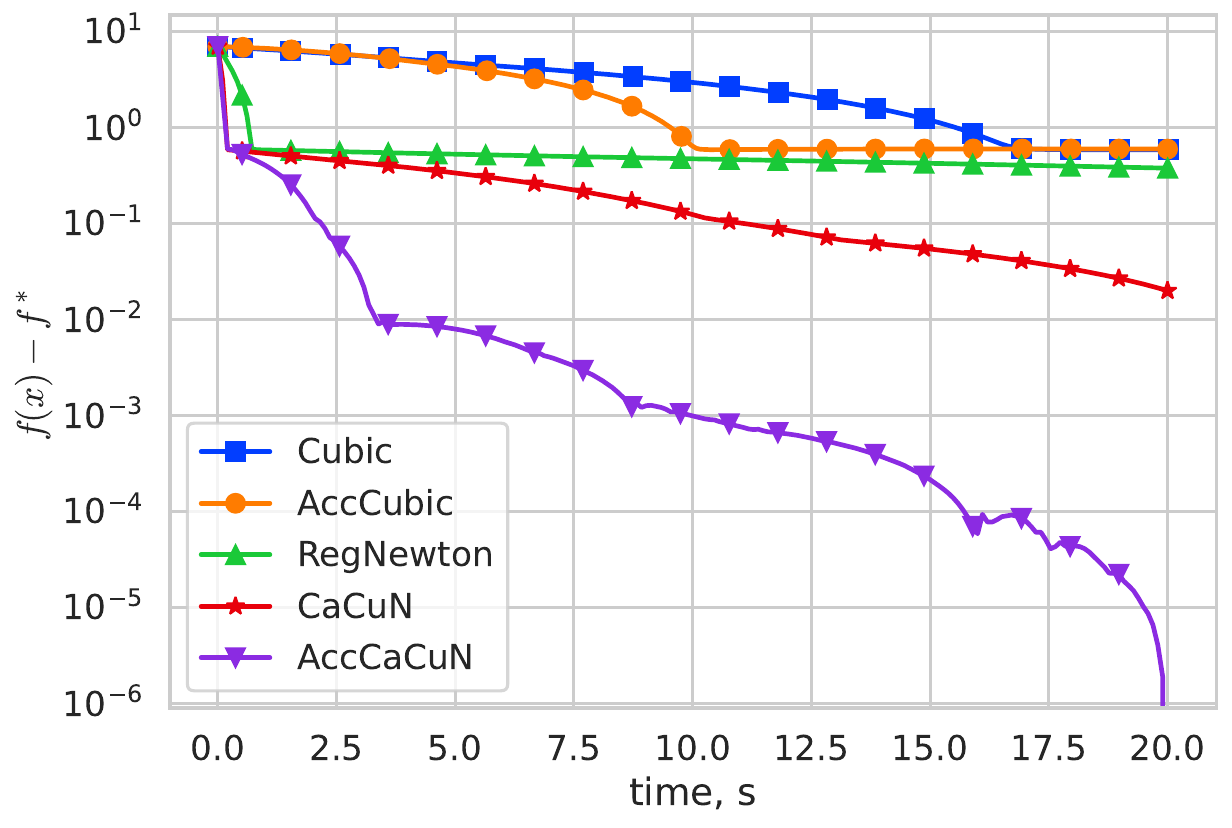}
        \caption{logsumexp (0.05)}
        \label{fig:logsumexpcacun}
    \end{minipage}
\end{figure}

The proposed methods mirror the iteration performance of their counterparts but show significantly better wall-clock performance.

\subsection{Regularized Newton Methods (adaptive)}
The adaptive schemes, in contrast, converge after a small number of iterations. Since, cubic Newton with a binary search in regularization is many times slower than \algname{AdaN}~\cite{mishchenko2023regularized} we compare our adaptive \algname{CaCuAdaN} and (\algname{CaCuAdaN+} algorithms only with \algname{AdaN} and \algname{AdaN+} from~\cite{mishchenko2023regularized}. The results are presented in \cref{fig:covtype,fig:w8a,fig:mushrooms,fig:logsumexp05}.

\begin{figure}[H]
    \centering
    \begin{minipage}[htp]{0.24\textwidth}
        \centering
        \includegraphics[width=1\linewidth]{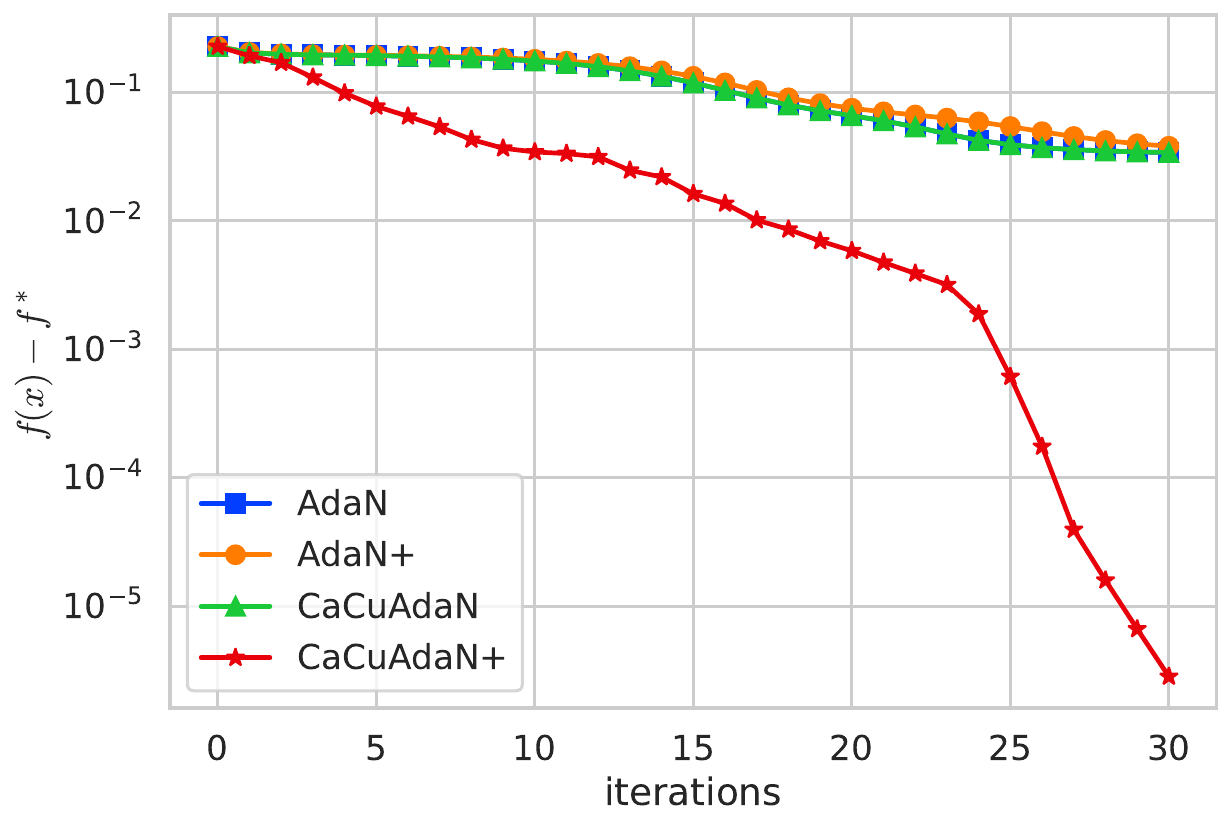}
        \includegraphics[width=1\linewidth]{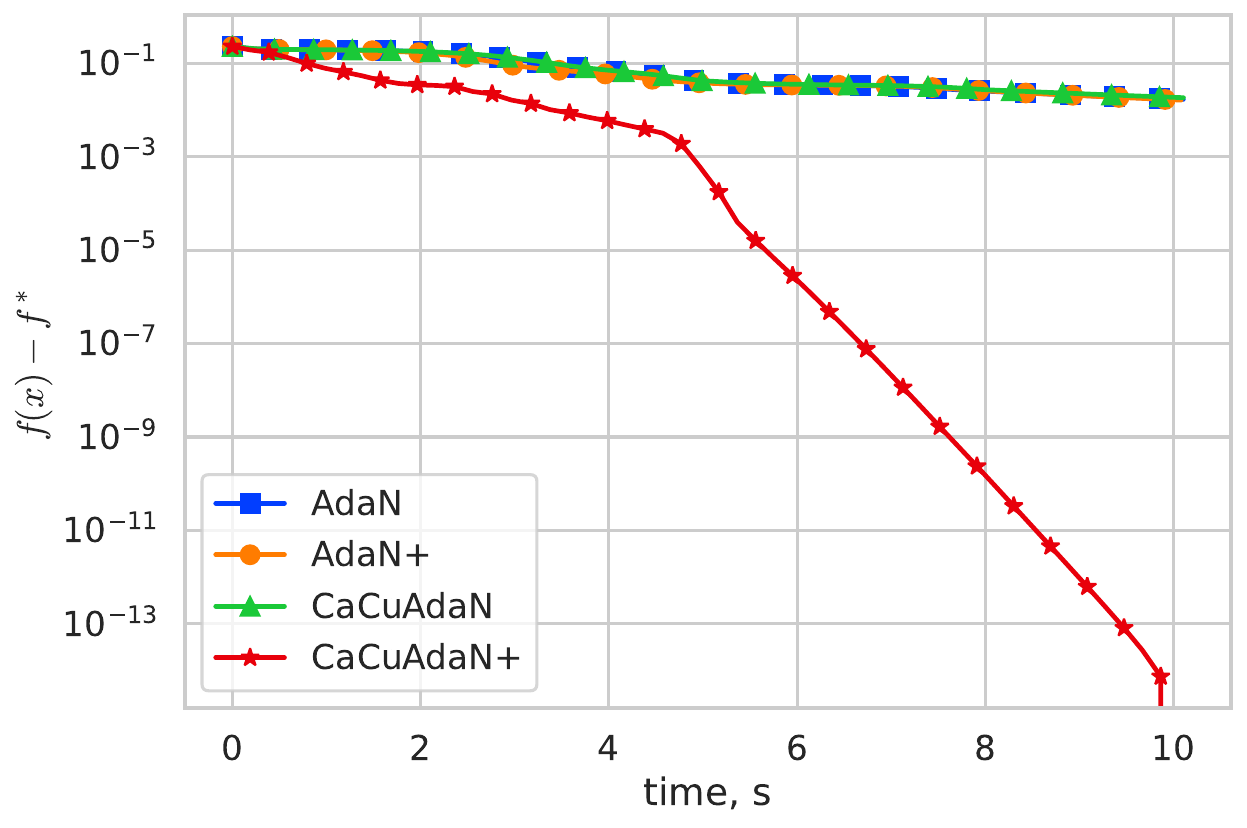}
        \caption{covtype}\label{fig:covtype}
    \end{minipage}
    \begin{minipage}[htp]{0.24\textwidth}
        \centering
        \includegraphics[width=1\linewidth]{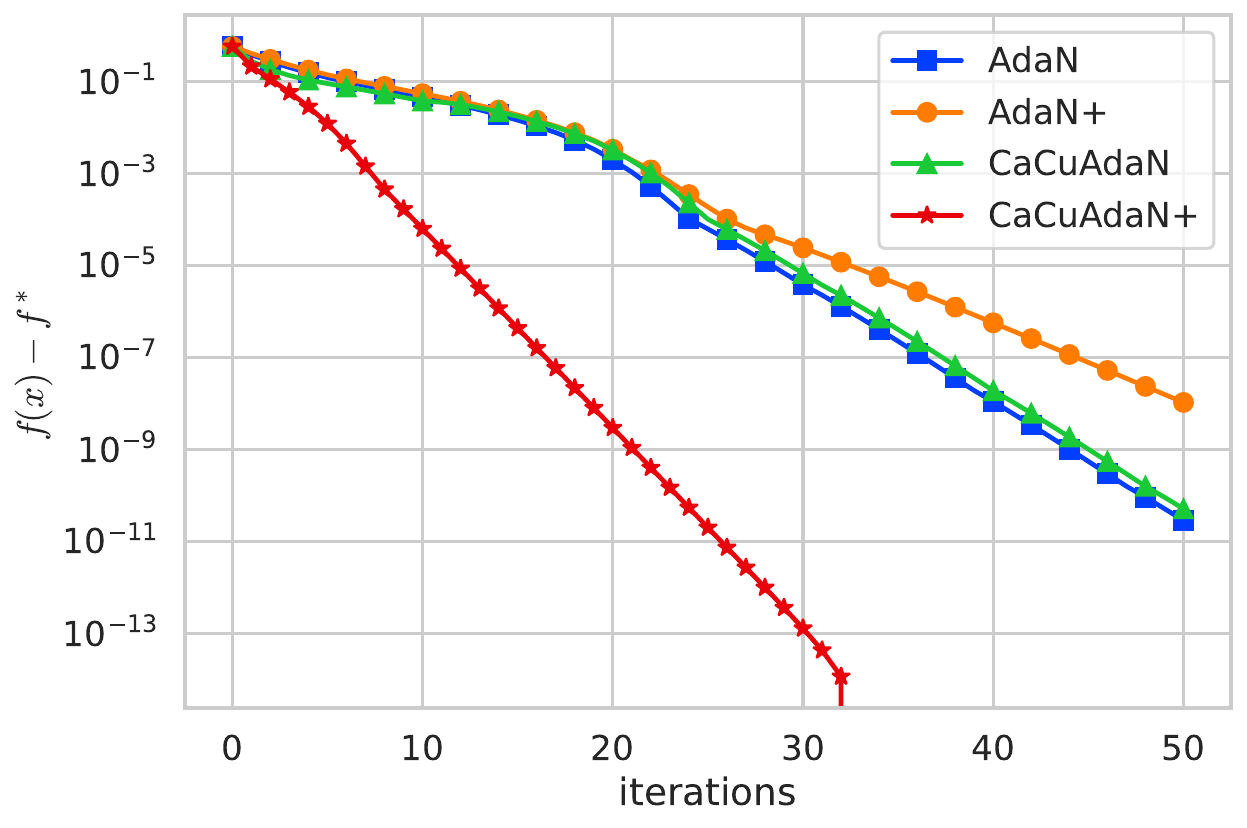}
        \includegraphics[width=1\linewidth]{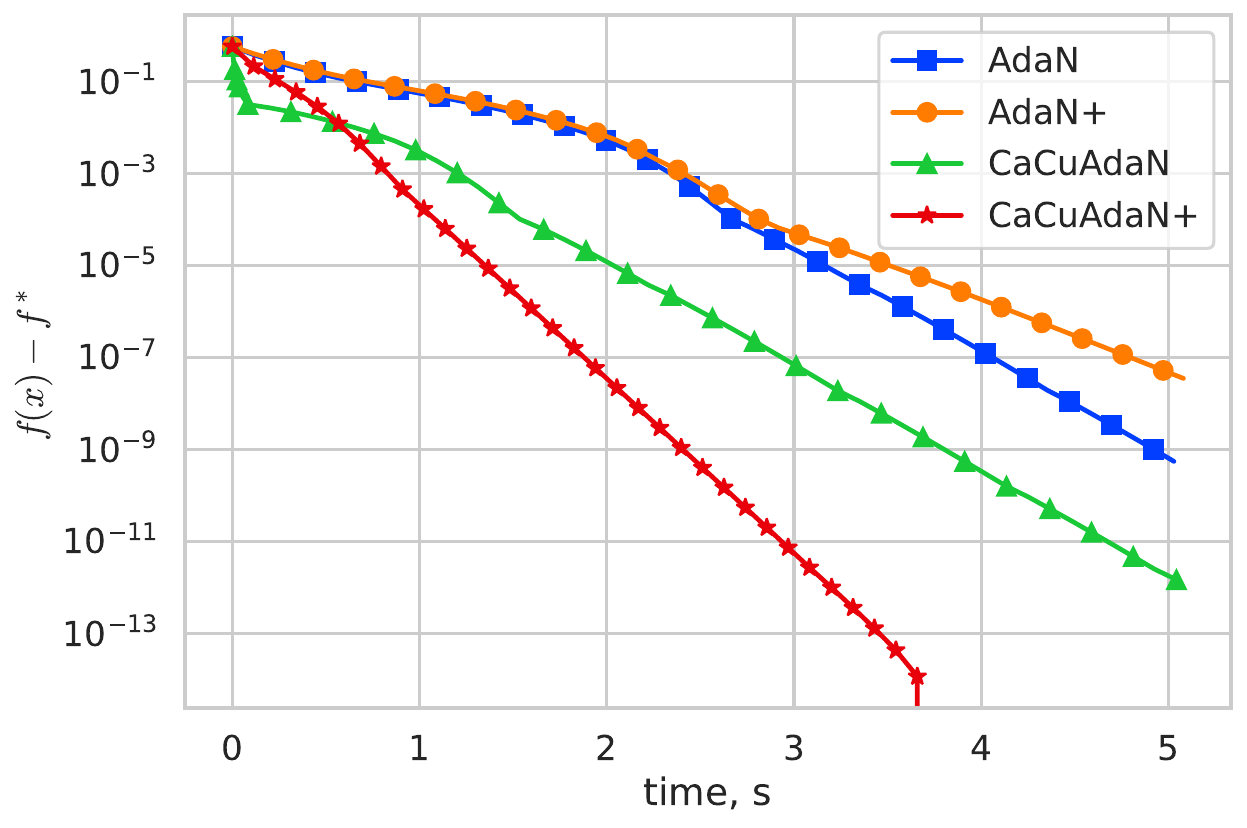}
        \caption{w8a}\label{fig:w8a}
    \end{minipage}
    \begin{minipage}[htp]{0.24\textwidth}
        \centering
        \includegraphics[width=1\linewidth]{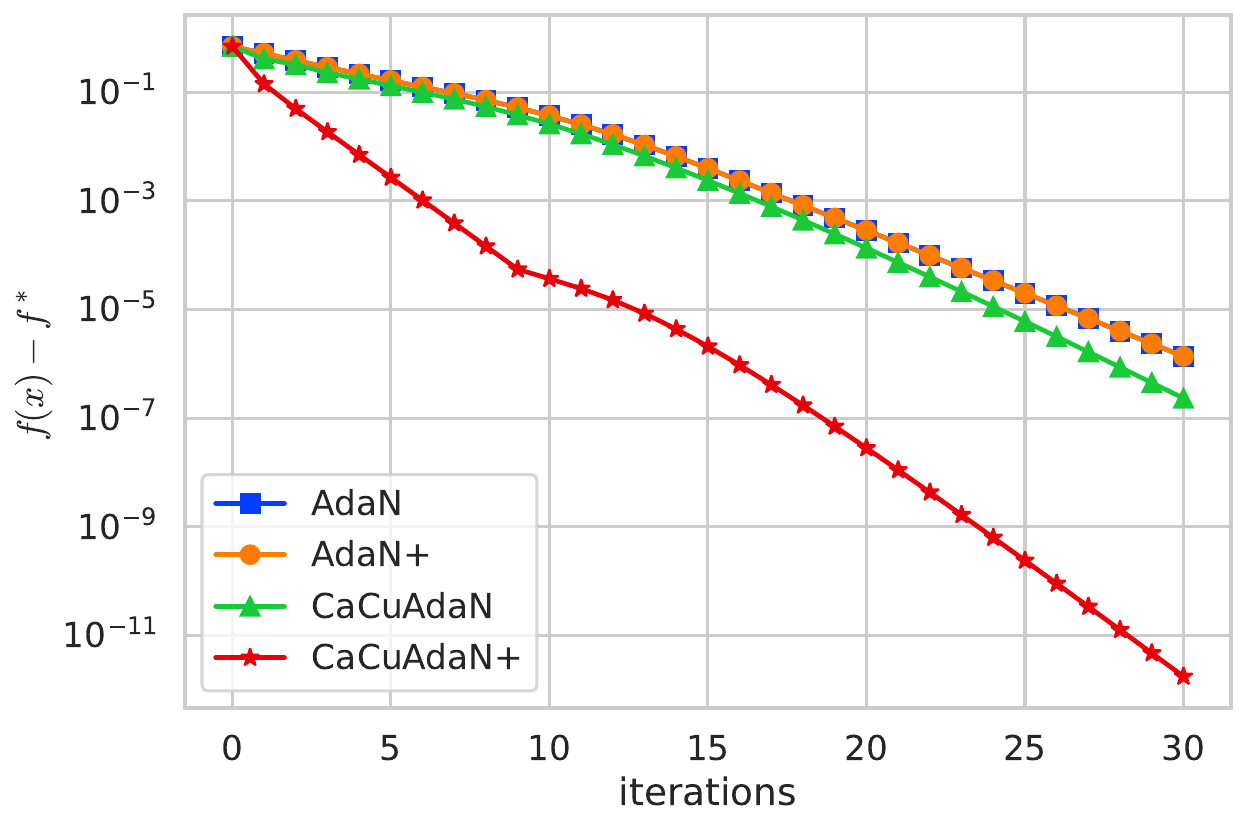}
        \includegraphics[width=1\linewidth]{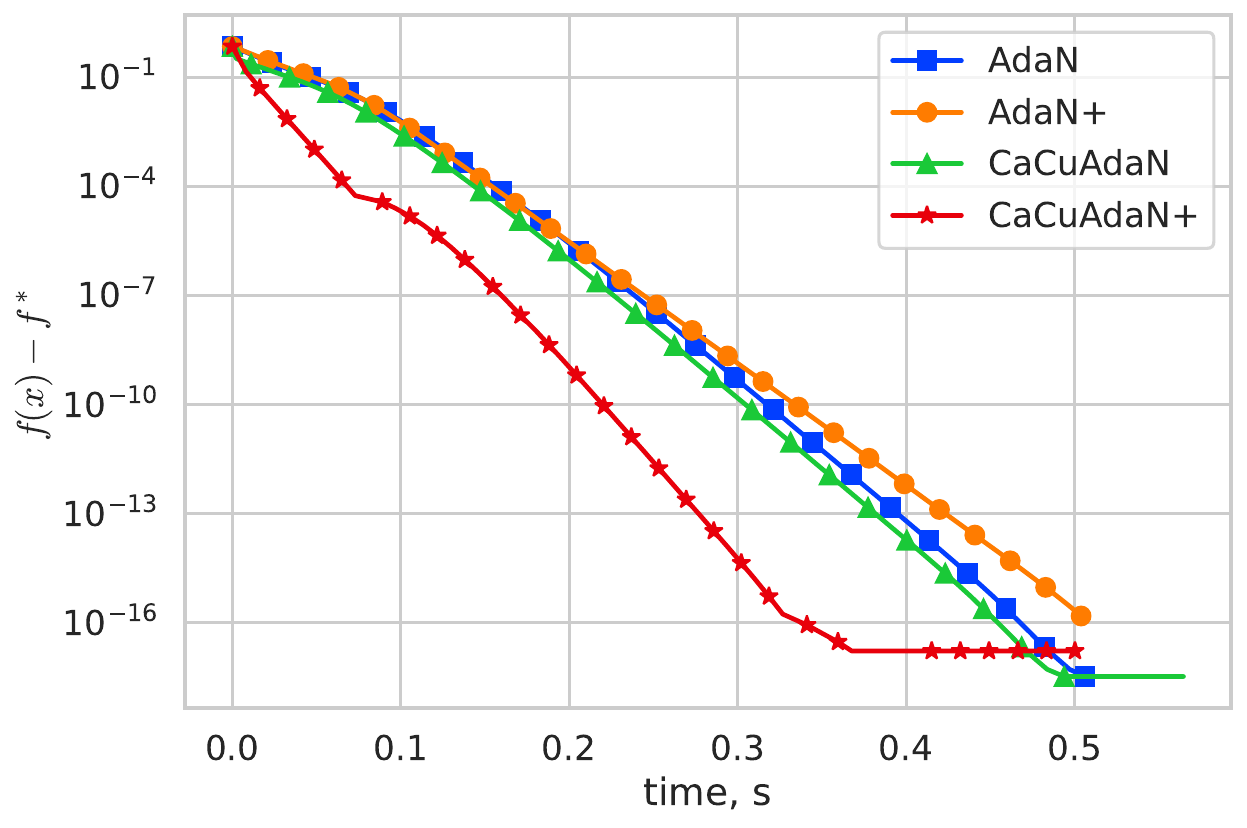}
        \caption{mushrooms}\label{fig:mushrooms}
    \end{minipage}
    \begin{minipage}[htp]{0.24\textwidth}
        \centering
        \includegraphics[width=1\linewidth]{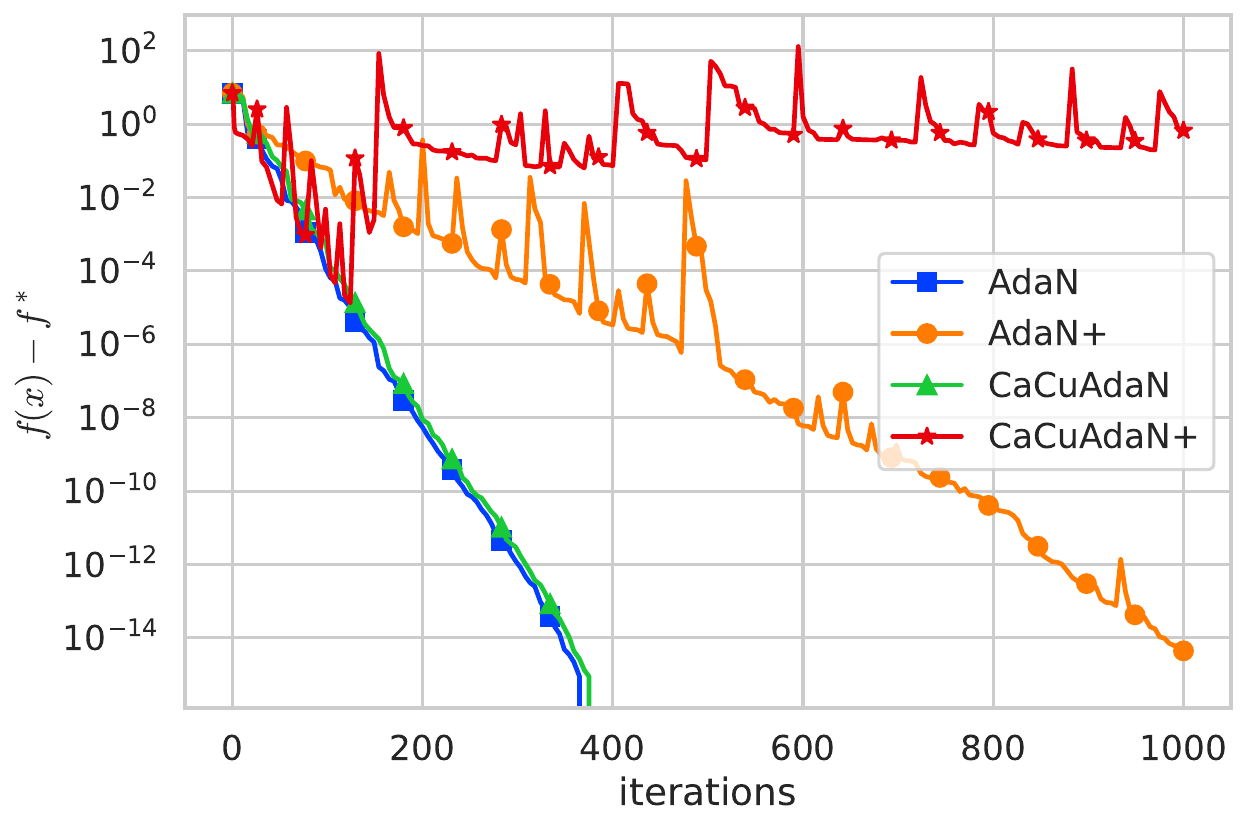}
        \includegraphics[width=1\linewidth]{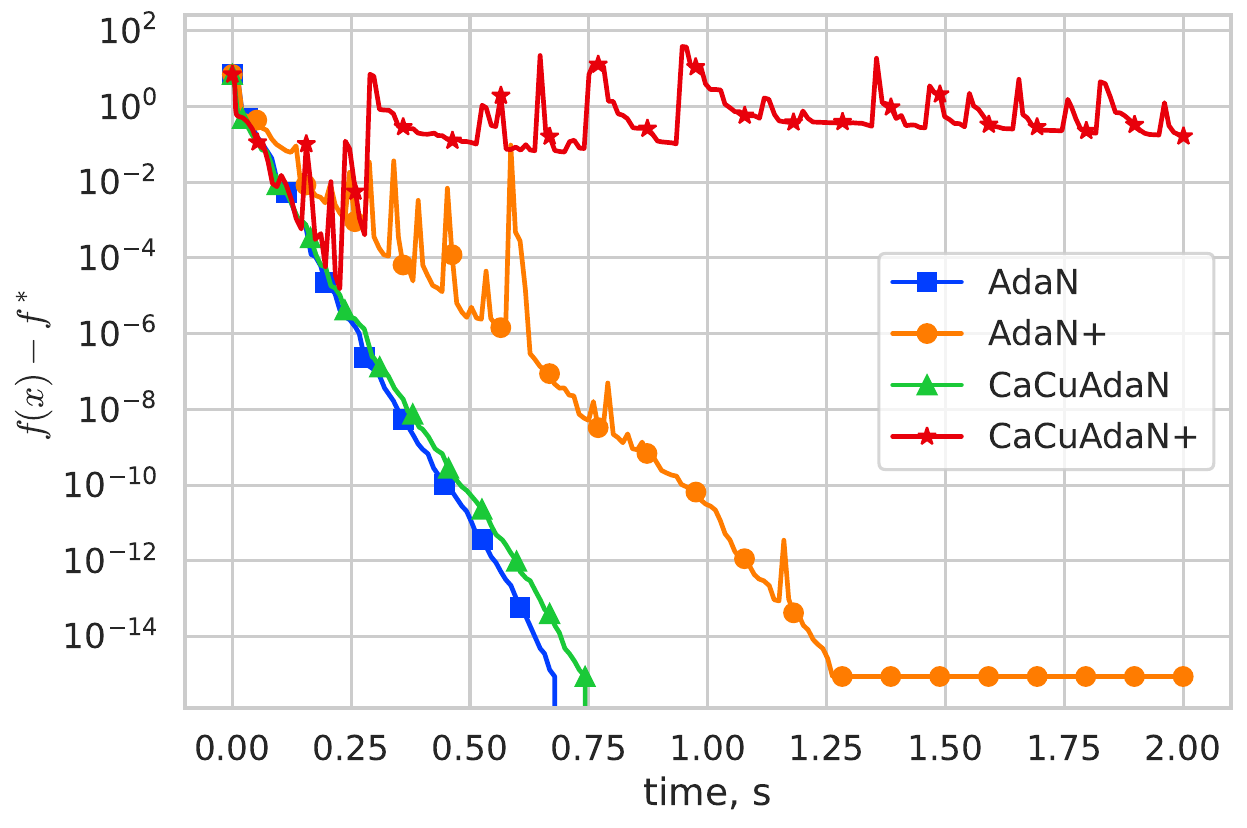}
        \caption{logsumexp (0.05)}\label{fig:logsumexp05}
    \end{minipage}
\end{figure}
Although \algname{CaCuAdaN+} significantly outperforms competing methods on logistic regression, it fails to converge on logsumexp due to too optimistic estimates of $H$.

\subsection{First-Order Methods}
Following~\cite{malitsky20adaptive}, we compare \algname{CaCuAdGD} (\cref{alg:cacuadgd}) with Nesterov's acceleration with Armijo-like line search from \cite{nesterov2013gradient};  gradient descent with Polyak stepsize~\cite{polyak1987introduction} and \algname{AdGD} (\cite{malitsky20adaptive}). The results are presented in \cref{fig:covtype-cacuadgd,fig:w8a-cacuadgd,fig:mushrooms-cacuadgd,fig:lse-005-cacuadgd,fig:lse-025-cacuadgd,fig:lse-010-cacuadgd,fig:lse-050-cacuadgd,fig:lse-075-cacuadgd}.

\begin{figure}[H]
    \centering
    \begin{minipage}[htp]{0.24\textwidth}
        \centering
        \includegraphics[width=1\linewidth]{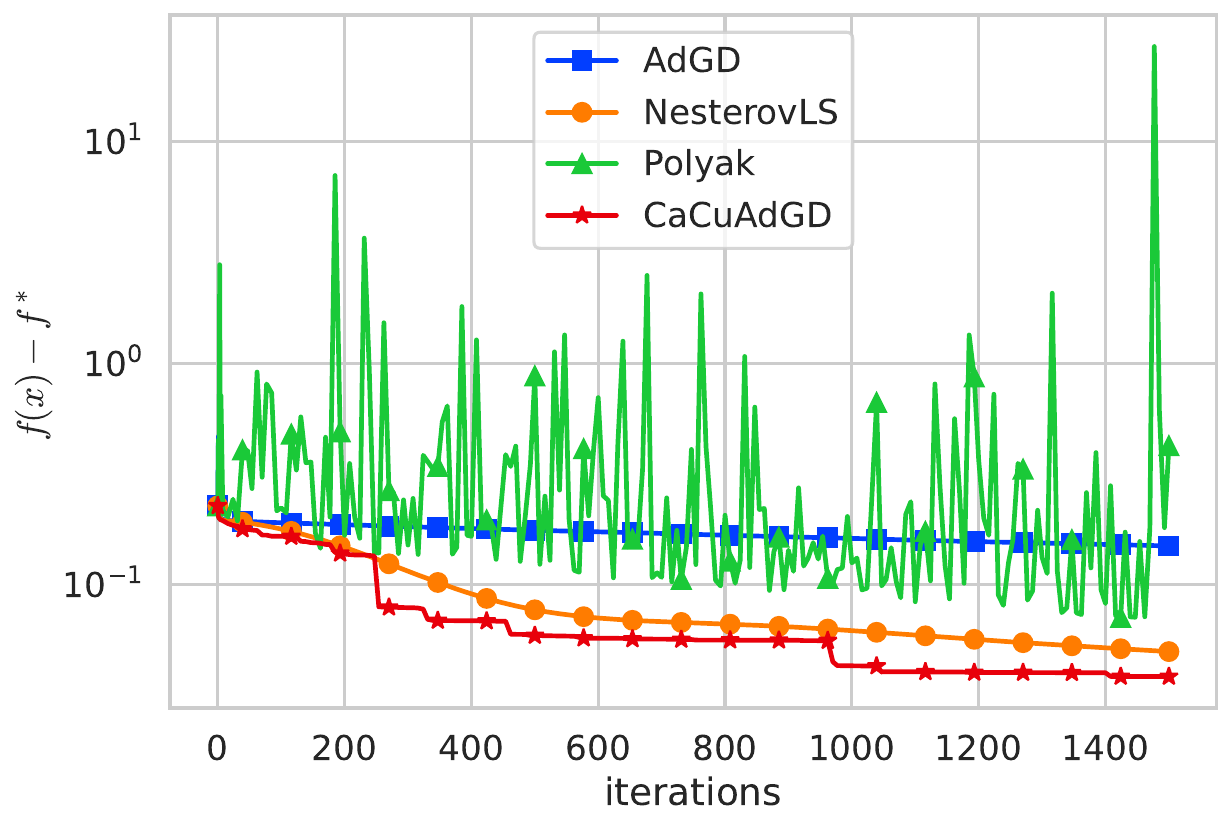}
        \includegraphics[width=1\linewidth]{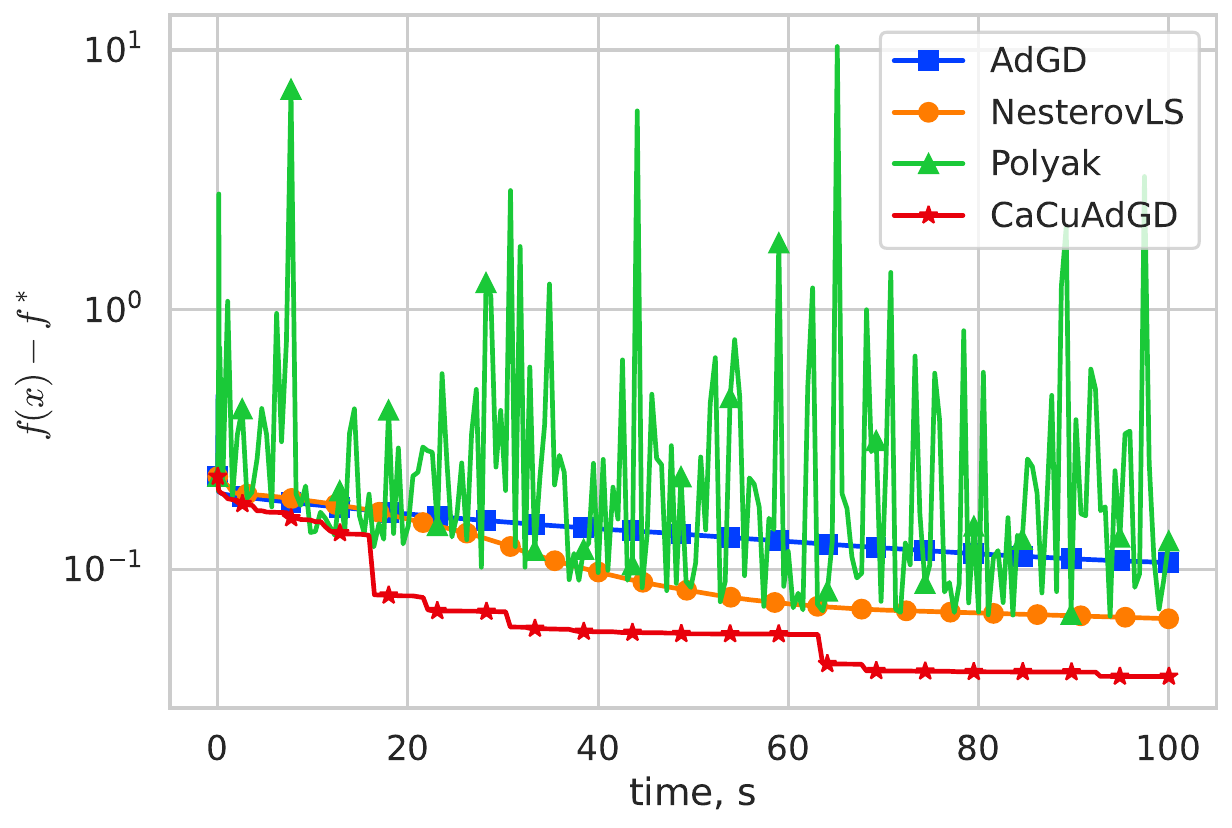}
        \caption{covtype}
        \label{fig:covtype-cacuadgd}
    \end{minipage}
    \begin{minipage}[htp]{0.24\textwidth}
        \centering
        \includegraphics[width=1\linewidth]{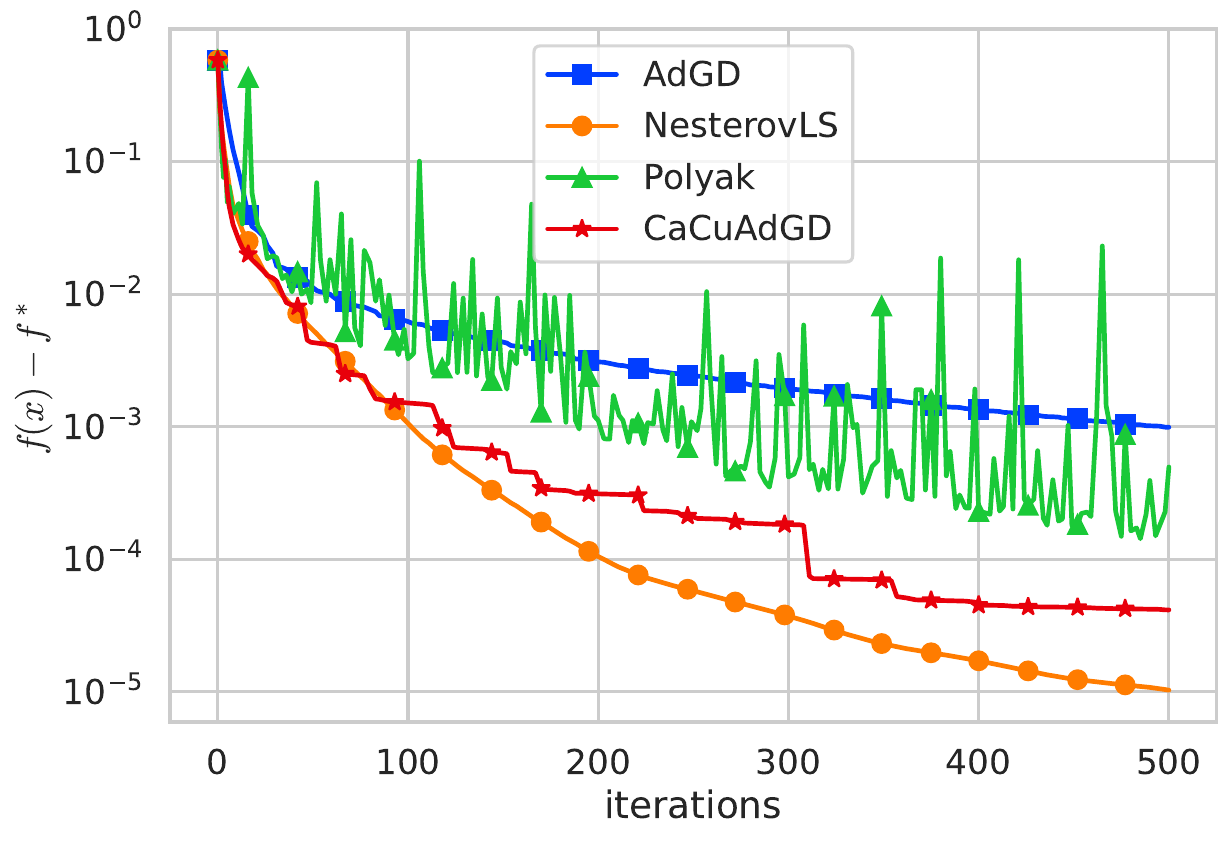}
        \includegraphics[width=1\linewidth]{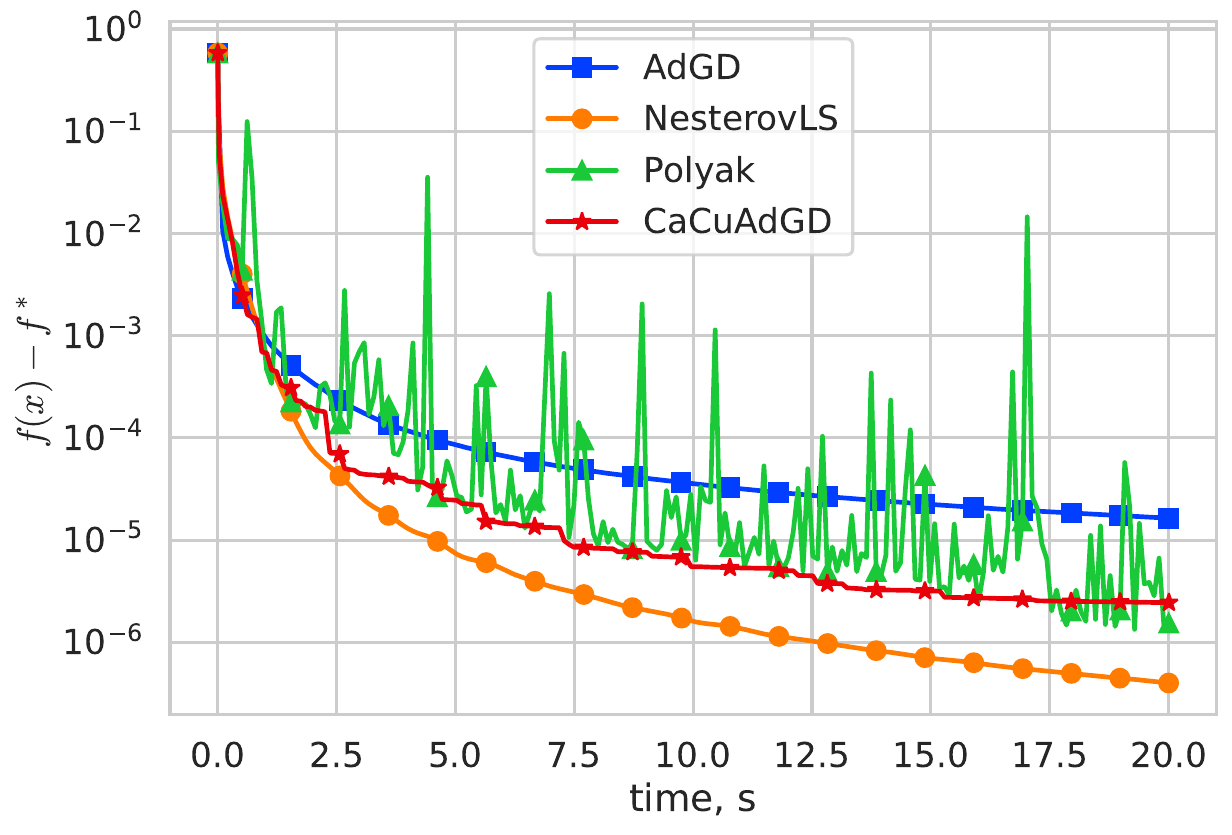}
        \caption{w8a}
        \label{fig:w8a-cacuadgd}
    \end{minipage}
    \begin{minipage}[htp]{0.24\textwidth}
        \centering
        \includegraphics[width=1\linewidth]{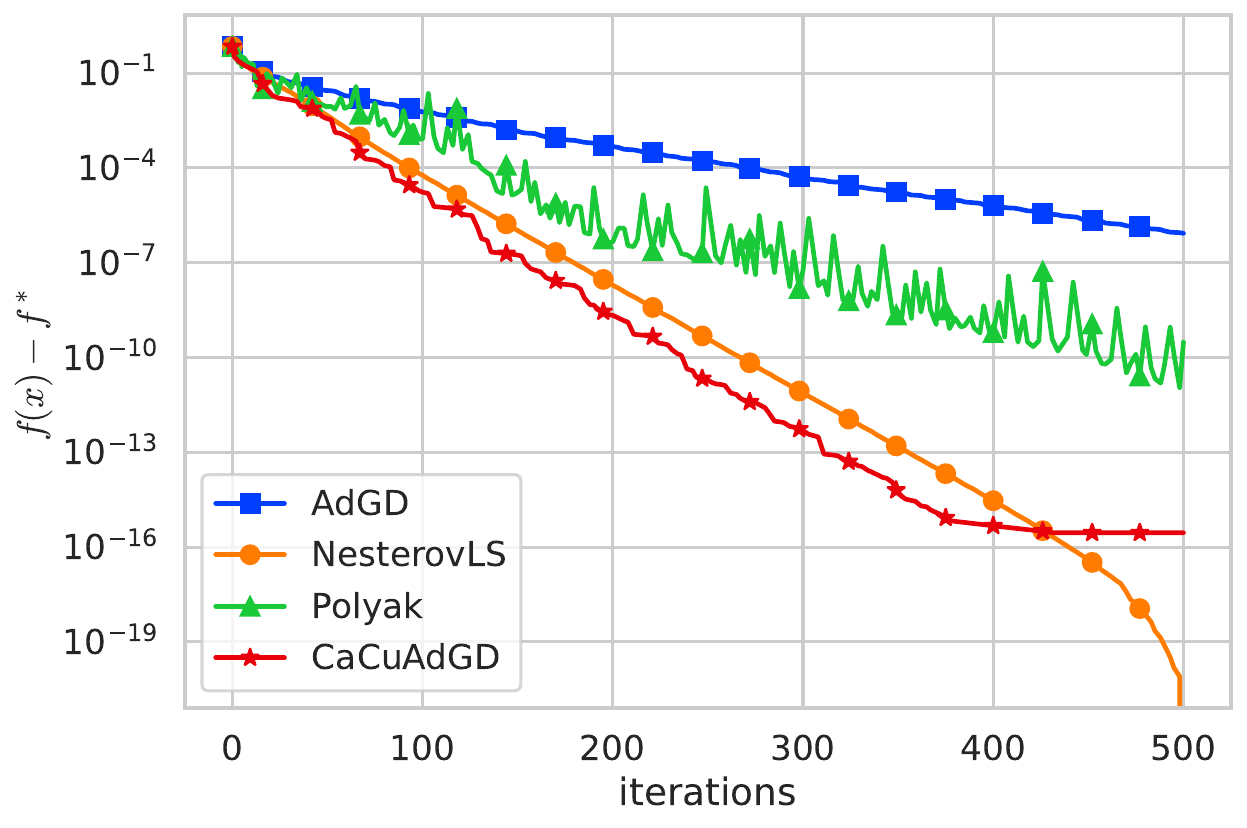}
        \includegraphics[width=1\linewidth]{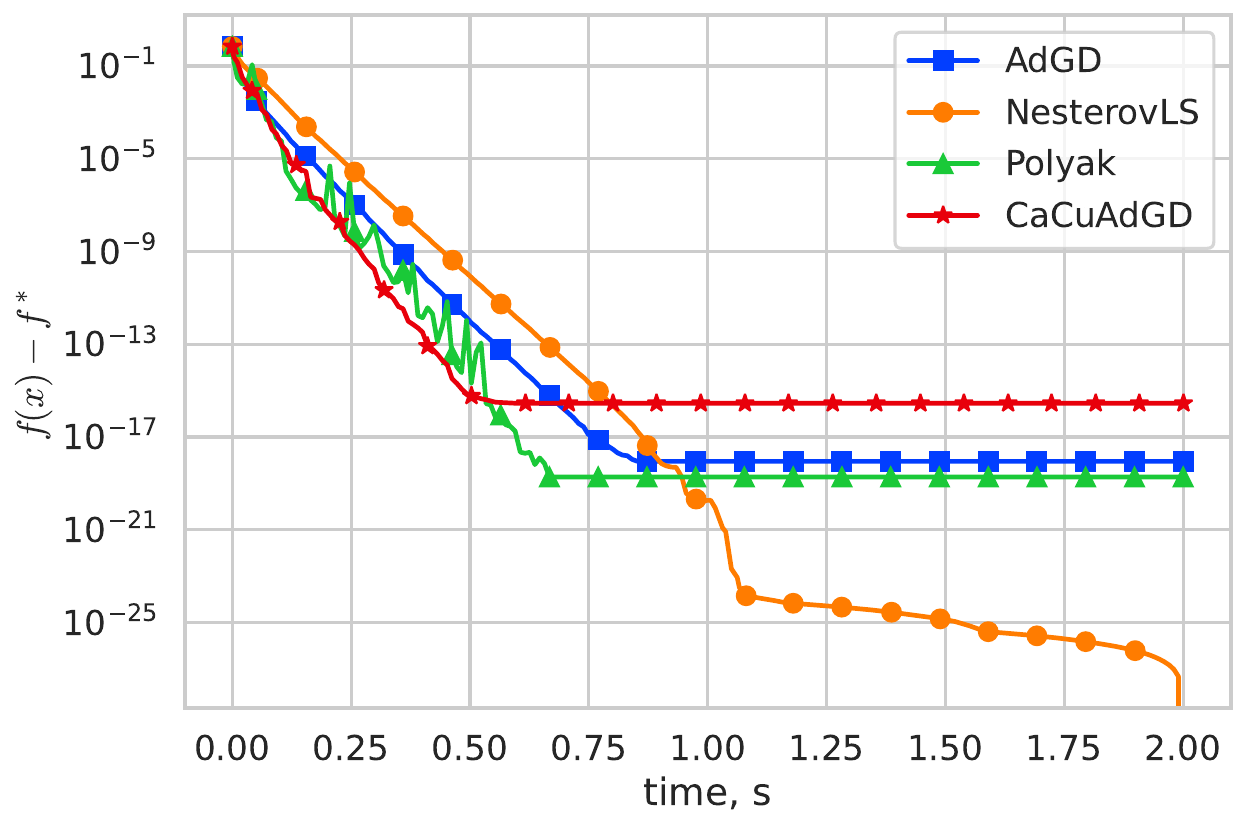}
        \caption{mushrooms}
        \label{fig:mushrooms-cacuadgd}
    \end{minipage}
    \begin{minipage}[htp]{0.24\textwidth}
        \centering
        \includegraphics[width=1\linewidth]{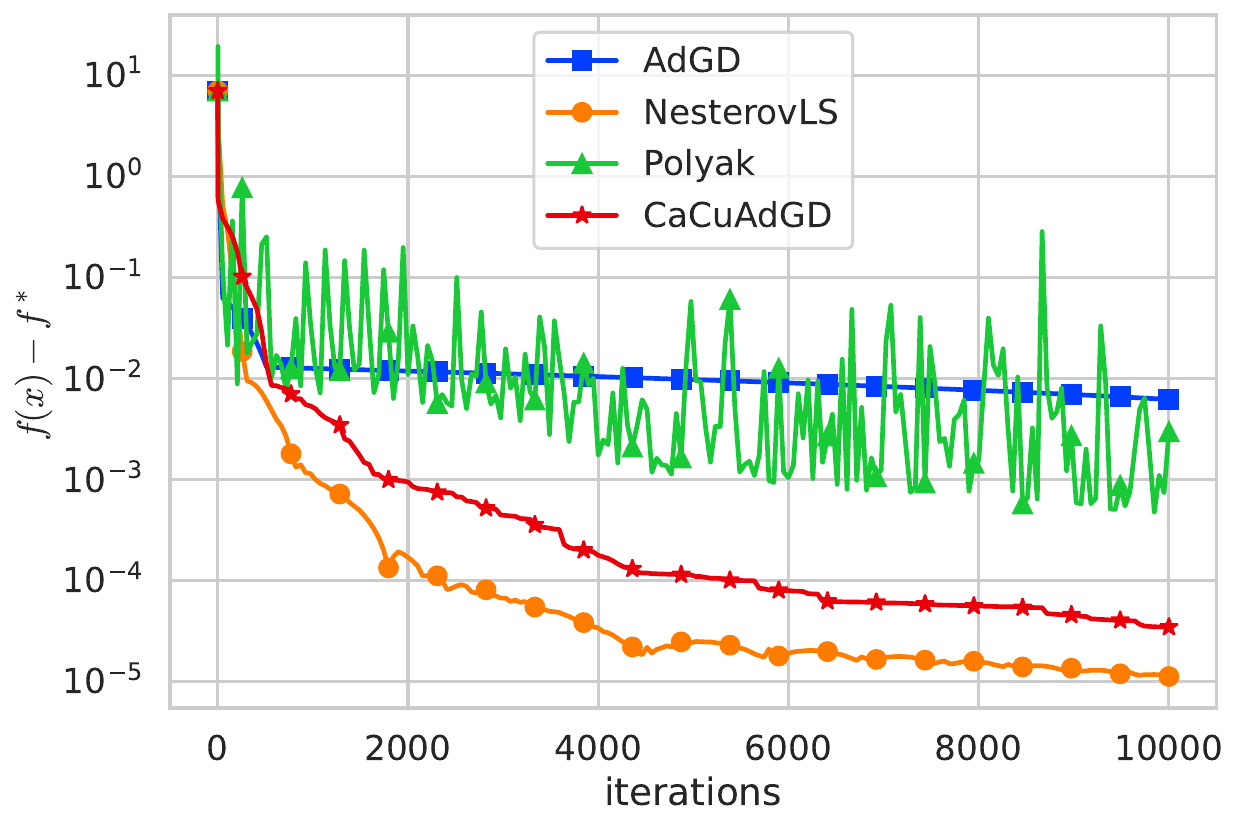}
        \includegraphics[width=1\linewidth]{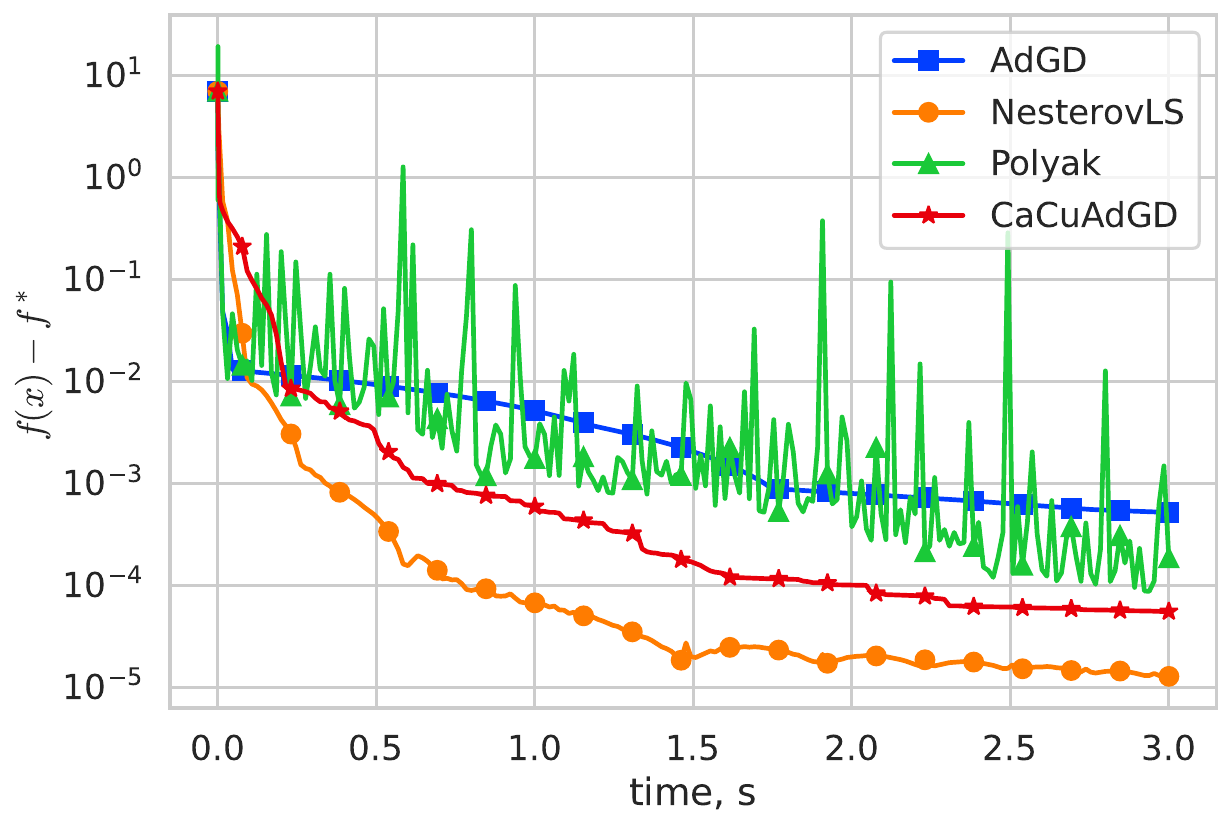}
        \caption{logsumexp (0.05)}
        \label{fig:lse-005-cacuadgd}
    \end{minipage}
\end{figure}
\begin{figure}[H]
    \centering
    \begin{minipage}[htp]{0.24\textwidth}
        \centering
        \includegraphics[width=1\linewidth]{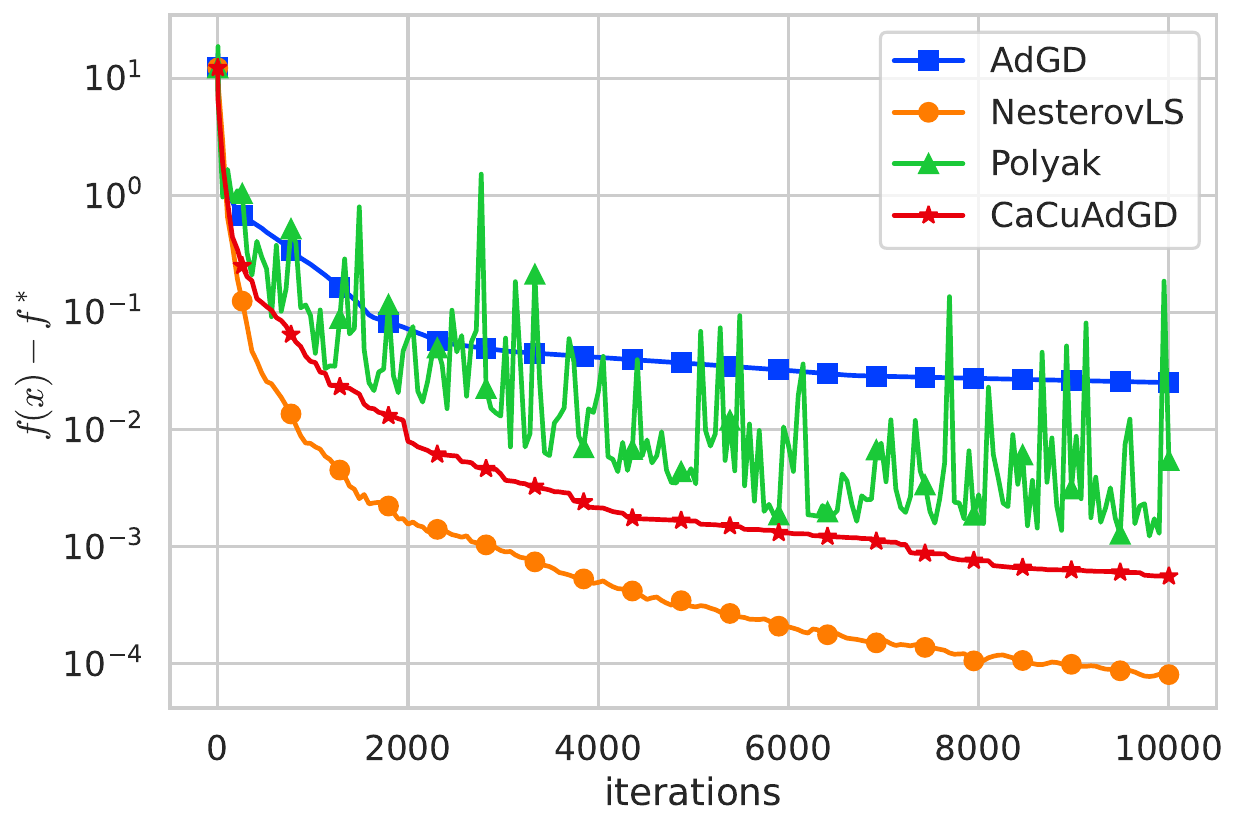}
        \includegraphics[width=1\linewidth]{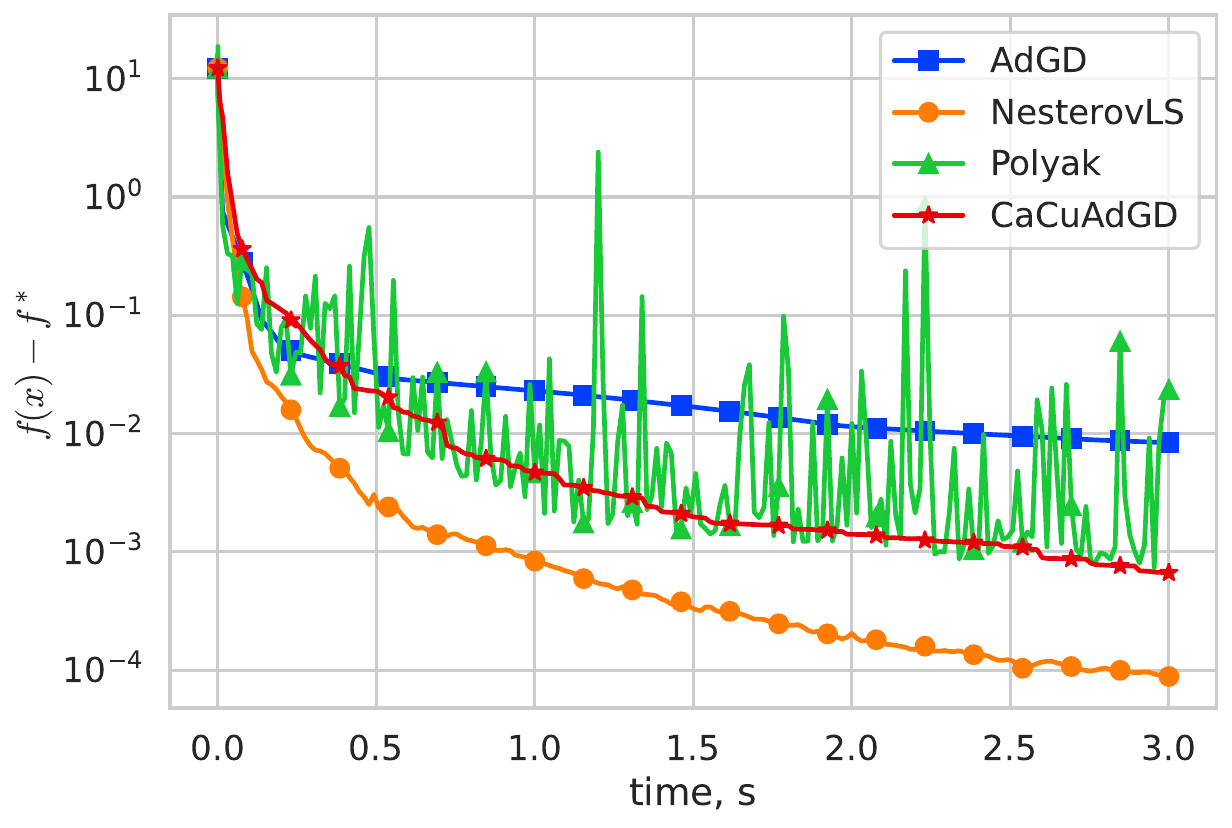}
        \caption{logsumexp (0.10)}
        \label{fig:lse-010-cacuadgd}
    \end{minipage}
    \begin{minipage}[htp]{0.24\textwidth}
        \centering
        \includegraphics[width=1\linewidth]{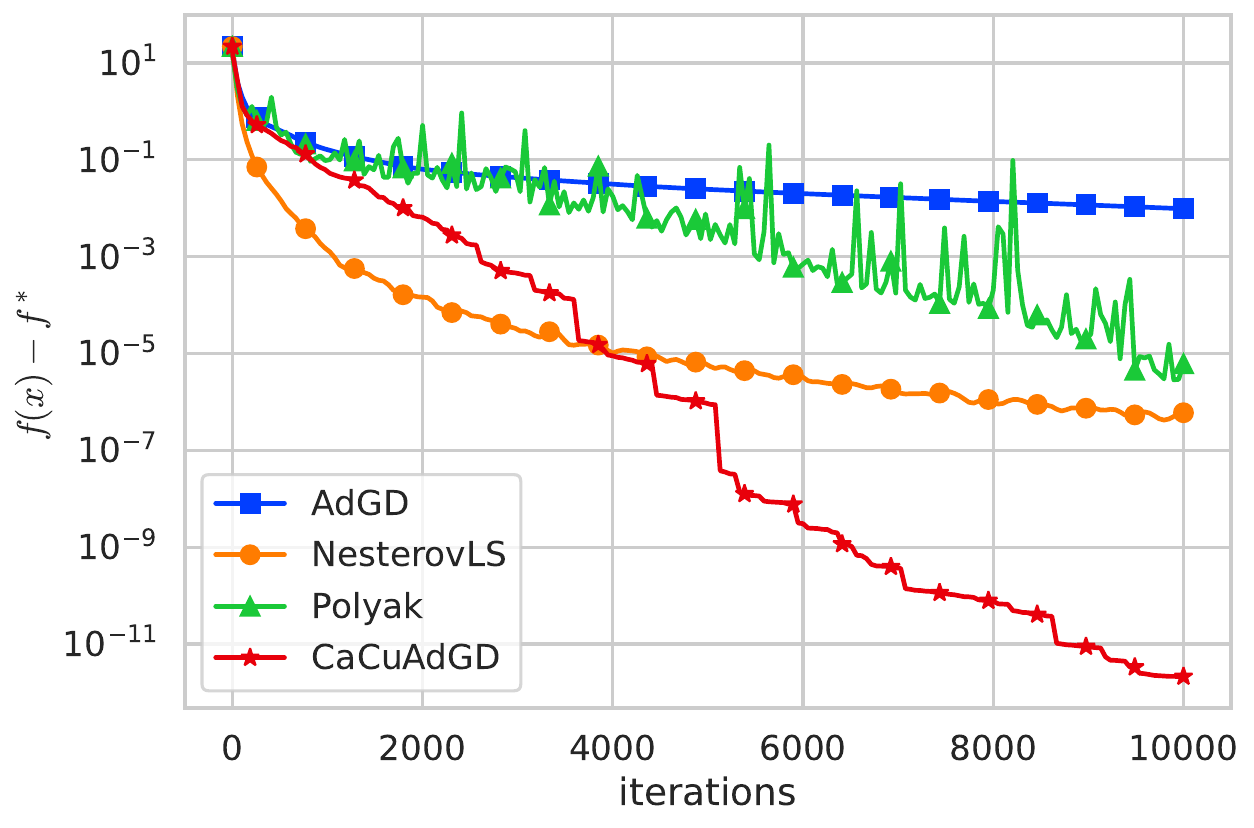}
        \includegraphics[width=1\linewidth]{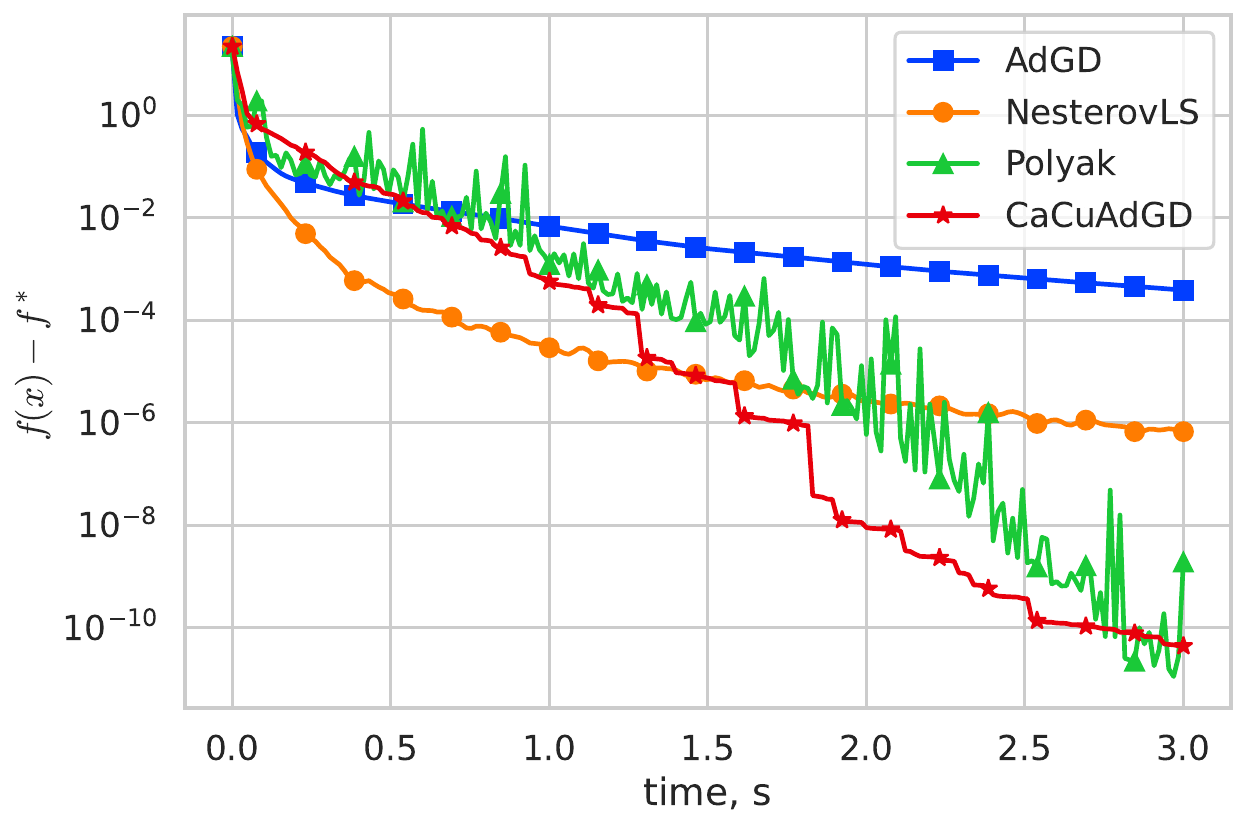}
        \caption{logsumexp (0.25)}
        \label{fig:lse-025-cacuadgd}
    \end{minipage}
    \begin{minipage}[htp]{0.24\textwidth}
        \centering
        \includegraphics[width=1\linewidth]{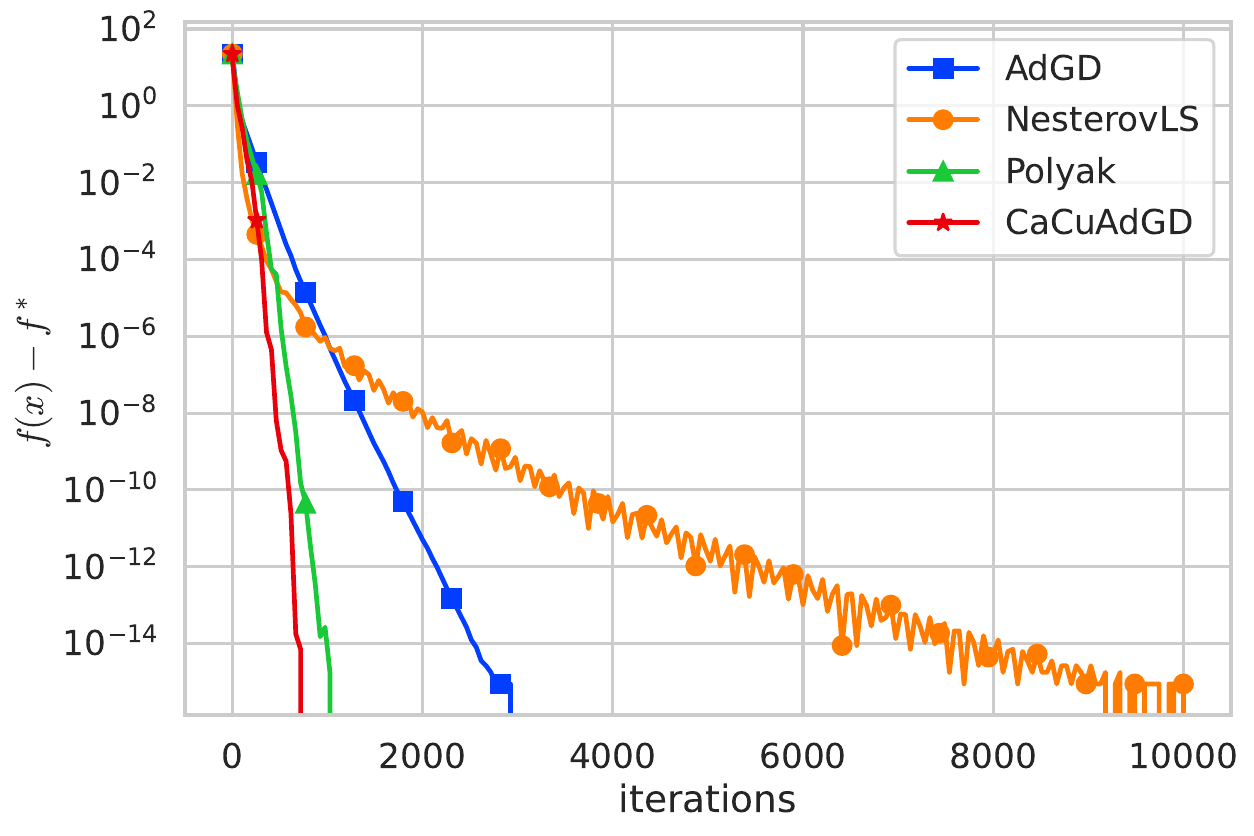}
        \includegraphics[width=1\linewidth]{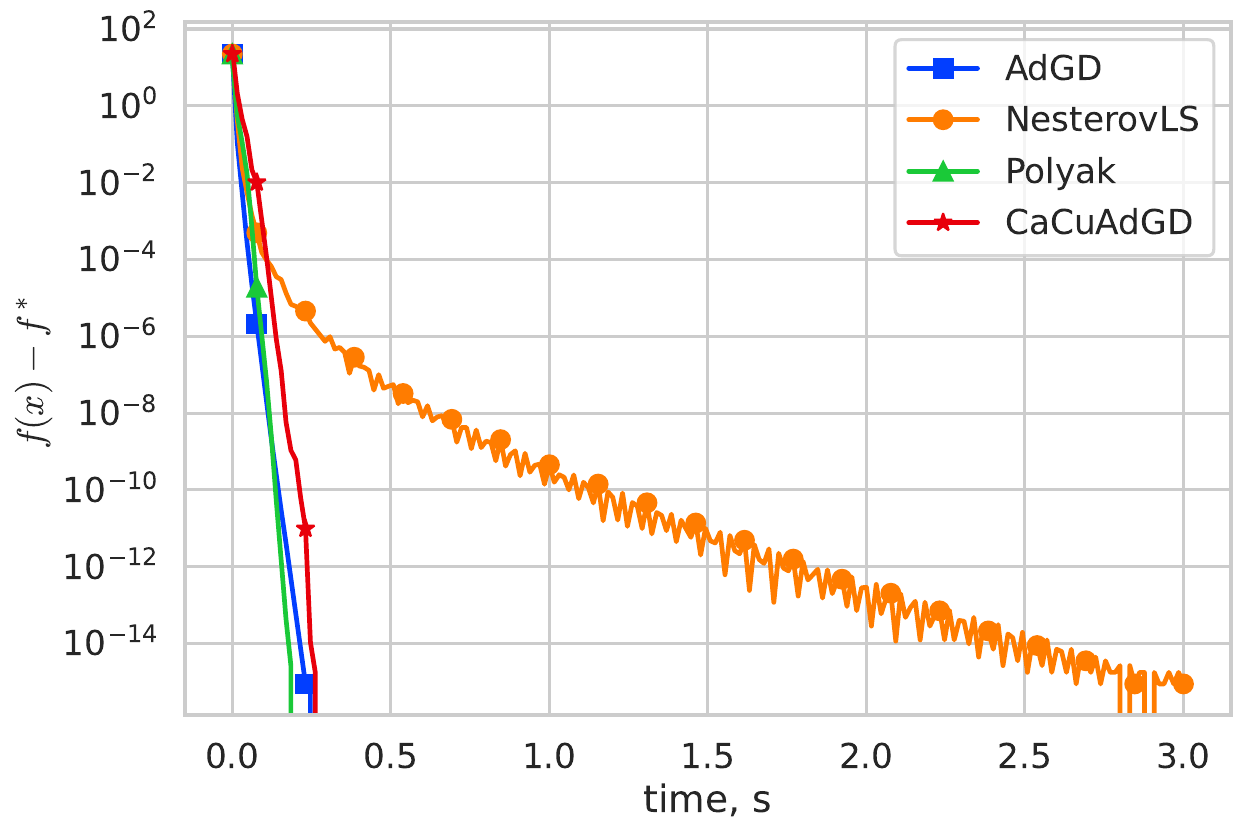}
        \caption{logsumexp (0.50)}
        \label{fig:lse-050-cacuadgd}
    \end{minipage}
    \begin{minipage}[htp]{0.24\textwidth}
        \centering
        \includegraphics[width=1\linewidth]{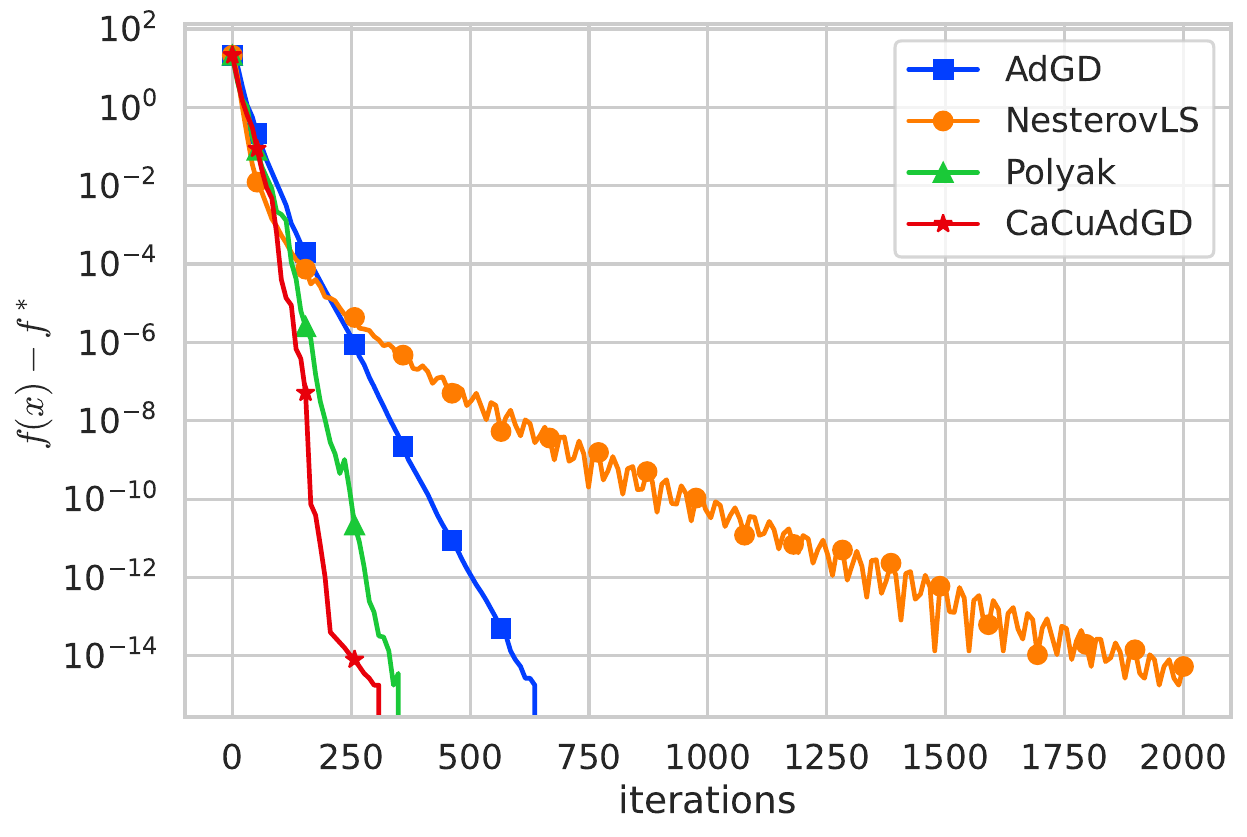}
        \includegraphics[width=1\linewidth]{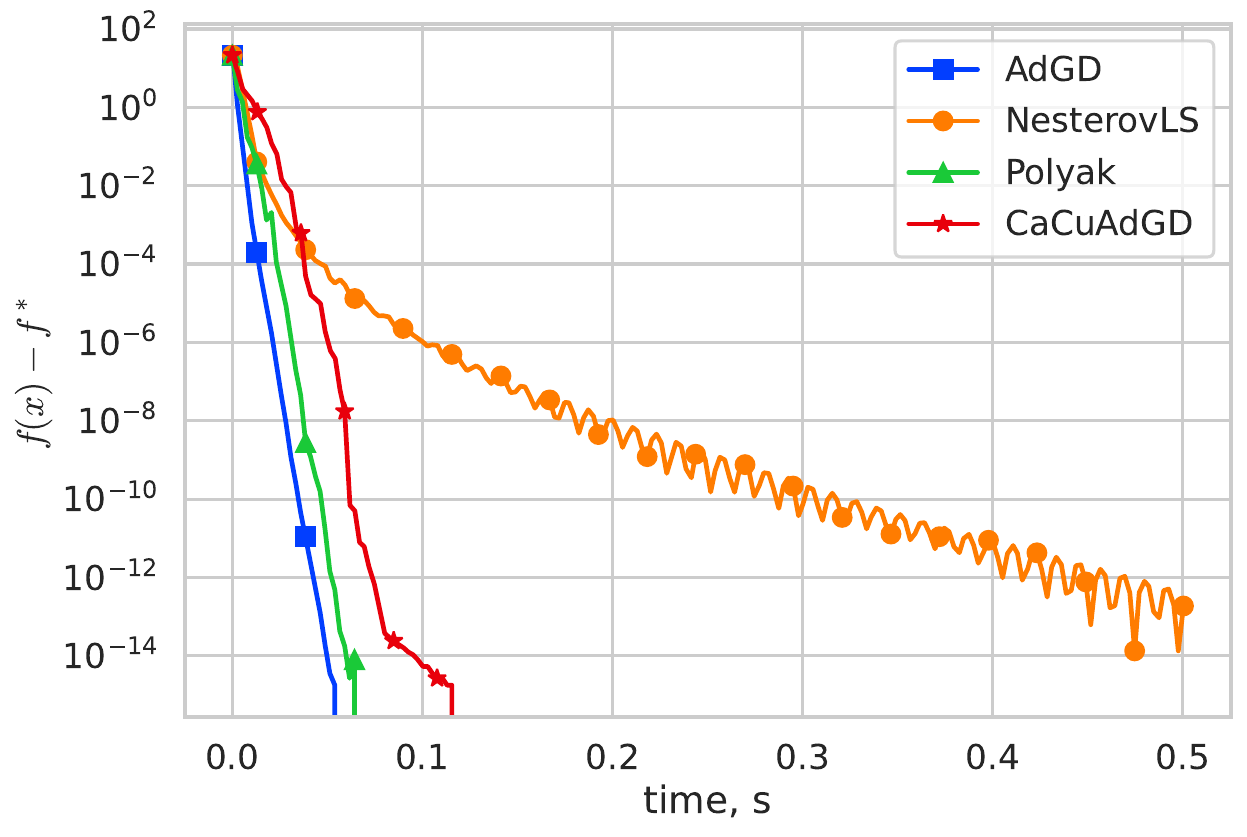}
        \caption{logsumexp (0.75)}
        \label{fig:lse-075-cacuadgd}
    \end{minipage}
\end{figure}
One can see, that \algname{CaCuAdGD} consistently outperforms \algname{AdGD} on hard problems and often matches or surpasses \algname{Polyak} and \algname{NesterovLS}, otherwise staying competitive.

\subsubsection{First-Order Method ($\ell_2$ regularized)}
We repeat the experiments with a small $\ell=10^{-7}$ term in~\cref{eq:logreg} to ensure a unique, well-conditioned optimum and to mirror standard practice. The results are presented in \cref{fig:covtype-l2-1e-7,fig:w8a-l2-1e-7,fig:mushrooms-l2-1e-7,fig:logsumexp005-l2-1e-7}.
\begin{figure}[H]
    \centering
    \begin{minipage}[htp]{0.24\textwidth}
        \centering
        \includegraphics[width=1\linewidth]{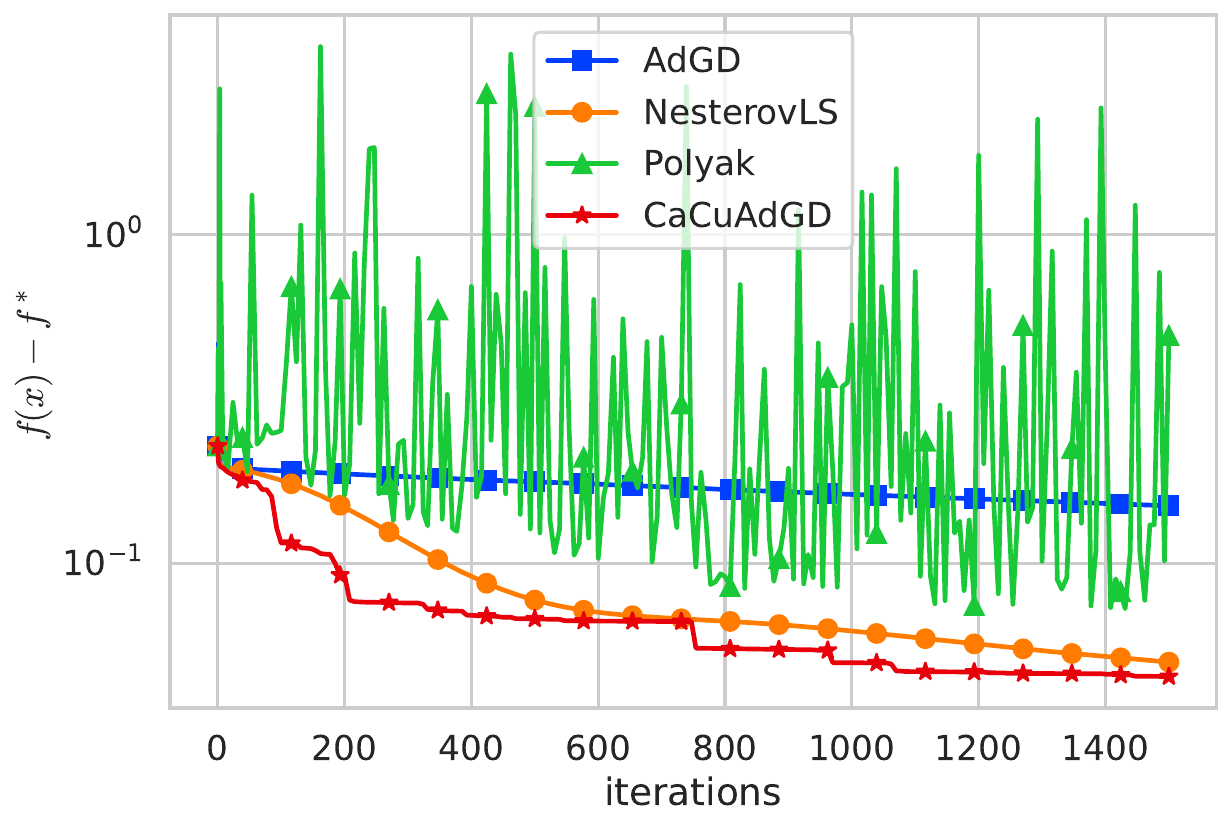}
        \includegraphics[width=1\linewidth]{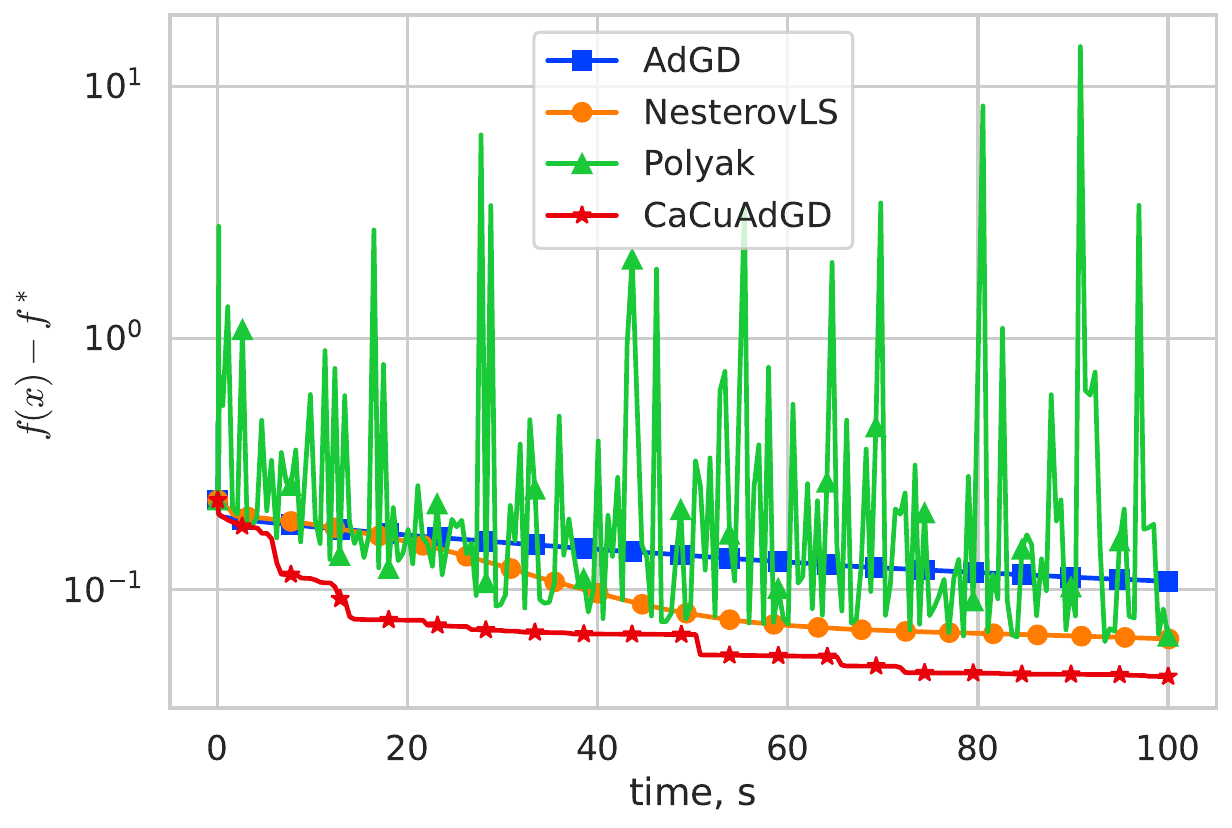}
        \caption{covtype}\label{fig:covtype-l2-1e-7}
    \end{minipage}
    \begin{minipage}[htp]{0.24\textwidth}
        \centering
        \includegraphics[width=1\linewidth]{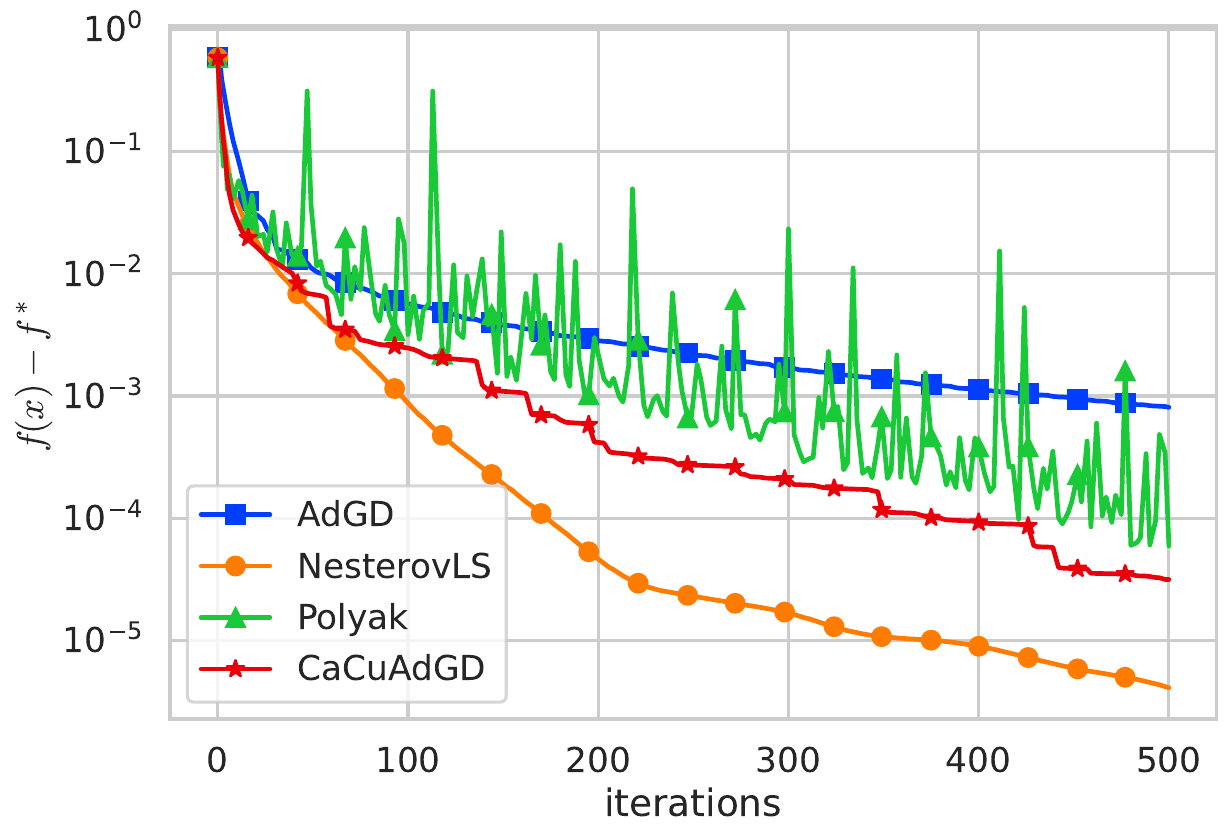}
        \includegraphics[width=1\linewidth]{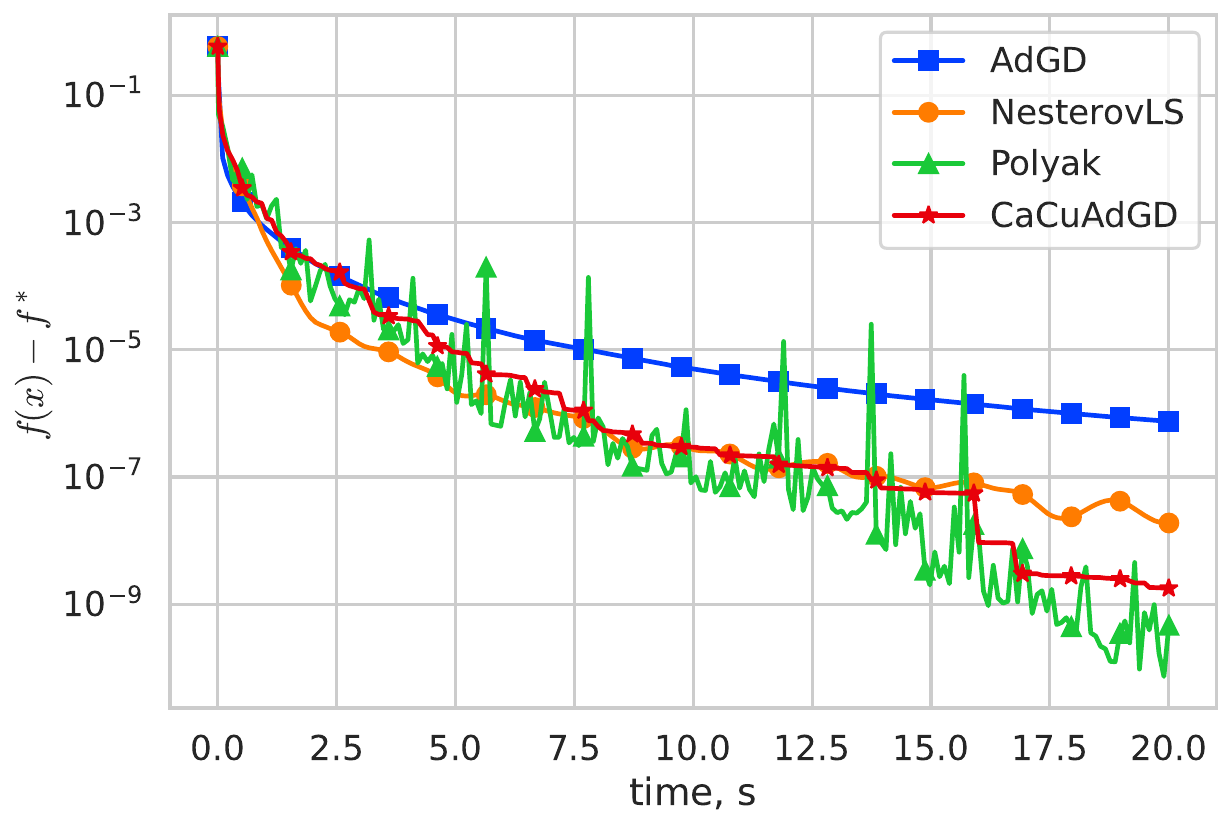}
        \caption{w8a}\label{fig:w8a-l2-1e-7}
    \end{minipage}
    \begin{minipage}[htp]{0.24\textwidth}
        \centering
        \includegraphics[width=1\linewidth]{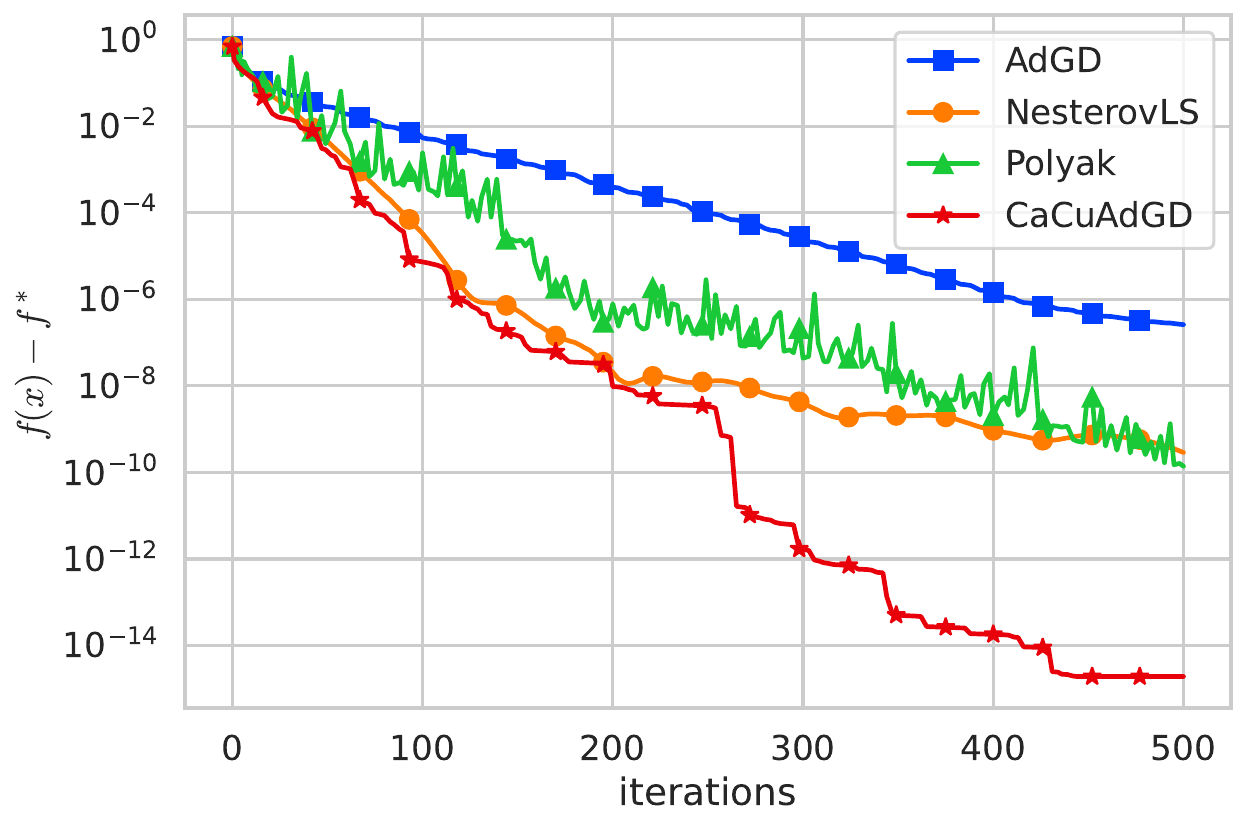}
        \includegraphics[width=1\linewidth]{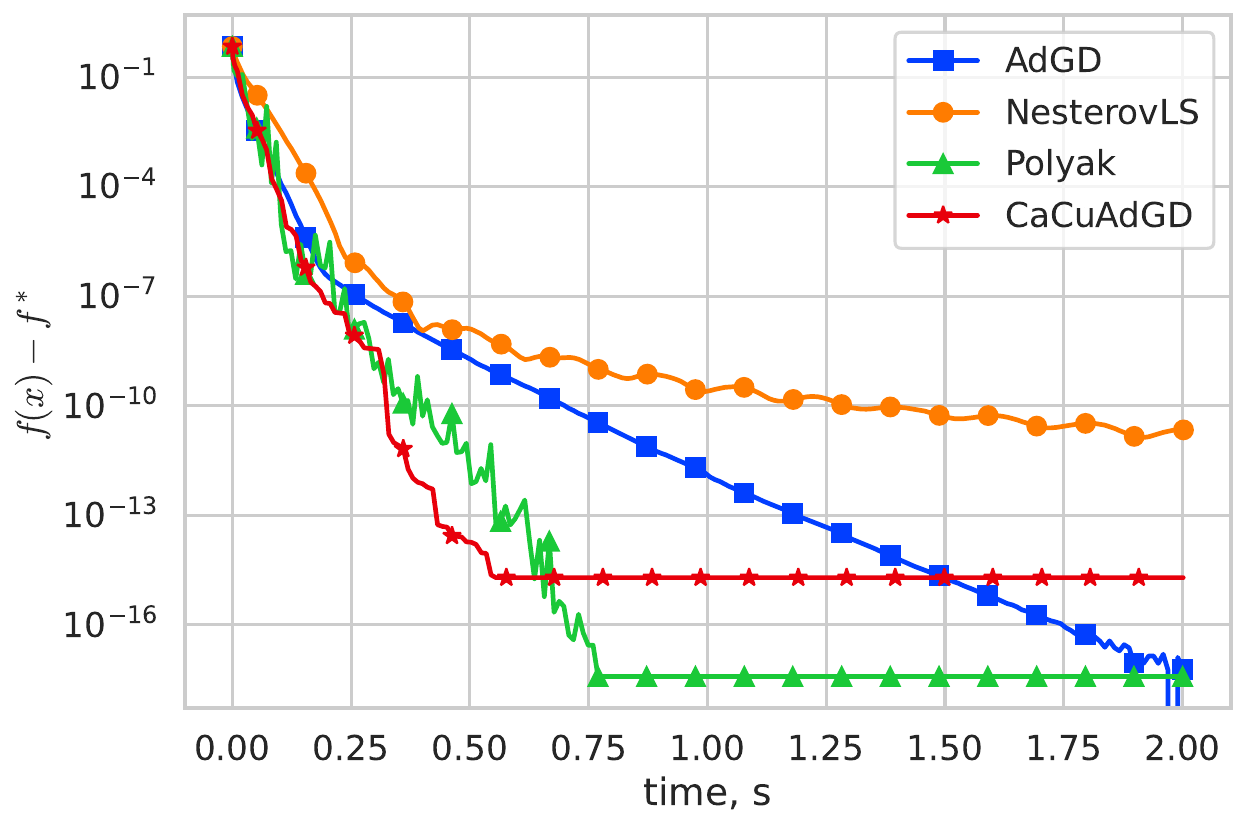}
        \caption{mushrooms}\label{fig:mushrooms-l2-1e-7}
    \end{minipage}
    \begin{minipage}[htp]{0.24\textwidth}
        \centering
        \includegraphics[width=1\linewidth]{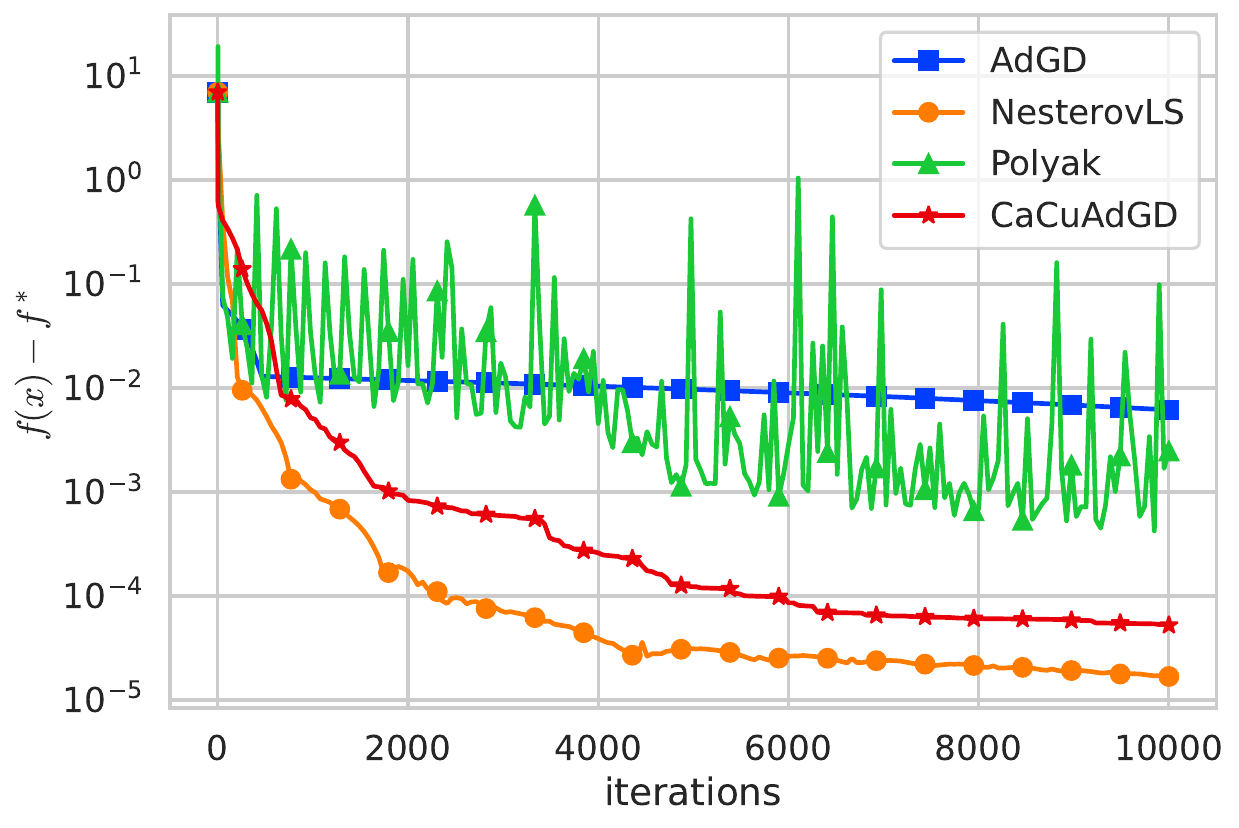}
        \includegraphics[width=1\linewidth]{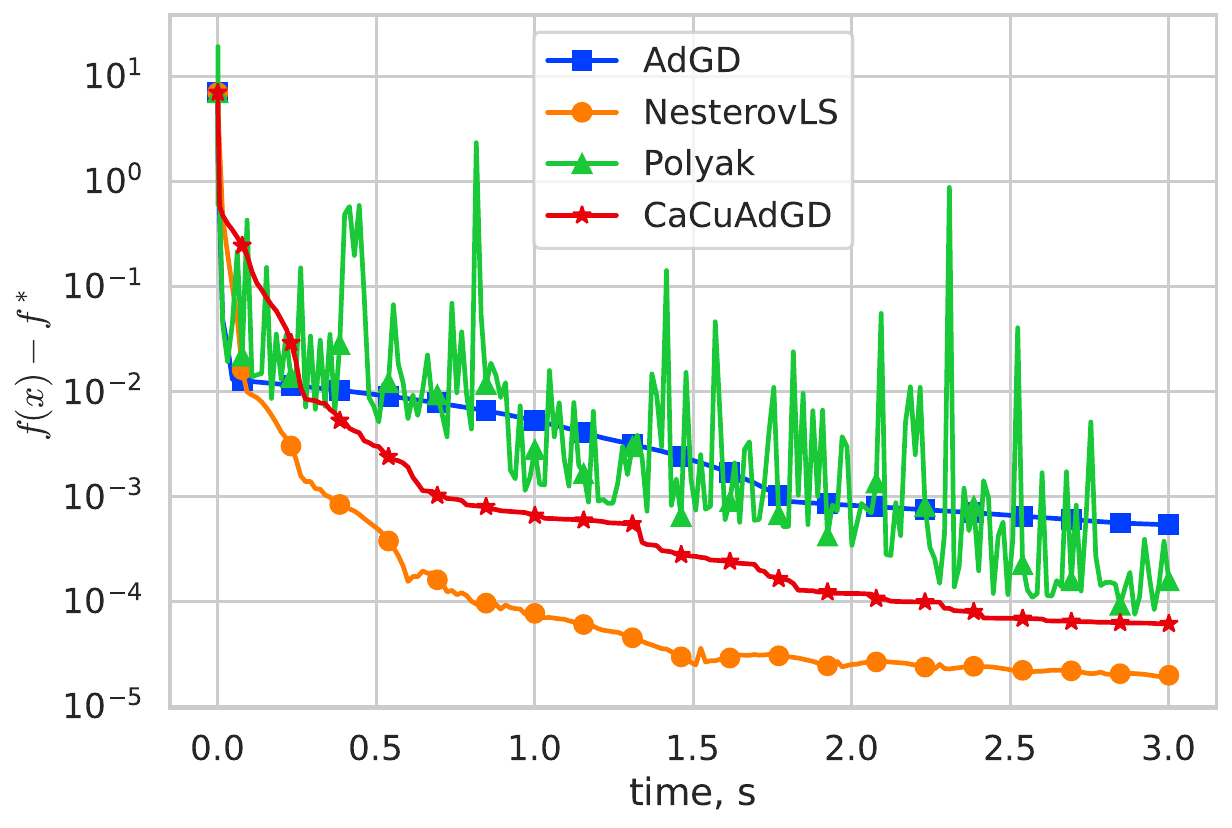}
        \caption{logsumexp (0.05)}
        \label{fig:logsumexp005-l2-1e-7}
    \end{minipage}
\end{figure}
One can spot a difference for `w8a' and `mushrooms': \algname{NewtonLS} converges slower, allowing \algname{CaCuAdGD} to show faster wall-clock performance.

\subsection{Stochastic Extension} \label{sec:stoch exp}
In this section we study logistic regression~\eqref{eq:logreg} in the stochastic finite-sum setting using \algname{SGD} ($x^{k+1} = x^k - {g^k}/{\widehat L}$) and \algname{CaCuSGD} (\cref{thm:cacusgd}). We train with shuffled mini-batches of size $n/10$ (drop-last) and sweep $\widehat L$ and $\widehat H$ (for \algname{CaCuSGD}). The results are presented in \cref{fig:hatLhundred,fig:hatL10,fig:hatL1}.

% \ref{fig:hatLhundred}
\begin{figure}[H]
    \centering
    \begin{minipage}[htp]{0.24\textwidth}
        \centering
        \includegraphics[width=1\linewidth]{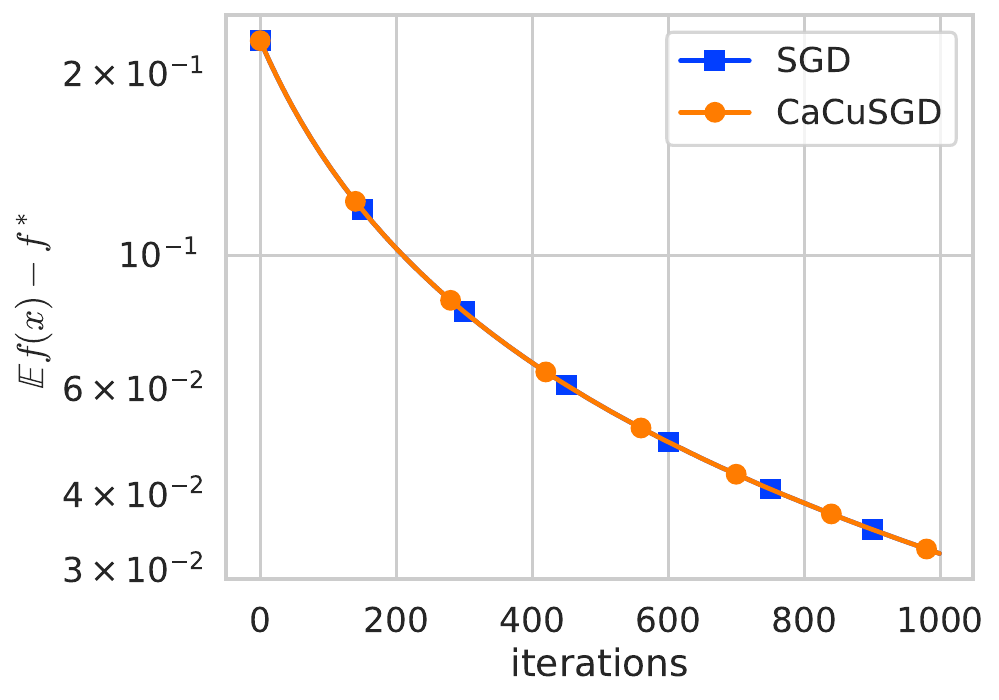}
        % \caption{$\widehat H = 0.1$}
    \end{minipage}
    \begin{minipage}[htp]{0.24\textwidth}
        \centering
        \includegraphics[width=1\linewidth]{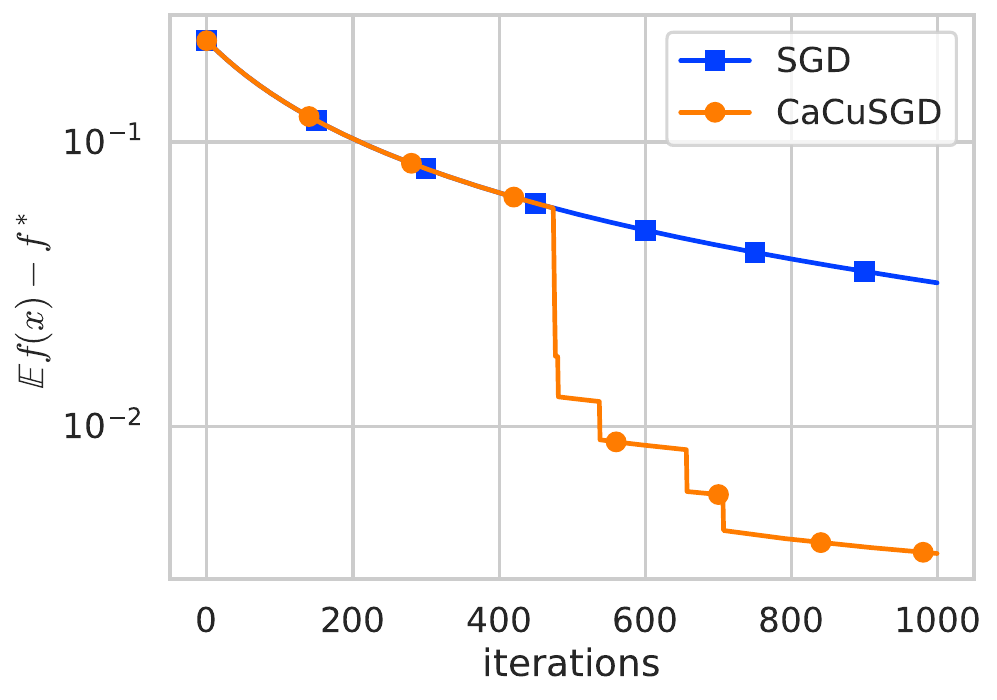}
        % \caption{$\widehat H = 1$}
    \end{minipage}
    \begin{minipage}[htp]{0.24\textwidth}
        \centering
        \includegraphics[width=1\linewidth]{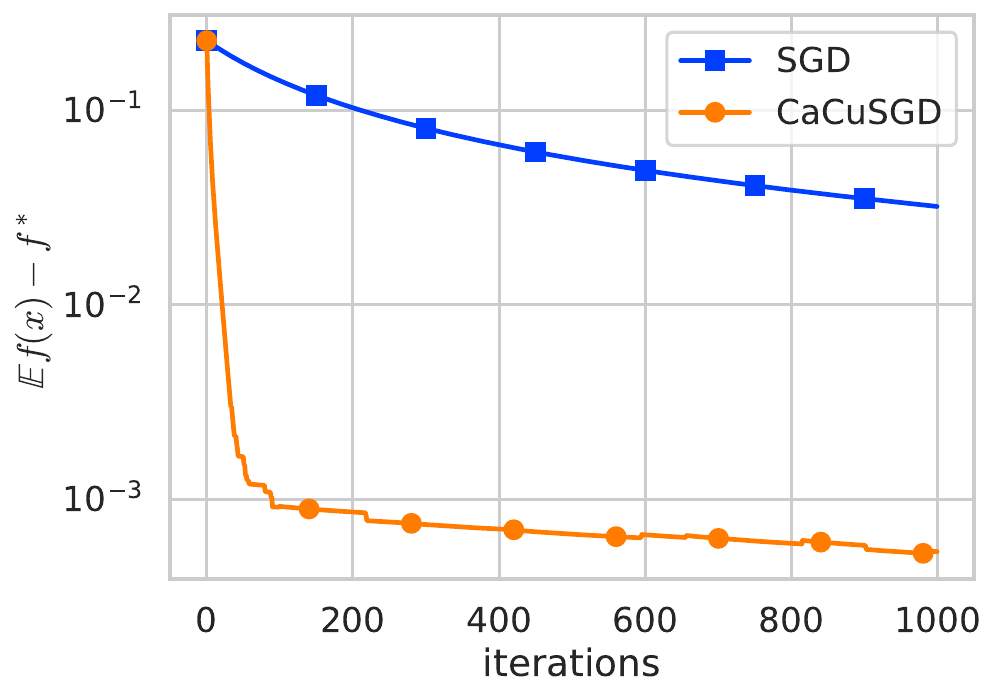}
        % \caption{$\widehat H = 10$}
    \end{minipage}
    \begin{minipage}[htp]{0.24\textwidth}
        \centering
        \includegraphics[width=1\linewidth]{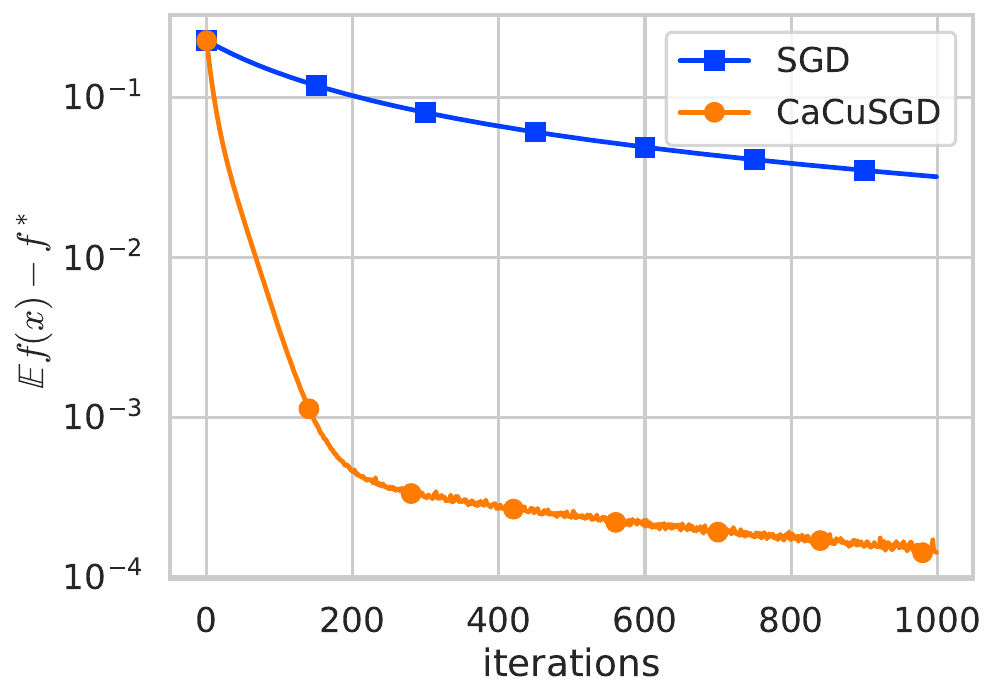}
        % \caption{$\widehat H = 100$}
    \end{minipage}
    \caption{$\widehat L = 100$, $\widehat H = \{0.1, 1, 10, 100\}$ (from left to right)}
    \label{fig:hatLhundred}
\end{figure}

\begin{figure}[H]
    \centering
    \begin{minipage}[htp]{0.24\textwidth}
        \centering
        \includegraphics[width=1\linewidth]{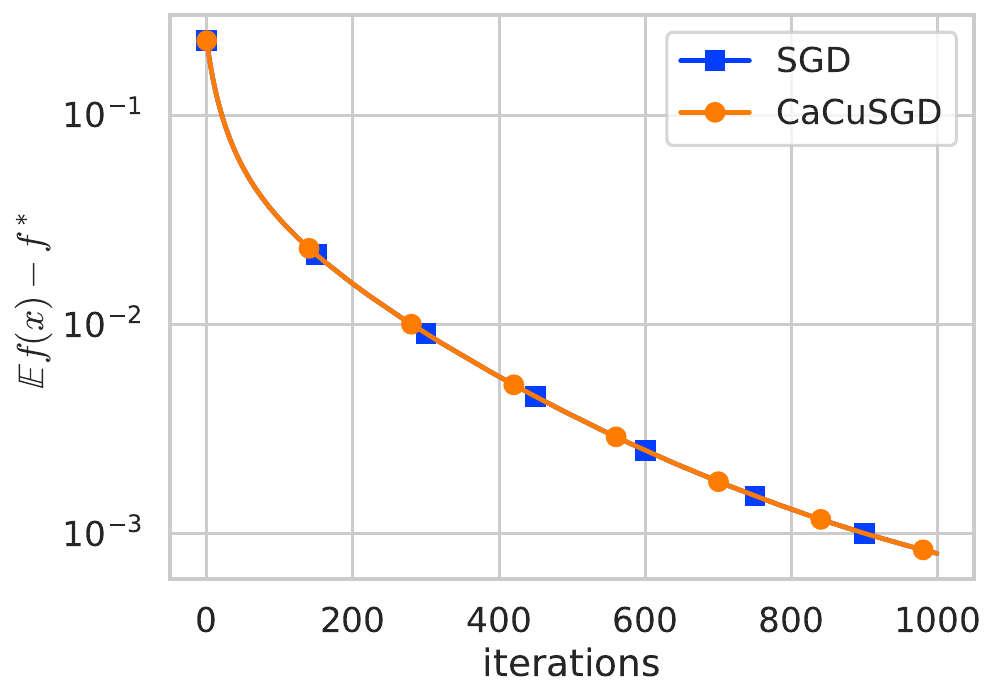}
        % \caption{$\widehat H = 0.1$}
    \end{minipage}
    \begin{minipage}[htp]{0.24\textwidth}
        \centering
        \includegraphics[width=1\linewidth]{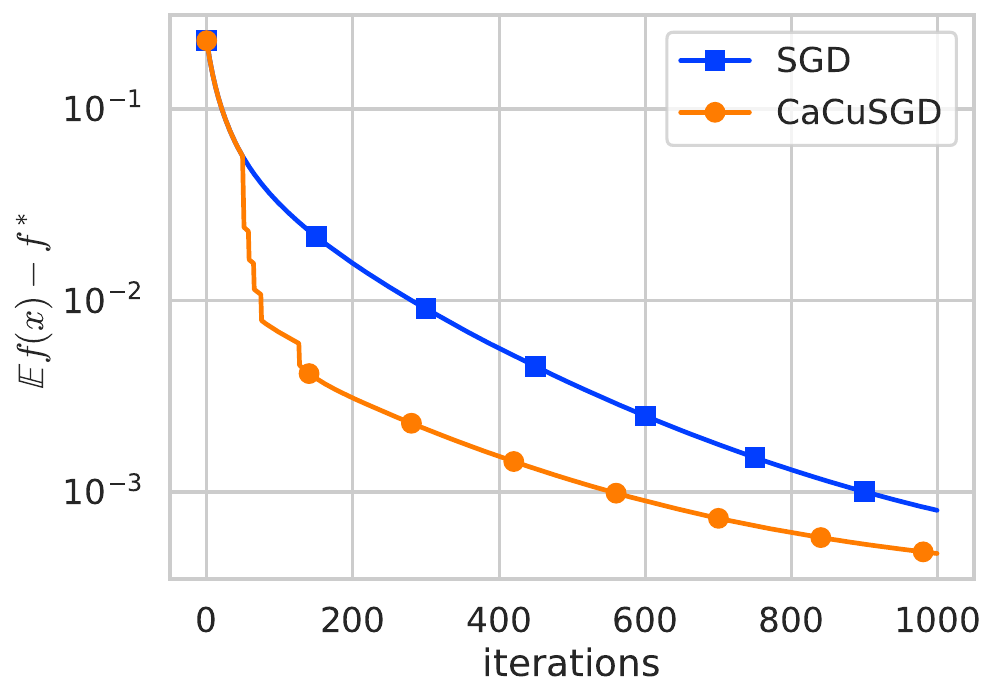}
        % \caption{$\widehat H = 1$}
    \end{minipage}
    \begin{minipage}[htp]{0.24\textwidth}
        \centering
        \includegraphics[width=1\linewidth]{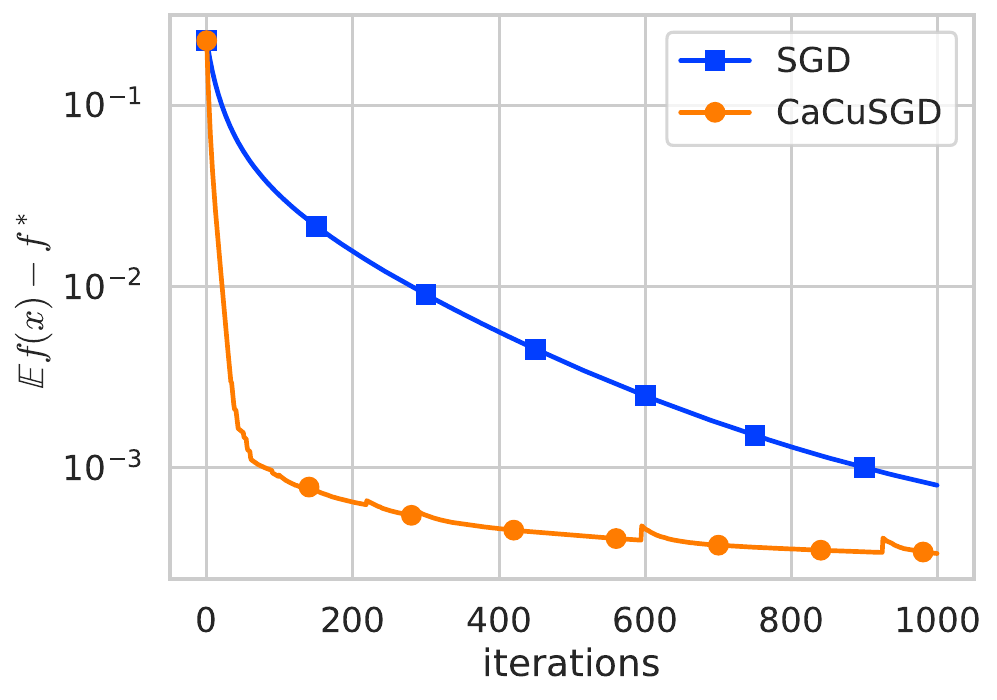}
        % \caption{$\widehat H = 10$}
    \end{minipage}
    \begin{minipage}[htp]{0.24\textwidth}
        \centering
        \includegraphics[width=1\linewidth]{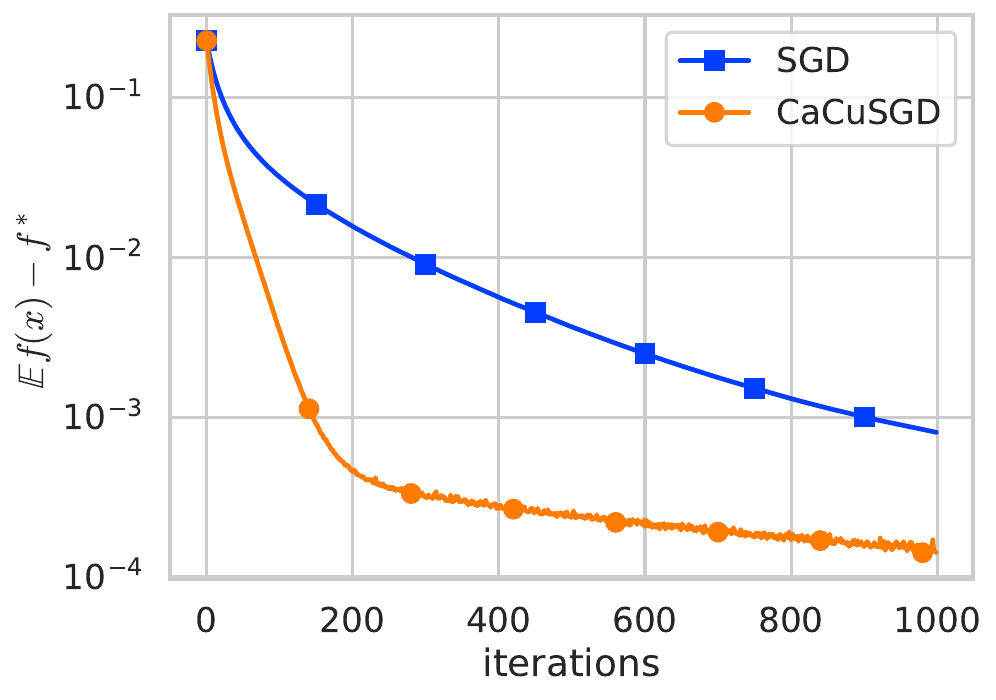}
        % \caption{$\widehat H = 100$}
    \end{minipage}
    \caption{$\widehat L = 10$, $\widehat H = \{0.1, 1, 10, 100\}$ (from left to right)}
    \label{fig:hatL10}
\end{figure}

\begin{figure}[H]
    \centering
    \begin{minipage}[htp]{0.24\textwidth}
        \centering
        \includegraphics[width=1\linewidth]{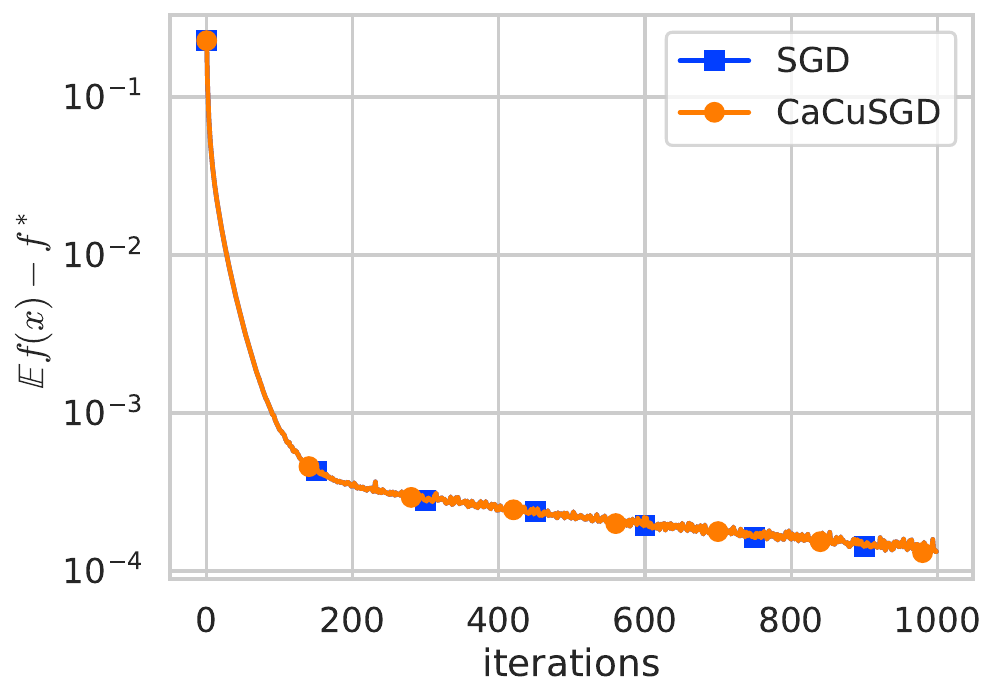}
        % \caption{$\widehat H = 0.1$}
    \end{minipage}
    \begin{minipage}[htp]{0.24\textwidth}
        \centering
        \includegraphics[width=1\linewidth]{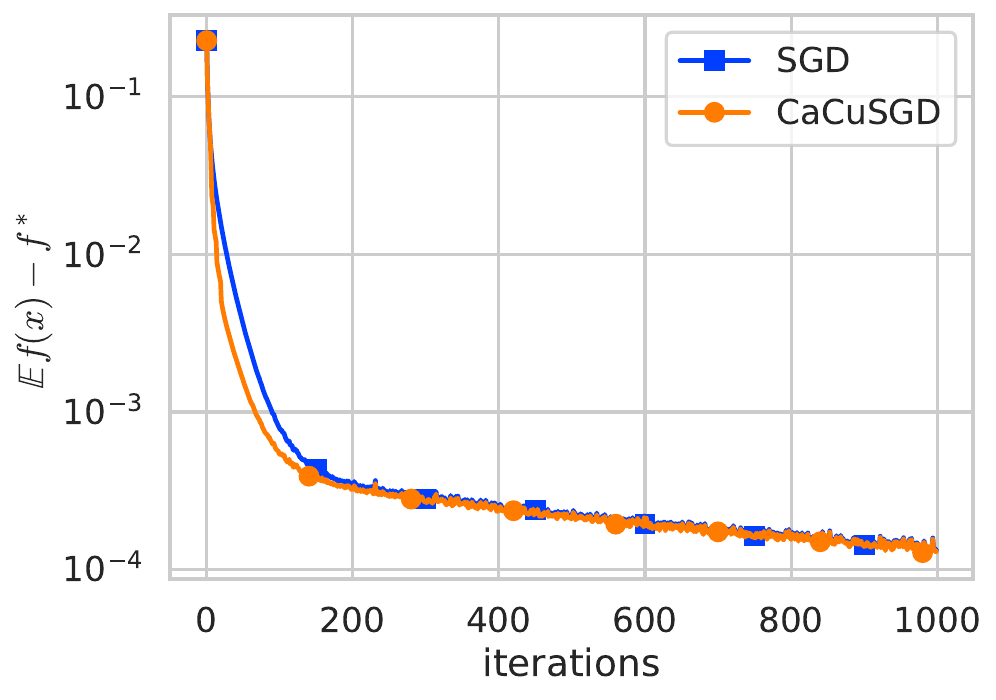}
        % \caption{$\widehat H = 1$}
    \end{minipage}
    \begin{minipage}[htp]{0.24\textwidth}
        \centering
        \includegraphics[width=1\linewidth]{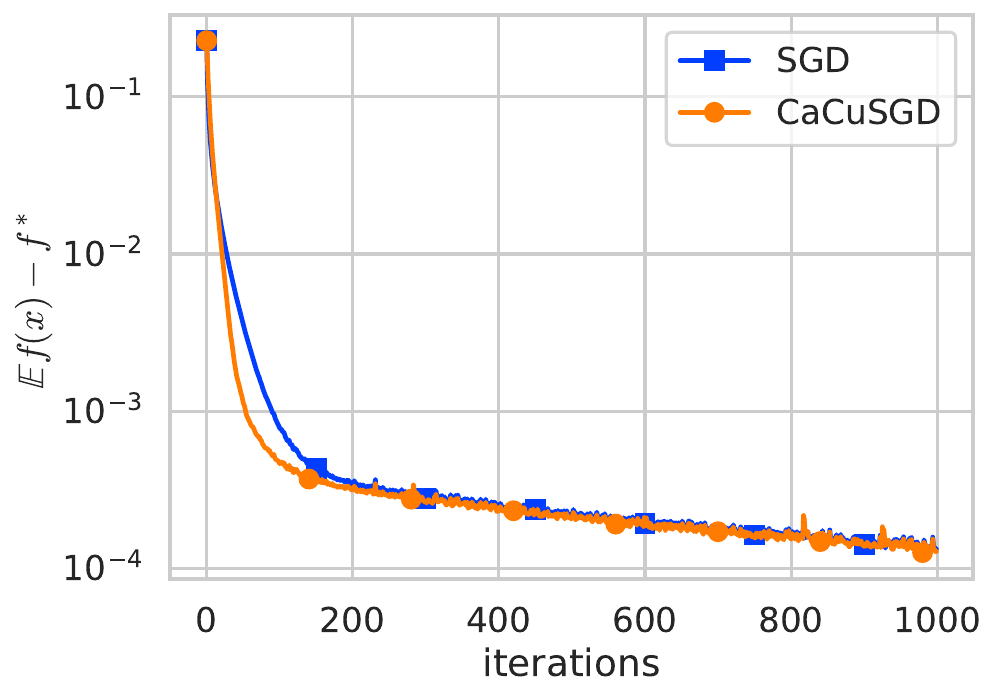}
        % \caption{$\widehat H = 10$}
    \end{minipage}
    \begin{minipage}[htp]{0.24\textwidth}
        \centering
        \includegraphics[width=1\linewidth]{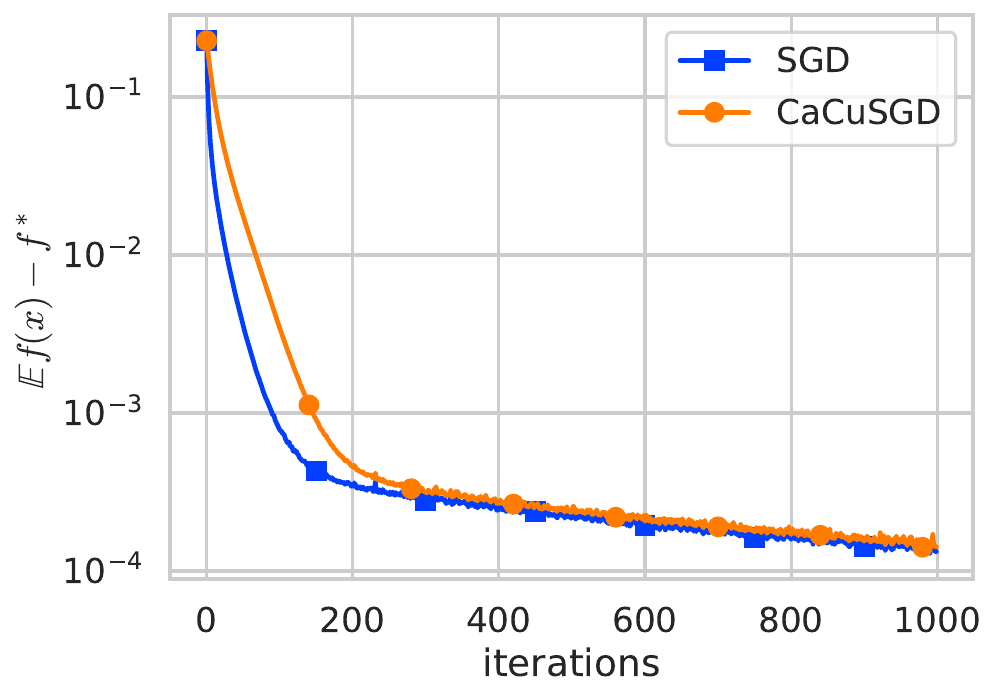}
        % \caption{$\widehat H = 100$}
    \end{minipage}
    \caption{$\widehat L = 1$, $\widehat H = \{0.1, 1, 10, 100\}$ (from left to right)}
    \label{fig:hatL1}
\end{figure}

Even with a conservative stepsize (large $\widehat L$), our scheme rapidly reaches the noise floor. In the rightmost panel ($\widehat H=100$), the floor remains the same but is attained sooner: smaller stepsizes yield a lower floor yet normally slow convergence, whereas \algname{CaCuSGD} reaches that (smaller) floor faster. Thus one can fix a small stepsize to target a low floor without elaborate tuning, in another words, the method effectively acts as an implicit scheduler.

% \section{Conclusion}

% future works directions:

% polyak stepsize under hessian smoothness

% two steps search: small step then near zero directinal smoothness

% cacuadgd acceleration -- momentum

% adaptivity without backtracking

% understading logsumexp

% variance reduction, fl setting

% higher order smoothness

% combining gradient step with qn step based hessian hvp approximation, i.e. damped cubic

% check the FL paper with HVPs (if it estimates true hessian-gradient product), and genebileize stochastic setup a a distributed one

\bibliographystyle{apalike}
\bibliography{refs}

\begin{thebibliography}{}

\bibitem[Agafonov et~al., 2025]{agafonov2025simple}
Agafonov, A., Ryspayev, V., Horv{\'a}th, S., Gasnikov, A., Tak{\'a}{\v{c}}, M., and Hanzely, S. (2025).
\newblock Simple stepsize for quasi-newton methods with global convergence guarantees.
\newblock {\em arXiv preprint arXiv:2508.19712}.

\bibitem[Agarwal et~al., 2016]{agarwal2016second}
Agarwal, N., Bullins, B., and Hazan, E. (2016).
\newblock Second-order stochastic optimization in linear time.
\newblock {\em stat}, 1050:15.

\bibitem[Bekas et~al., 2007]{bekas2007estimator}
Bekas, C., Kokiopoulou, E., and Saad, Y. (2007).
\newblock An estimator for the diagonal of a matrix.
\newblock {\em Applied numerical mathematics}, 57(11-12):1214--1229.

\bibitem[Berahas et~al., 2022]{berahas2022quasi}
Berahas, A., Jahani, M., Richt{\'a}rik, P., and Tak{\'a}{\v{c}}, M. (2022).
\newblock Quasi-{N}ewton methods for machine learning: forget the past, just sample.
\newblock {\em Optimization Methods and Software}, 37(5):1668--1704.

\bibitem[Chayti et~al., 2023]{chayti2023unified}
Chayti, E.~M., Doikov, N., and Jaggi, M. (2023).
\newblock Unified convergence theory of stochastic and variance-reduced cubic newton methods.
\newblock {\em arXiv preprint arXiv:2302.11962}.

\bibitem[Chen et~al., 2023]{chen2023generalized}
Chen, Z., Zhou, Y., Liang, Y., and Lu, Z. (2023).
\newblock Generalized-smooth nonconvex optimization is as efficient as smooth nonconvex optimization.
\newblock In {\em International Conference on Machine Learning}, pages 5396--5427. PMLR.

\bibitem[Defazio and Mishchenko, 2023]{defazioLearningRateFreeLearningDAdaptation2023}
Defazio, A. and Mishchenko, K. (2023).
\newblock Learning-rate-free learning by d-adaptation.
\newblock In {\em Proceedings of the 40th International Conference on Machine Learning}, ICML'23. JMLR.org.

\bibitem[Ghanbari and Scheinberg, 2018]{ghanbari2018proximal}
Ghanbari, H. and Scheinberg, K. (2018).
\newblock Proximal quasi-{N}ewton methods for regularized convex optimization with linear and accelerated sublinear convergence rates.
\newblock {\em Computational Optimization and Applications}, 69:597--627.

\bibitem[Gorbunov et~al., 2020]{gorbunov2020unified}
Gorbunov, E., Hanzely, F., and Richt{\'a}rik, P. (2020).
\newblock A unified theory of sgd: Variance reduction, sampling, quantization and coordinate descent.
\newblock In {\em International Conference on Artificial Intelligence and Statistics}, pages 680--690. PMLR.

\bibitem[Griewank, 1981]{griewank1981modification}
Griewank, A. (1981).
\newblock The modification of {Newton}’s method for unconstrained optimization by bounding cubic terms.
\newblock Technical report, Technical report NA/12.

\bibitem[Ivgi et~al., 2023]{ivgi2023dog}
Ivgi, M., Hinder, O., and Carmon, Y. (2023).
\newblock {D}o{G} is {SGD}'s best friend: A parameter-free dynamic step size schedule.
\newblock {\em arXiv:2302.12022}.

\bibitem[Jahani and Rusakov, 2022]{jahani2022doubly}
Jahani, M. and Rusakov, S. (2022).
\newblock Doubly adaptive scaled algorithm for machine learning using 2nd order information.
\newblock In {\em International Conference on Learning Representations}.

\bibitem[Jin et~al., 2024]{jin2024exact}
Jin, Q., Jiang, R., and Mokhtari, A. (2024).
\newblock Non-asymptotic global convergence rates of {BFGS} with exact line search.

\bibitem[Jin et~al., 2022]{jin2022sharpened}
Jin, Q., Koppel, A., Rajawat, K., and Mokhtari, A. (2022).
\newblock Sharpened quasi-{N}ewton methods: Faster superlinear rate and larger local convergence neighborhood.
\newblock In {\em International Conference on Machine Learning}, pages 10228--10250. PMLR.

\bibitem[Kamzolov et~al., 2023]{kamzolov2023cubic}
Kamzolov, D., Ziu, K., Agafonov, A., and Tak{\'a}{\v{c}}, M. (2023).
\newblock Cubic regularization is the key! {T}he first accelerated quasi-{N}ewton method with a global convergence rate of $ \mathcal{O} (k^{-2}) $ for convex functions.
\newblock {\em arXiv preprint arXiv:2302.04987}.

\bibitem[Khaled et~al., 2023]{Khaled2023DoWGUA}
Khaled, A., Mishchenko, K., and Jin, C. (2023).
\newblock Dowg unleashed: An efficient universal parameter-free gradient descent method.
\newblock {\em ArXiv}, abs/2305.16284.

\bibitem[Malitsky and Mishchenko, 2020]{malitsky20adaptive}
Malitsky, Y. and Mishchenko, K. (2020).
\newblock Adaptive gradient descent without descent.
\newblock In {\em Proceedings of the 37th International Conference on Machine Learning}, volume 119 of {\em Proceedings of Machine Learning Research}, pages 6702--6712. PMLR.

\bibitem[Malitsky and Mishchenko, 2024]{malitsky2024adaptive}
Malitsky, Y. and Mishchenko, K. (2024).
\newblock Adaptive proximal gradient method for convex optimization.
\newblock {\em Advances in Neural Information Processing Systems}, 37:100670--100697.

\bibitem[Mishchenko, 2023]{mishchenko2023regularized}
Mishchenko, K. (2023).
\newblock Regularized newton method with global convergence.
\newblock {\em SIAM Journal on Optimization}, 33(3):1440--1462.

\bibitem[Mishchenko and Defazio, 2024]{mishchenko2024prodigy}
Mishchenko, K. and Defazio, A. (2024).
\newblock Prodigy: An expeditiously adaptive parameter-free learner.
\newblock In {\em Forty-first International Conference on Machine Learning}.

\bibitem[Mishkin et~al., 2024]{mishkin2024directional}
Mishkin, A., Khaled, A., Wang, Y., Defazio, A., and Gower, R. (2024).
\newblock Directional smoothness and gradient methods: Convergence and adaptivity.
\newblock {\em Advances in Neural Information Processing Systems}, 37:14810--14848.

\bibitem[Nesterov, 2008]{nesterov2008accelerating}
Nesterov, Y. (2008).
\newblock Accelerating the cubic regularization of {N}ewton’s method on convex problems.
\newblock {\em Mathematical Programming}, 112(1):159--181.

\bibitem[Nesterov, 2013]{nesterov2013gradient}
Nesterov, Y. (2013).
\newblock Gradient methods for minimizing composite functions.
\newblock {\em Mathematical Programming}, 140(1):125--161.

\bibitem[Nesterov, 2022]{nesterov2019inexact}
Nesterov, Y. (2022).
\newblock Inexact basic tensor methods for some classes of convex optimization problems.
\newblock {\em Optimization Methods and Software}, 37(3):878--906.

\bibitem[Nesterov and Polyak, 2006]{nesterov2006cubic}
Nesterov, Y. and Polyak, B.~T. (2006).
\newblock Cubic regularization of {Newton} method and its global performance.
\newblock {\em Mathematical Programming}, 108(1):177--205.

\bibitem[Nocedal and Wright, 2006]{nocedal2006numerical}
Nocedal, J. and Wright, S.~J. (2006).
\newblock {\em Numerical optimization}.
\newblock Springer.

\bibitem[Polyak, 1987]{polyak1987introduction}
Polyak, B. (1987).
\newblock {\em Introduction to Optimization}.
\newblock Optimization Software.

\bibitem[Rodomanov and Nesterov, 2021a]{rodomanov2021greedy}
Rodomanov, A. and Nesterov, Y. (2021a).
\newblock Greedy quasi-{N}ewton methods with explicit superlinear convergence.
\newblock {\em SIAM Journal on Optimization}, 31(1):785--811.

\bibitem[Rodomanov and Nesterov, 2021b]{rodomanov2021new}
Rodomanov, A. and Nesterov, Y. (2021b).
\newblock New results on superlinear convergence of classical quasi-{N}ewton methods.
\newblock {\em Journal of Optimization Theory and Applications}, 188:744--769.

\bibitem[Scheinberg and Tang, 2016]{scheinberg2016practical}
Scheinberg, K. and Tang, X. (2016).
\newblock Practical inexact proximal quasi-{N}ewton method with global complexity analysis.
\newblock {\em Mathematical Programming}, 160:495--529.

\bibitem[Scieur, 2024]{scieur2024adaptive}
Scieur, D. (2024).
\newblock Adaptive quasi-{N}ewton and anderson acceleration framework with explicit global (accelerated) convergence rates.
\newblock In {\em International Conference on Artificial Intelligence and Statistics}, pages 883--891. PMLR.

\bibitem[Sen and Mohan, 2023]{sen2023foplahd}
Sen, M. and Mohan, C.~K. (2023).
\newblock Foplahd: Federated optimization using locally approximated hessian diagonal.
\newblock In {\em International Conference on Big Data Analytics}, pages 235--245. Springer.

\bibitem[Smee et~al., 2025]{smee2025first}
Smee, O., Roosta, F., and Wright, S.~J. (2025).
\newblock First-ish order methods: Hessian-aware scalings of gradient descent.
\newblock {\em arXiv preprint arXiv:2502.03701}.

\bibitem[Thomsen and Doikov, 2024]{thomsen2024complexity}
Thomsen, D.~B. and Doikov, N. (2024).
\newblock Complexity of minimizing regularized convex quadratic functions.
\newblock {\em arXiv preprint arXiv:2404.17543}.

\bibitem[Wang et~al., 2024]{wang2024global}
Wang, S., Fadili, J., and Ochs, P. (2024).
\newblock Global non-asymptotic super-linear convergence rates of regularized proximal quasi-{N}ewton methods on non-smooth composite problems.
\newblock {\em arXiv preprint arXiv:2410.11676}.

\bibitem[Yao et~al., 2018]{yao2018hessian}
Yao, Z., Gholami, A., Lei, Q., Keutzer, K., and Mahoney, M.~W. (2018).
\newblock Hessian-based analysis of large batch training and robustness to adversaries.
\newblock {\em Advances in Neural Information Processing Systems}, 31.

\bibitem[Yao et~al., 2021]{yao2021adahessian}
Yao, Z., Gholami, A., Shen, S., Mustafa, M., Keutzer, K., and Mahoney, M. (2021).
\newblock Adahessian: An adaptive second order optimizer for machine learning.
\newblock In {\em proceedings of the AAAI conference on artificial intelligence}, volume~35, pages 10665--10673.

\bibitem[Zhang et~al., 2020a]{zhang2020improved}
Zhang, B., Jin, J., Fang, C., and Wang, L. (2020a).
\newblock Improved analysis of clipping algorithms for non-convex optimization.
\newblock In {\em Advances in Neural Information Processing Systems}.

\bibitem[Zhang et~al., 2020b]{zhang2020why}
Zhang, J., He, T., Sra, S., and Jadbabaie, A. (2020b).
\newblock Why gradient clipping accelerates training: A theoretical justification for adaptivity.
\newblock In {\em International Conference on Learning Representations}.

\end{thebibliography}

\newpage
\appendix

\section{Auxiliary results}

\begin{lemma}[Young’s inequalitys]\label{lem:young-pq}
Let $p,q>1$ with $\tfrac1p+\tfrac1q=1$. For all $a,b\ge0$,
\[
  ab  \le  \frac{a^p}{p} + \frac{b^q}{q}.
\]
\end{lemma}

\begin{lemma}For all $b>0, \xi>0, \beta>0$,
\[
  \frac{b}{\sqrt{\xi}} \le \frac{\beta}{\xi}+\frac{b^2}{4\beta}.
\]
\end{lemma}
\begin{proof}
    Young’s inequality $2xy\le x^2+y^2$ with
    $x=\sqrt{\beta/\xi}$ and $y=b/(2\sqrt{\beta})$ gives
    $ \frac{b}{\sqrt{\xi}}=2xy\le x^2+y^2=\frac{\beta}{\xi}+\frac{b^2}{4\beta}$.
\end{proof}

\begin{lemma}\label{lem:three-halves-split}
For all $a,b\in\R^d$ and any norm $\|\cdot\|$,
\[
  \|a\|^{3/2}  \le  \sqrt{2}\bigl(\|a-b\|^{3/2} + \|b\|^{3/2}\bigr).
\]
Equivalently,
\[
  -\|a-b\|^{3/2}  \le  -\tfrac{1}{\sqrt{2}}\|a\|^{3/2} + \|b\|^{3/2}.
\]
\end{lemma}

\newpage
% Accelerating the cubic regularization of Newton’s method

% Multiplying \eqref{eq:cubicstepopt} by $T-x$ yields
% \begin{equation}
% \langle \nabla f(x),x-T\rangle = \langle \nabla^2 f(x)(T-x),T-x\rangle + \frac{M}{2}r_M(x)^3.
% \tag{3.5}
% \end{equation}
% If $M\ge \widehat{L}^3$, then by (2.12),
% \begin{align}
% f(x)-f(T) \ge f(x)-\widehat f_M(x;T)
% &= \langle \nabla f(x),x-T\rangle - \tfrac12 \langle \nabla^2 f(x)(T-x),T-x\rangle - \tfrac{M}{6}r_M(x)^3 \notag\\
% &= \tfrac12 \langle \nabla^2 f(x)(T-x),T-x\rangle + \tfrac{M}{3}r_M(x)^3.
% \tag{3.6}
% \end{align}
% In particular, by convexity of $f$,
% \begin{equation}
% f(x)-f(T) \ge \frac{M}{3}r_M(x)^3 \ge \frac{M}{3}\Big(\frac{2}{\widehat{L}^3+M}\|\nabla f(T)\|_*\Big)^{3/2}.
% \tag{3.7}
% \end{equation}
% A global view of the step gives
% \begin{equation}
% f(T)   (M\ge \widehat{L}^3)  \le  \min_{y}\Big\{ f_2(x;y)+\tfrac{M}{6}\|y-x\|^3\Big\}
%  \le  \min_{y}\Big\{ f(y)+\tfrac{\widehat{L}^3+M}{6}\|y-x\|^3\Big\}.
% \tag{3.8}
% \end{equation}

\newpage
\section{Missing Proofs for Accelerated CaCuN} \label{sec:cacun ap}

Following~\cite{nesterov2006cubic}, define the mapping
\begin{equation}
    T_M(x) \stackrel{\mathrm{def}}{=} \argmin_{y}\cbr*{  f(x) + \inp{\nabla f(x)}{y-x} + \frac{1}{2}\inp{\nabla^2 f(x)(y-x)}{y-x} + \frac{M}{3}\|y-x\|^3}.
\end{equation}
Then $T=T_M(x)$ is the unique solution of
\begin{equation} \label{eq:cubicstepopt}
    \nabla f(x) + \nabla^2 f(x)(T-x) + M \|T-x\|(T-x) = 0.
\end{equation}
Denote $r_M(x)=\|x-T_M(x)\|$. Using~\eqref{eq:cubicstepopt} and the $2H$-Lipschitz Hessian bound~\cref{eq:grad_dif_hess}, we get
\begin{align} \label{eq:gnormubound}
    \|\nabla f(T)\|
    &= \big\|\nabla f(T) - \nabla f(x) - \nabla^2 f(x)(T-x) - Mr_M(x)(T-x)\big\| \notag
    \\
    &\leq (H+M) r_M(x)^2.
\end{align}

\begin{lemma}\cite[Lemma 4]{nesterov2006cubic}
    \label{lem:anycubic}
    For any $x \in \R^d$
    \begin{equation*}
        T_{H_k}(x^{k}) \leq \min_y\sbr*{f(y) + \frac{H + H_k}{3}\|y-x^k\|^3}
    \end{equation*}
\end{lemma}
\begin{proof}
    By \cref{as:hessian_smooth}    
    \begin{align*}
        f(y) \leq f(x) + \inp{\nabla f(x)}{y-x} + \frac{1}{2}\inp{\nabla^2 f(x)(y-x)}{y-x} + \frac{H}{3}\norm{y-x}^3.
    \end{align*}
    
    The lower bound in the above implies

    \begin{equation}
        \label{res}
        f(x) + \inp{\nabla f(x)}{y-x} + \frac{1}{2}\inp{\nabla^2 f(x)(y-x)}{y-x} \leq f(y) + \frac{H}{3}\|y-x\|^3
    \end{equation}
    
    Let
    
    $$
        T_{H_k}(x^{k}) = \argmin_y \sbr*{f(x^k) + \inp{\nabla f(x^k)}{y-x^k} + \frac{1}{2}\inp{\nabla^2 f(x^k)(y-x^k)}{y-x^k} + \frac{H_k}{3}\|y-x^k\|^3}
    $$
    
    Then adding $\frac{H_k}{3}\|y-x^k\|^3$ to both sides of \eqref{res} and minimizing w.r.t $y$ gives
    
    \begin{multline}
        f(T_{H_k}(x^{k})) = \min_y \sbr*{f(x^k) + \inp{\nabla f(x^k)}{y-x^k} + \frac{1}{2}\inp{\nabla^2 f(x^k)(y-x^k)}{y-x^k} + \frac{M_k}{3}\|y-x^k\|^3}
        \\
        \leq \min_y\sbr*{f(y) + \frac{H + H_k}{3}\|y-x^k\|^3}
    \end{multline}
\end{proof}

\begin{lemma}\cite[Lemma 6]{nesterov2008accelerating}
\label{lem:nesterov6}
    If $M \ge 2H$, then
    \begin{equation}
        \langle \nabla f(T), x - T\rangle \ge \frac{1}{\sqrt{H+M}} \|\nabla f(T)\|^{3/2}.
    \end{equation}
\end{lemma}

\begin{proof}
    Let $T=T_M(x)$ and Denote $r=r_M(x)=\|x-T_M(x)\|$. Using \eqref{eq:grad_dif_hess} and \eqref{eq:cubicstepopt},
    \begin{align*}
        H^2 r^4
        = \rbr*{H\|T-x\|^2}^2
        \geq \|\nabla f(T)-\nabla f(x)-\nabla^2 f(x)(T-x)\|^2
        = \norm*{\nabla f(T)+Mr(T-x)}^2,
    \end{align*}
    hence
    \[
        \|\nabla f(T)\|^2 + 2Mr \langle \nabla f(T),T-x\rangle +  M^2 r^4  \leq H^2 r^4.
    \]
    Rearranging,
    \begin{equation}\label{eq:decresenomin}
        \langle \nabla f(T), x - T\rangle \geq \frac{1}{2Mr} \|\nabla f(T)\|^2 + \frac{M^2 - H^2}{2M} r^3.
    \end{equation}
    For $M \geq 2H$, the derivative in $r$ of the right-hand side of \eqref{eq:decresenomin} is nonnegative by~\eqref{eq:gnormubound}:
    \[
        -\frac{1}{2Mr^2}\|\nabla f(T)\|^2 + \frac{3}{2M}(M^2-H^2)r^2
        \geq -\frac{1}{2Mr^2}\|\nabla f(T)\|^2 + \frac{(H+M)^2}{M}\frac{r^2}{2} \geq 0.
    \]
    Thus, its minimum is attained at the boundary point
    \[
        r = \left( \frac{1}{H+M} \|\nabla f(T)\| \right)^{1/2}.
    \]

\end{proof}

Introduce a sequence of estimate functions:
\begin{equation*}
    \psi_k(x) = \ell_k(x) + \frac{N}{3} \|x - x_0\|^3,\quad k = 1, 2, \dots
\end{equation*}
where $\ell_k(x)$ are linear functions and $N$ is a positive real parameter.

And a sequence of scaling parameters $\{A_k\}_{k=1}^{\infty}$:
\begin{equation*}
    A_{k+1} \coloneqq A_k + a_k\quad k = 1, 2, \dots
\end{equation*}

\begin{lemma}
    Let~\cref{as:convexity,as:bounded-level-set,as:hessian_smooth} hold. Then the iterates generated by~\cref{alg:acccacun} with  $M_0 \leq H$ satisfy the following
    \begin{equation}\label{ca}
        A_k f(x^k) \leq \psi_k^* \equiv \min_x \psi_k(x), \tag{C1}        
    \end{equation}
    \begin{equation}\label{cb}
        \psi_k(x) \leq A_k f(x) + \frac{2H + N}{3} \|x - x_0\|^3,\quad \forall x \in \R^d  \tag{C2}
    \end{equation}
\end{lemma} 
\begin{proof}
    Let us ensure that relations hold for $k = 1$. 

    We choose
    \begin{equation}
        A_1 = 1, \quad \ell_1(x) \equiv f(x^1), x \in \R^d,
    \end{equation}
    so 
    \begin{equation}
        \psi_1(x) = f(x^1) + \frac{N}{3} \|x - x_0\|^3.
    \end{equation}
    Then $\psi_1^* = f(x^1)$, so~\cref{ca} holds for both cases below.

    \begin{leftbar}
        If $\quad f\rbr*{y^0 - \frac{\nabla f(y^0)}
            {\sqrt{M_0\norm{\nabla f(y^0)}}}} \leq f(y^0)-  \frac{\sqrt2}{3\sqrt{M_0}} \norm{\nabla f(y^0)}^{3/2}$for some $M_0 \leq H$    then we do the gradient step $x^{1} = y^0 - \frac{\nabla f(y^0)}{\sqrt{M_0\norm{\nabla f(y^0)}}}$, and have that
        \begin{multline*}
            f(x^{1}) \leq f(y^0) -  \frac{\sqrt2}{3\sqrt{M_0}} \norm{\nabla f(y^0)}^{3/2} \equiv \min_y \sbr*{f(y^0) + \inp{\nabla f(y^0)}{y-y^0} + \frac{2M_0}{3}\|y-y^0\|^3}
            \\
            \leq \min_y \sbr*{f(y) + \frac{2H}{3}\|y-y^0\|^3}.
        \end{multline*}
    
    \end{leftbar}

    \begin{leftbar}
        Else we perform the \algname{CRN} step $x^1 = T_{2H}(x^0)$, and by \cref{lem:anycubic} we have
        \begin{equation*}
            f(x^1)  \leq \min_y \sbr*{f(x) + \frac{2H}{3} \|x^1 - x^0\|^3}.
        \end{equation*}
    \end{leftbar}

    Thus in both cases
    \begin{multline}
        \psi_1(x) = f(x^1) + \frac{N}{3} \|x - x_0\|^3 \leq \min_{y} \left[ f(y) + \frac{2H}{3} \|y - x_0\|^3 \right] + \frac{N}{3} \|x - x_0\|^3
        \\
        \leq f(x) + \frac{2H+N}{3} \|x - x_0\|^3,
    \end{multline}
    and~\eqref{cb} follows.

    Assume now that relations~\eqref{ca},\eqref{cb} hold for some $k \geq 1$. 
    
    Denote
    \begin{equation*}
        \psi_k(x) \equiv \ell_k(x) + \frac{N}{3} \|x - x_0\|^3 
        % \\
        \equiv \ell_k(x) + N d(x) 
        % \\
        % &\geq& \min_x \sbr*{\ell_k(x) + \frac{N}{3} \|x - x_0\|^3}
        % \\
        % &\geq& \psi_k^* + \frac{N}{2} \cdot \frac{1}{6} \|x - v_k\|^3 
        % \\
        % &\geq& A_k f(x_k) + \frac{N}{2} \cdot \frac{1}{6} \|x - v_k\|^3.
    \end{equation*}
    and
    \begin{equation*}
        v_k = \arg\min_x \psi_k(x).
    \end{equation*}

    Then the optimality condition reads as
    \begin{equation*}
        \nabla \ell_k(v^k) + N \nabla d(v^k) = 0.
    \end{equation*}

    Next, applying \cite[Lemma~4]{nesterov2008accelerating} with $p = 3$, for any $x \in \R^d$, we have
    \[
        d(x) - d(v^k) - \langle \nabla d(v^k), x - v^k \rangle 
        \geq \frac{1}{6} \|x - v^k\|^3,
    \]
    and by convexity
    \begin{equation*}
        \ell_k(x) \geq \inp{\nabla \ell_k(v^k)}{x - v^
        k}.
    \end{equation*}

    The latter implies
    \begin{eqnarray*}
        \psi_k(x) &\geq& \inp{\nabla \ell_k(v^k)}{x - v^
        k} + Nd(v^k) + \inp{N\nabla d(v^k)}{x - v^k} + \frac{N}{6} \|x - v^k\|^3
        \\
        &=& \psi_k(v^k) + \frac{N}{6} \|x - v^k\|^3 \equiv  \psi_k^* + \frac{N}{6} \|x - v^k\|^3 
        \\
        &\geq& A_k f(x^k) + \frac{N}{6} \|x - v^k\|^3,
    \end{eqnarray*}
    where the last follows from~\eqref{ca} for $k$.

    Let us choose some $a_k > 0$. Define
    \begin{align}\label{eq:comb}
        \alpha_k &= \frac{a_k}{A_k + a_k}, \notag
        \\
        y_k &= (1 - \alpha_k) x_k + \alpha_k v_k.
    \end{align}

    \begin{leftbar}
        If $f\rbr*{y^k - \frac{\nabla f(y^k)}
            {\sqrt{M_k\norm{\nabla f(y^k)}}}} \leq f(y^k)-  \frac{1}{\sqrt{3M_k}} \norm{\nabla f(y^k)}^{3/2}$ with $M_k \le H$ then the gradient step is performed
        \begin{eqnarray*}
            x_{k+1} &=& y^k - \frac{\nabla f(y^k)}{\sqrt{M_k\norm{\nabla f(y^k)}}}
            \\
            \psi_{k+1}(x) &=& \psi_k(x) + a_k \big[ f(y_{k}) + \langle \nabla f(y_{k}), x - y_{k} \rangle \big].
        \end{eqnarray*}
        In view of~\eqref{cb} for $k$, for any $x \in \R^d$ we have
        \begin{align*}
            \psi_{k+1}(x) &= \psi_k(x) + a_k \big[ f(y_{k}) + \langle \nabla f(y_{k}), x - y_{k} \rangle \big]
            \\
            &\leq A_k f(x) + \frac{2H + N}{3} \|x - x_0\|^3
            + a_k \big[ f(y_{k}) + \langle \nabla f(y_{k}), x - y_{k} \rangle \big]
            \\
            &\leq (A_k + a_k) f(x) + \frac{2H + N}{3} \|x - x_0\|^3,
        \end{align*}
        where the last inequality follows from convexity and implies~\eqref{cb} for $k+1$.
        
        Next,
        \begin{eqnarray*}
            \lefteqn{\psi_{k+1}^* = \min_x \left\{ \psi_k(x) + a_k \left[ f(y^k) + \langle \nabla f(y^k), x - y^k \rangle \right] \right\}}
            \\
            &\geq& \min_x \left\{ A_k f(x^k) + \frac{N}{6} \|x - v^k\|^3
            + a_k \left[ f(y^k) + \langle \nabla f(y^k), x - y^k \rangle \right] \right\} 
            \\
            &\overset{\eqref{eq:convexity}}{\geq}& \min_x \left\{ A_k\rbr{ f(y^k) + \inp{\nabla f(y^k)}{x^k - y^k}} + \frac{N}{6} \|x - v^k\|^3
            + a_k \left[ f(y^k) + \langle \nabla f(y^k), x - y^k \rangle \right] \right\} 
            \\
            &=&\min_x \left\{ (A_k + a_k) f(y^{k}) 
            + A_k \langle \nabla f(y^{k}), x^k - y^{k} \rangle 
            + a_k \langle \nabla f(y^k), x - y^k \rangle
            + \frac{N}{6} \|x - v^k\|^3 \right\}
            \\
            &=& A_{k+1} f(y^{k}) 
            + \langle \nabla f(y^{k}), A_{k} (x^k - y^k) + a_k (v^k - y^{k}) \rangle 
            \\
            && + \min_x \left\{ a_k \langle \nabla f(y^k), x - v^k \rangle 
            + \frac{N}{6} \|x - v^k\|^3 \right\} 
            \\
            &\overset{\eqref{eq:comb}}{=}& \min_x \left\{ A_{k+1} f(x^{k+1})  +  \frac{A_{k+1}}{\sqrt{3H}} \norm{\nabla f(y^k)}^{3/2} 
            +  a_k \langle \nabla f(y^k), x - v^k \rangle 
            + \frac{N}{6} \|x - v^k\|^3 \right\}.
        \end{eqnarray*}
    
        Hence, our choice of parameters must ensure the following inequality:
        \[
            \frac{A_{k+1}}{\sqrt{3H}} \norm{\nabla f(y^k)}^{3/2} 
            +  a_k \langle \nabla f(y^k), x - v^k \rangle 
            + \frac{N}{6} \|x - v^k\|^3 \geq 0
        \]
         Using \cite[Lemma 2]{nesterov2008accelerating} with $p = 3$, $s = a_k \nabla f(y^{k})$, and 
        $\sigma = \frac12 N$, we come to the following condition:
        \begin{equation}
            \label{eq:akgrad}
            A_{k+1} \sqrt{\frac{1}{3H}} \geq \frac{2}{3} \sqrt{\frac{2}{N}}  a_k^{3/2}.
        \end{equation}
    \end{leftbar}

    \begin{leftbar}
        Else we perform \algname{CRN} step
        \begin{eqnarray*}
            x_{k+1} &=& T_{2H}(y_k)
            \\
            \psi_{k+1}(x) &=& \psi_k(x) + a_k \big[ f(x_{k+1}) + \langle \nabla f(x_{k+1}), x - x_{k+1} \rangle \big].
        \end{eqnarray*}
        In view of~\eqref{cb} for $k$, for any $x \in \R^d$ we have
        \begin{align*}
            \psi_{k+1}(x) &= \psi_k(x) + a_k \big[ f(x_{k+1}) + \langle \nabla f(x_{k+1}), x - x_{k+1} \rangle \big]
            \\
            &\leq A_k f(x) + \frac{2H + N}{3} \|x - x_0\|^3
            + a_k \big[ f(x_{k+1}) + \langle \nabla f(x_{k+1}), x - x_{k+1} \rangle \big] 
            \\
            &\leq (A_k + a_k) f(x) + \frac{2H + N}{3} \|x - x_0\|^3,
        \end{align*}
        we the last inequality follows from convexity and implies~\eqref{cb} for $k+1$.

        Next, 
        \begin{eqnarray*}
            \lefteqn{\psi_{k+1}^* = \min_x \left\{ \psi_k(x) + a_k \left[ f(x^{k+1}) + \langle \nabla f(x^{k+1}), x - x^{k+1} \rangle \right] \right\}}
            \\
            &\geq& \min_x \left\{ A_k f(x^k) + \frac{N}{6} \|x - v^k\|^3
            + a_k \left[ f(x^{k+1}) + \langle \nabla f(x^{k+1}), x - x^{k+1} \rangle \right] \right\} 
            \\
            &\geq& \min_x \left\{ (A_k + a_k) f(x^{k+1}) 
            + A_k \langle \nabla f(x^{k+1}), x^k - x^{k+1} \rangle 
            + a_k \langle \nabla f(x^{k+1}), x - x^{k+1} \rangle 
            + \frac{N}{6} \|x - v^k\|^3 \right\}
            \\
            &=& A_{k+1} f(x^{k+1}) 
            + \langle \nabla f(x^{k+1}), A_{k+1} y^k - a_k v^k - A_k x^{k+1} \rangle 
            \\
            && + \min_x \left\{ a_k \langle \nabla f(x^{k+1}), x - x^{k+1} \rangle 
            + \frac{N}{6} \|x - v^k\|^3 \right\} 
            \\
            &=& \min_x \left\{ A_{k+1} f(x^{k+1}) 
            + A_{k+1} \langle \nabla f(x^{k+1}), y^k - x^{k+1} \rangle  + a_k \langle \nabla f(x^{k+1}), x - v^k \rangle 
            + \frac{N}{6} \|x - v^k\|^3 \right\}.
        \end{eqnarray*}

        Further, by \cref{lem:nesterov6} we have
        \[
            \langle \nabla f(x^{k+1}), y^k - x^{k+1} \rangle 
            \geq \sqrt{\frac{1}{M+H}}  \|\nabla f(x^{k+1})\|^{3/2}.
        \]
        Hence, our choice of parameters must ensure the following inequality:
        \[
            A_{k+1} \sqrt{\frac{1}{3H}} \|\nabla f(x^{k+1})\|^{3/2}
            + a_k \langle \nabla f(x^{k+1}), x - v^k \rangle
            + \frac{N}{6} \|x - v^k\|^3 \geq 0, \quad \forall x \in \R^d.
        \]

        Using \cite[Lemma 2]{nesterov2008accelerating} with $p = 3$, $s = a_k \nabla f(x^{k+1})$, and 
        $\sigma = \frac12 N$, we come to the following condition:
        \begin{equation}\label{eq:akcubic}
               A_{k+1} \sqrt{\frac{1}{3H}} \geq \frac{2}{3} \sqrt{\frac{2}{N}}  a_k^{3/2}.
        \end{equation}
    \end{leftbar}

    In both cases one can satisfy the conditions by choosing
    \[
        A_k = \frac{k(k + 1)(k + 2)}{6},
    \]
    \[
        a_k = A_{k+1} - A_k 
        = \frac{(k + 1)(k + 2)(k + 3)}{6} - \frac{k(k + 1)(k + 2)}{6}
        = \frac{(k + 1)(k + 2)}{2}
        \]
    for $k \geq 1$.
    Since
        \[
        a_k^{-3/2} A_{k+1} 
        = 2^{3/2} \frac{(k + 1)(k + 2)(k + 3)}{6  [(k + 1)(k + 2)]^{3/2}}
        = 2^{1/2} \frac{k + 3}{3  [(k + 1)(k + 2)]^{1/2}} \geq \frac{\sqrt{2}}{3},
    \]
    inequalities~\eqref{eq:akgrad} and~\eqref{eq:akcubic} lead to the following condition:
    \[
        \frac{1}{\sqrt{3H}} \geq \frac{2}{\sqrt{N}}.
    \]
    
    Hence one can chose
    \begin{equation*}
        N=12H.
    \end{equation*}

\end{proof}

\begin{theorem}
    Let~\cref{as:convexity,as:bounded-level-set,as:hessian_smooth} hold. Then the iterates generated by~\cref{alg:acccacun} with  $M_0 \leq H$ satisfy the following
    \[
        f(x^k) - f(x^*) \leq \frac{14 \cdot 2H \|x^0 - x^*\|^3}{k(k + 1)(k + 2)}.
    \]
\end{theorem}
\begin{proof}
    Indeed, we have shown that
    \[
        A_k f(x^k) \overset{\eqref{ca}}{\leq} \psi_k^* \overset{\eqref{cb}}{\leq} A_k f(x^*) 
        + \frac{2H + N}{3} \|x^0 - x^*\|^3.
    \]
\end{proof}

\begin{remark}
    Note that the point $v^k$ can be found in (4.8) by an explicit formula.  
    Consider
    \[
        s^k = \nabla \ell_k(x).
    \]
    This vector does not depend on $x$ since the function $\ell_k(x)$ is linear.  
    Then
    \[
        v^k = x^0 - \frac{s^k}{\sqrt{N \|s^k\|}}.
    \]
\end{remark}

\newpage

\section{Missing Proofs for First Order Results}
\subsection{Examples}

\begin{example}\label{prop:iso-cubic}
Let $f(x)=\frac{H}{3}\|x\|^3$ with $H>0$. Then $\nabla^2 f$ is $2H$-Lipschitz and \eqref{eq:accclass} holds with $C=4H$.
\end{example}
\begin{proof}
Write $r=\|x\|$ and $e=x/r$ for $x\neq 0$. Then
\[
\nabla f(x)=H r x,\qquad
\nabla^2 f(x)=H\Big(r I+\frac{xx^\top}{r}\Big)=H r (I+ee^\top).
\]
(At $x=0$, set $\nabla^2 f(0)=0$, which is the continuous extension since $\|xx^\top/r\|=r\to0$.)

\emph{Verification of \eqref{eq:accclass}.}
Using the formulas above,
\[
\nabla^2 f(x) \nabla f(x)=H\Big(r I+\frac{xx^\top}{r}\Big)(H r x)
=2H^2 r^2 x,
\]
hence
\[
\frac{\langle \nabla f(x), \nabla^2 f(x) \nabla f(x)\rangle}{\|\nabla f(x)\|^2}
=\frac{(H r x)^\top(2H^2 r^2 x)}{\|H r x\|^2}
=2H r.
\]
Since $\|\nabla f(x)\|=H r^2$, we get
\[
2H r  =  2\sqrt{H} \sqrt{\|\nabla f(x)\|}
 =  \sqrt{(4H) \|\nabla f(x)\|},
\]
so \eqref{eq:accclass} holds with $C=4H$ (with equality for $x\neq 0$).

\textbf{Lipschitzness of the Hessian.} Fix a unit vector $u$ and define
\[
\phi(t):=u^\top\nabla^2 f\big(y+t(x-y)\big)u,\quad t\in[0,1].
\]
Let $z(t)=y+t(x-y)$, $r=\|z\|$, $\alpha=\langle u,z\rangle$, and $c=\alpha/r$ (define by continuity at $r=0$). Then
\[
\phi(t)=H\left(r+\frac{\alpha^2}{r}\right),\quad
\phi'(t)=H\Big(r'(1-c^2)+2c \alpha'\Big),
\]
with $r'=\langle z/r, x-y\rangle$ and $\alpha'=\langle u, x-y\rangle$. Thus $|r'|\le\|x-y\|$, $|\alpha'|\le\|x-y\|$, and $|c|\le 1$, giving
\[
|\phi'(t)|\le H\big((1-c^2)+2|c|\big) \|x-y\|\le 2H \|x-y\|.
\]
By the mean value theorem,
\[
|u^\top(\nabla^2 f(x)-\nabla^2 f(y))u|
=|\phi(1)-\phi(0)|
\le \int_0^1 |\phi'(t)| dt
\le 2H \|x-y\|.
\]
Taking the supremum over unit $u$ yields
$\|\nabla^2 f(x)-\nabla^2 f(y)\|\le 2H \|x-y\|$.
\end{proof}

\begin{example}\label{prop:sep-cubic}
Let $f(x)=\sum_{i=1}^d \frac{H_i}{3}|x_i|^3$ with $H_i>0$ and $H_{\max}=\max_iH_i$.
Then $\nabla^2 f$ is $2H_{\max}$-Lipschitz and \eqref{eq:accclass} holds with $C=4H_{\max}$.
\end{example}
\begin{proof}
\emph{Formulas.} $\displaystyle \nabla f(x)=(H_i x_i|x_i|)_i,\qquad
\nabla^2 f(x)=\mathrm{diag}(2H_i|x_i|).$

\emph{Inequality \eqref{eq:accclass}.} Since $\nabla^2 f(x)$ is diagonal with entries $2H_i|x_i|$,
\[
\frac{\langle \nabla f(x), \nabla^2 f(x) \nabla f(x)\rangle}{\|\nabla f(x)\|^2}
=\frac{\sum_i 2H_i|x_i| (H_i x_i|x_i|)^2}{\sum_i (H_i x_i|x_i|)^2}
\le \max_i 2H_i|x_i|.
\]
Let $\alpha_i:=H_i x_i^2\ge 0$. Then $\|\nabla f(x)\|=(\sum_i \alpha_i^2)^{1/2}$ and
\[
\max_i 2H_i|x_i|
=2\max_i \sqrt{H_i} \sqrt{\alpha_i}
\le 2\sqrt{H_{\max}} \max_i \sqrt{\alpha_i}
\le 2\sqrt{H_{\max}} (\sum_i \alpha_i^2)^{1/4}
=2\sqrt{H_{\max}} \|\nabla f(x)\|^{1/2}.
\]
Thus \eqref{eq:accclass} holds with $C=4H_{\max}$.

\textbf{Lipschitzness of the Hessian.} For any $x,y$,
\[
\|\nabla^2 f(x)-\nabla^2 f(y)\|
=\Big\|\mathrm{diag}\big(2H_i(|x_i|-|y_i|)\big)\Big\|
\le \max_i 2H_i |x_i-y_i|
\le 2H_{\max} \|x-y\|.
\]
\end{proof}

\begin{example}[Logistic function]\label{ex:logistic}
Let $a\in\mathbb{R}^d\setminus\{0\}$ and $f(x)=\log\big(1+e^{-a^\top x}\big)$. Then $\nabla^2 f$ is $\frac{\|a\|^3}{6\sqrt{3}}$-Lipschitz and \eqref{eq:accclass} holds with $C=\frac{4}{27}\|a\|^3$.
\end{example}

\begin{proof}
Set $s=a^\top x$ and $\sigma(t)=\frac{1}{1+e^{-t}}$. Using $\sigma(s)+\sigma(-s)=1$,
\[
    \nabla f(x)
    =\frac{-e^{-s}}{1+e^{-s}} a
    =-\sigma(-s) a,
\]
and
\[
    \nabla^2 f(x)=\sigma(s)\sigma(-s) aa^\top.
\]

\emph{Inequality \eqref{eq:accclass}.}
We have
\[
\frac{\langle \nabla f,\nabla^2 f \nabla f\rangle}{\|\nabla f\|^2}
=\sigma(s)\sigma(-s) \|a\|^2,\qquad
\|\nabla f(x)\|=\sigma(-s) \|a\|.
\]
Thus \eqref{eq:accclass} is equivalent to
\[
\sigma(s)\sigma(-s) \|a\|^2 \le \sqrt{C \sigma(-s) \|a\|}
 \quad \text{or}\quad
C \ge \sigma(s)^2\sigma(-s) \|a\|^3.
\]
Let $p=\sigma(s)\in(0,1)$. Maximizing $p^2(1-p)$ on $[0,1]$ gives $\max_p p^2(1-p) = 4/27$ at $p=2/3$, hence $C=\frac{4}{27}\|a\|^3$.

\textbf{Lipschitzness of the Hessian.}  With $\phi(s)=\sigma(s)\sigma(-s)$ one obtains
\[
    \|\nabla^2 f(x)-\nabla^2 f(y)\|
    =\|a\|^2 |\phi(a^\top x)-\phi(a^\top y)|
    \le \|a\|^2 \big(\sup_s|\phi'(s)|\big) |a^\top(x-y)|
    \le \frac{\|a\|^3}{6\sqrt{3}} \|x-y\|,
\]
since $\phi'(s)=\sigma(s)\sigma(-s) (1-2\sigma(s))$ and
$\max_{p\in[0,1]} p(1-p)|1-2p|=1/(6\sqrt{3})$, while $|a^\top(x-y)|\le \|a\| \|x-y\|$.
\end{proof}

\subsection{Missing Proofs for CaCuAdGD}

\subsection{Lipschitz case}

\begin{lemma}\label{lem:two-regime-ap}
    Let~\cref{as:convexity,as:bounded-level-set,as:Lsmooth,as:hessian_smooth} hold. Then the iterates generated by~\cref{alg:cacuadgd} with $\alpha=0.5$ satisfy the following
    
    \begin{eqnarray}
        &\text{ (i) }& \bigl(f(x^k)-f_\star\bigr)^{1/2}\ge \frac{3}{2}\frac{LD^2}{\sqrt{HD^3}}, \quad \text{then}\quad  f(x^{k+1}) - f_\star  \le  f(x^k) - f_\star - \frac{1}{3\sqrt{HD^3}}\bigl(f(x^k)-f_\star\bigr)^{3/2} \label{eq first stage ap}
        \\
        &\text{(ii) }& \bigl(f(x^k)-f_\star\bigr)^{1/2}\le \frac{3}{2}\frac{LD^2}{\sqrt{HD^3}}, \quad \text{then}\quad  f(x^{k+1}) - f_\star  \le  f(x^k) - f_\star - \frac{2}{9LD^2}\bigl(f(x^k)-f_\star\bigr)^{2}.  \label{eq second stage ap}
    \end{eqnarray}
\end{lemma}
\begin{proof}
By inequalities~\eqref{eq:condition} and~\eqref{eq:certificate} either
\[
f(x^{k+1}) \le f(x^k) - \frac{1}{3\sqrt{H_k}}\|\nabla f(x^k)\|^{3/2}\quad(H_k\le H),
\]
or
\[
f(x^{k+1}) \le f(x^k) - \frac{2}{9L_k}\|\nabla f(x^k)\|^{2}\quad(L_k\le L).
\]

By convexity,
\[
f(x^k)-f_\star \le \langle \nabla f(x^k),x^k-x^\star\rangle \le \|\nabla f(x^k)\| \|x^k-x^\star\|\le D\|\nabla f(x^k)\|.
\]
Hence $\|\nabla f(x^k)\|\ge \bigl(f(x^k)-f_\star\bigr)/D$. Using $H_k\le H$ and $L_k\le L$,
\[
f(x^{k+1}) \le f(x^k) - \frac{1}{3\sqrt{H}}\Bigl(\frac{f(x^k)-f_\star}{D}\Bigr)^{3/2}
= f(x^k) - \frac{1}{3\sqrt{HD^3}}\bigl(f(x^k)-f_\star\bigr)^{3/2},
\]
\[
f(x^{k+1}) \le f(x^k) - \frac{2}{9L}\Bigl(\frac{f(x^k)-f_\star}{D}\Bigr)^{2}
= f(x^k) - \frac{2}{9LD^2}\bigl(f(x^k)-f_\star\bigr)^{2}.
\]
The threshold $\bigl(f(x^k)-f_\star\bigr)^{1/2}= \frac{3}{2}\frac{LD^2}{\sqrt{HD^3}}$ is the unique solution of
\[
\frac{2}{9LD^2}\bigl(f(x^k)-f_\star\bigr)^2
= \frac{1}{3\sqrt{HD^3}}\bigl(f(x^k)-f_\star\bigr)^{3/2},
\]
which determines which of the two decreases is smaller; the stated cases follow, as does the definition of $K_1$.

\textbf{Detailed derivation.}
If $(f(x^k) - f_\star)^{1/2} \geq \frac32 \frac{LD^2}{\sqrt{HD^3}}$ then
\begin{multline*}
    f(x^{k+1}) \leq f(x^k) -  \frac{2}{9L_k} \norm{\nabla f(x^k)}^{2} \leq f(x^k) -  \frac{2}{9L} \norm{\nabla f(x^k)}^{2} \leq f(x^k) -  \frac{2}{9LD^2} \rbr*{f(x^k) - f_\star}^{2} 
    \\
    \leq f(x^k) -  \frac{1}{3\sqrt{HD^3}} \rbr*{f(x^k) - f_\star}^{3/2}
\end{multline*}
coincides with
\begin{equation*}
    f(x^{k+1}) \leq f(x^k) -  \frac{1}{3\sqrt{H_k}} \norm{\nabla f(x^k)}^{3/2} \leq f(x^k) -  \frac{1}{3\sqrt{H}} \norm{\nabla f(x^k)}^{3/2} \leq f(x^k) -  \frac{1}{3\sqrt{HD^3}} \rbr*{f(x^k) - f_\star}^{3/2}.
\end{equation*}
This implies that
\begin{equation*}
    f(x^{k+1}) \leq f(x^k) -  \frac{1}{3\sqrt{HD^3}} \rbr*{f(x^k) - f_\star}^{3/2}
\end{equation*}
for any value of $M_k$.

Similarly if $(f(x^k) - f_\star)^{1/2} \leq \frac32 \frac{LD^2}{\sqrt{HD^3}}$ then
\begin{equation*}
    f(x^{k+1}) \leq f(x^k) -  \frac{2}{9L_k} \norm{\nabla f(x^k)}^{2} \leq f(x^k) -  \frac{2}{9L} \norm{\nabla f(x^k)}^{2} \leq f(x^k) -  \frac{2}{9LD^2} \rbr*{f(x^k) - f_\star}^{2} 
\end{equation*}
and 
\begin{multline*}
    f(x^{k+1}) \leq f(x^k) -  \frac{1}{3\sqrt{H_k}} \norm{\nabla f(x^k)}^{3/2} \leq f(x^k) -  \frac{1}{3\sqrt{H}} \norm{\nabla f(x^k)}^{3/2} \leq f(x^k) -  \frac{1}{3\sqrt{HD^3}} \rbr*{f(x^k) - f_\star}^{3/2}
    \\
    \leq f(x^k) -  \frac{2}{9LD^2} \rbr*{f(x^k) - f_\star}^{2} .
\end{multline*}
Thus for any $M_k$
\begin{equation*}\label{eq:non acc regime}
    f(x^{k+1})
    \leq f(x^k) -  \frac{2}{9LD^2} \rbr*{f(x^k) - f_\star}^{2} .
\end{equation*}
\end{proof}

\begin{corollary}
     Let~\cref{as:convexity,as:bounded-level-set,as:Lsmooth,as:hessian_smooth} hold. Then the iterates generated by~\cref{alg:cacuadgd} with $\alpha=0.5$ satisfy $\forall k:\bigl(f(x^k)-f_\star\bigr)^{1/2} \ge \frac{3}{2}\frac{LD^2}{\sqrt{HD^3}}$ the following
    % \begin{equation*}
    %     f(x^{k+1}) - f_\star \leq \frac{81HD^3}{k^2} \rbr*{1+\rbr*{\frac{f(x^0) - f_\star}{9HD^3}}^{1/2}}^2.
    % \end{equation*}
    \begin{equation*}
        f(x^{k+1}) - f_\star \leq \frac{\rbr*{9\sqrt{HD^3}+3\sqrt{f(x^0) - f_\star}}^2}{k^2} .
    \end{equation*}
\end{corollary}
\begin{proof}
    In \cite[Lemma A.1]{nesterov2019inexact} it is shown that if a non–negative sequence $\{\xi_t\}$ satisfies for $\beta > 0$
    \begin{equation*}
        \xi_k-\xi_{k+1}\geq\xi_{k+1}^{1+\beta}
    \end{equation*}
    then for all $k\geq 0$
    \begin{equation}
        \label{eq:nesterov_sequence_trick_1}\xi_k\leq\Bigl[\bigl(1+\tfrac1\beta\bigr)\bigl(1+\xi_0^{\beta}\bigr)\tfrac1k\Bigr]^{1/\beta}.
    \end{equation}

    With $\xi_k = \frac{f(x^k) - f_\star}{9HD^3}$, $\beta = 0.5$ Inequality~\eqref{eq first stage ap} implies
    \begin{equation*}
        f(x^{k+1}) - f_\star \leq \frac{81HD^3}{k^2} \rbr*{1+\rbr*{\frac{f(x^0) - f_\star}{9HD^3}}^{1/2}}^2.
    \end{equation*}
    % for all $k<K_1$.
    % and with $\xi_k = \frac{2(f(x^k) - f_\star)}{9LD^2}$ and $\beta=1$ inequality~\eqref{eq:non acc regime} implies
    % \begin{equation*}
    %     f(x^{k+1}) - f_\star \leq \frac{9LD^2}{k - K_1}  +\frac{f(x^{K_1+1}) - f_\star}{k - K_1}, \quad k \ge K_1+1
    % \end{equation*}
\end{proof}

\begin{theorem}\label{thm lsmoothglobal ap}
    Let~\cref{as:convexity,as:bounded-level-set,as:Lsmooth,as:hessian_smooth} hold. Then the iterates generated by~\cref{alg:cacuadgd} with $\alpha=0.5$ satisfy for all $k\ge0$ the following
    \begin{equation}\label{eq:global-noab-beta}
      f_k - f_*
       \le 
      \frac{9 LD^2}{ k  +  LD^2 \Phi(\xi_0)  -  \dfrac{2HD}{L} },
    \end{equation}
    where $\Phi(\xi) \coloneqq  \frac{9}{2} \frac{1}{\xi}
       + 
      \frac{6\sqrt{H}}{L\sqrt{D}} \frac{1}{\sqrt{\xi}}$.
    
    Or when the algorithm reaches non-accelerated regime ($(f(x^k) - f_\star)^{1/2} \leq \frac32 \frac{LD^2}{\sqrt{HD^3}}$)
    \begin{equation*}
        \xi_k
       \le 
      \frac{27}{2} \frac{LD^2}{ k+LD^2 \Phi(\xi_0) }.
    \end{equation*}
    
\end{theorem}

\begin{proof}
Let $\xi_k \coloneqq f(x^k)-f_\star$. The inequalities \eqref{eq first stage ap} and \eqref{eq second stage ap} imply
\begin{equation*}
  \xi_k-\xi_{k+1}
   > 
  0,
\end{equation*}
or $x^k$ is a solution.

For any $u>v>0$,
\[
  \frac{1}{v}-\frac{1}{u} \ge \frac{u-v}{u^2},
  \qquad
  \frac{1}{\sqrt{v}}-\frac{1}{\sqrt{u}} \ge \frac{u-v}{2u^{3/2}}.
\]
With $u=\xi_k$, $v=\xi_{k+1}$,
\[
  \Phi(\xi_{k+1})-\Phi(\xi_k)
   \ge 
  \frac{9}{2} \frac{\xi_k-\xi_{k+1}}{\xi_k^{2}}
   + 
  \frac{6\sqrt{H}}{L\sqrt{D}}\cdot\frac{\xi_k-\xi_{k+1}}{2 \xi_k^{3/2}}.
\]

Consider the two cases:

\emph{(i) } if  $\xi_k-\xi_{k+1}\ge \frac{2}{9LD^2}\xi_k^2$, then
\[
  \Phi(\xi_{k+1})-\Phi(\xi_k) \ge \frac{9}{2}\cdot\frac{2}{9LD^2} = \frac{1}{LD^2}.
\]

\emph{(ii) } if $\xi_k-\xi_{k+1}\ge \frac{1}{3\sqrt{HD^{3}}}\xi_k^{3/2}$ , then
\[
  \Phi(\xi_{k+1})-\Phi(\xi_k) \ge 
  \frac{6\sqrt{H}}{L\sqrt{D}}\cdot\frac{1}{2}\cdot\frac{1}{3\sqrt{H}D^{3/2}}
   = \frac{1}{LD^2}.
\]

So in all cases,
\begin{equation}\label{eq:tele-noab}
  \Phi(\xi_k) \ge \Phi(\xi_0) + \frac{k}{LD^2}.
\end{equation}

To turn \eqref{eq:tele-noab} into an upper bound on $\xi_k$, use for any $\beta>0$ and any $\xi>0$:
\begin{equation}\label{eq:any-beta}
  \frac{6\sqrt{H}}{L\sqrt{D}}\cdot\frac{1}{\sqrt{\xi}}
   \le 
  \frac{\beta}{\xi}
   + 
  \frac{1}{4\beta}\Bigl(\frac{6\sqrt{H}}{L\sqrt{D}}\Bigr)^{2}
   = 
  \frac{\beta}{\xi} + \frac{9H}{\beta L^2 D}.
\end{equation}
Hence
\[
  \Phi(\xi_k)
   = 
  \frac{9}{2} \frac{1}{\xi_k}  +  \frac{6\sqrt{H}}{L\sqrt{D}} \frac{1}{\sqrt{\xi_k}}
   \le 
  \frac{\tfrac{9}{2}+\beta}{\xi_k}  +  \frac{9H}{\beta L^2 D}.
\]
Combine with \eqref{eq:tele-noab} and rearrange:
\[
  \Phi(\xi_0) + \frac{k}{LD^2}
   \le 
  \frac{\tfrac{9}{2}+\beta}{\xi_k} + \frac{9H}{\beta L^2 D}
  \quad\text{or}\quad
  \xi_k
   \le 
  \frac{\bigl(\tfrac{9}{2}+\beta\bigr)}{ \Phi(\xi_0) + \frac{k}{LD^2} - \frac{9H}{\beta L^2 D} },
\]
which is \eqref{eq:global-noab-beta} (with $\beta=9/2$).

Assume we are in the non-accelerated region, i.e.
\[
  \xi_k  \le  \frac{9}{4} \frac{L^2 D}{H}
  \quad\text{or}\quad
  \frac{H}{L^2 D}  \le  \frac{9}{4} \frac{1}{\xi_k}.
\]
Taking square roots (all quantities are positive),
\[
  \frac{\sqrt{H}}{L\sqrt{D}}
   \le 
  \frac{3}{2} \frac{1}{\sqrt{\xi_k}}.
\]
Hence, we get for the potential
\[
  \Phi(\xi_k)
   = 
  \frac{9}{2} \frac{1}{\xi_k}
   + 
  \frac{6\sqrt{H}}{L\sqrt{D}} \frac{1}{\sqrt{\xi_k}}
   \le 
  \frac{9}{2} \frac{1}{\xi_k}
   + 
  6\cdot\frac{3}{2}\cdot\frac{1}{\xi_k}
   = 
  \frac{27}{2} \frac{1}{\xi_k}.
\]
Combining with the telescoping one has
\[
  \Phi(\xi_k) \ge \Phi(\xi_0) + \frac{k}{LD^2},
\]
and finally obtains
\[
  \Phi(\xi_0)+\frac{k}{LD^2}
   \le 
  \Phi(\xi_k)
   \le 
  \frac{27}{2} \frac{1}{\xi_k}
  \quad\text{or}\quad
  \xi_k
   \le 
  \frac{27}{2} \frac{LD^2}{ k+LD^2 \Phi(\xi_0) }.
\]

\end{proof}

\subsection{Relaxed smoothness}

\begin{lemma}[Scalar logistic bound]\label{lem:scalar-tight-enough}
For all $s\in\R$,
\[
  \sigma(s) \sigma(-s) \le \frac{2}{3\sqrt{3}} \sqrt{\log\bigl(1+e^{-s}\bigr)},
\]
where  $\sigma(s)=\frac{1}{1+e^{-s}}$
\end{lemma}

\begin{proof}
Let $t=e^{-s}>0$. Then
\[
  \sigma(s)\sigma(-s)
   = \frac{t}{(1+t)^2},
  \qquad
  \log\bigl(1+e^{-s}\bigr)=\log(1+t).
\]
We use the elementary inequality
\[
  \log(1+t) \ge \frac{t}{1+t}\qquad\text{for all }t>0,
\]
which follows from the convexity of $-\log$ or by defining
$\phi(t)=\log(1+t)-\frac{t}{1+t}$ and noting $\phi'(t)=\frac{t}{(1+t)^2}\ge0$ with $\phi(0)=0$.
Hence
\[
  \frac{\sigma(s)\sigma(-s)}{\sqrt{\log(1+t)}}
   = \frac{t}{(1+t)^2 \sqrt{\log(1+t)}}
   \le  \frac{t}{(1+t)^2}\cdot\frac{1}{\sqrt{t/(1+t)}}
   =  \frac{\sqrt{t}}{(1+t)^{3/2}} \leq \frac{2}{3\sqrt{3}}.
\]
Therefore,
\[
  \sigma(s)\sigma(-s)
   \le  \frac{2}{3\sqrt{3}} \sqrt{\log(1+t)}
   =  \frac{2}{3\sqrt{3}} \sqrt{\log\bigl(1+e^{-s}\bigr)}.
\]
\end{proof}

\begin{example}[Logistic function, $L_0=0$]\label{ex:dir-one-app}
Let $a\in\R^d\setminus\{0\}$, $s=\inp{a}{x}$, and $f(x)=\log(1+e^{-s})$. Then $f_\star=0$ and for all $x\in\R^d$,
\[
\frac{\inp{\nabla f(x)}{\nabla^2 f(x) \nabla f(x)}}{\norm{\nabla f(x)}^2}
 \le \frac{2}{3\sqrt{3}} \norm{a}^2 \bigl(f(x)-f_\star\bigr)^{1/2}.
\]
\end{example}

\begin{proof}
Write $\sigma(t)=\frac{1}{1+e^{-t}}$. Then
\[
\nabla f(x)=-a \sigma(-s),\qquad \nabla^2 f(x)=aa^\top \sigma(s)\sigma(-s),
\]
and hence
\[
\frac{\inp{\nabla f}{\nabla^2 f \nabla f}}{\norm{\nabla f}^2}
=\frac{\norm{a}^4 \sigma(s)\sigma(-s)^3}{\norm{a}^2 \sigma(-s)^2}
=\norm{a}^2 \sigma(s)\sigma(-s).
\]

By~\cref{lem:scalar-tight-enough}

\[
\frac{\inp{\nabla f}{\nabla^2 f \nabla f}}{\norm{\nabla f}^2}
\le \frac{2}{3\sqrt{3}} \norm{a}^2 \sqrt{f(x)}=\frac{2}{3\sqrt{3}} \norm{a}^2 \sqrt{f(x)-f_\star},
\]
\end{proof}
where we used that $f_\star=0$.

\begin{lemma}[Zero infimum under linear separability]\label{lem:sep}
If there exists $v\in\R^d$ and $\gamma>0$ such that $y_i\inp{a_i}{v}\ge\gamma$ for all $i$,
then for $f(x)=\frac1n\sum_{i=1}^n \log\bigl(1+e^{-y_i\inp{a_i}{x}}\bigr)$ one has $f_\star=0$.
\end{lemma}

\begin{proof}
For any $\tau>0$,
\[
    f(\tau v)=\frac1n\sum_{i=1}^n \log\bigl(1+e^{-\tau y_i\inp{a_i}{v}}\bigr)
     \le \log\bigl(1+e^{-\tau\gamma}\bigr) \xrightarrow[\tau\to\infty]{} 0.
\]
Hence $\inf_x f(x)=0$.
\end{proof}

\begin{example}[Empirical logistic; $L_0=0$ under separability]\label{ex:dir-emp-app}

Let there exists $v\in\R^d$ and $\gamma>0$ such that $y_i\inp{a_i}{v}\ge\gamma$ for all $i$,
then for $f(x)=\frac1n\sum_{i=1}^n \log\bigl(1+e^{-y_i\inp{a_i}{x}}\bigr)$ one has $f_\star=0$ and for all $x\in\R^d$,
\[
    \frac{\inp{\nabla f(x)}{\nabla^2 f(x) \nabla f(x)}}{\norm{\nabla f(x)}^2}
     \le \frac{2}{3\sqrt{3}} \sqrt{\frac{1}{n}\sum_{i=1}^n \norm{a_i}^4} \bigl(f(x)-f_\star\bigr)^{1/2}.
\]
\end{example}

\begin{proof}
If $\nabla f(x)=0$, the left-hand side is $0$ and the claim holds. Hence assume $\nabla f(x)\neq 0$.

With $\sigma(t)=\frac{1}{1+e^{-t}}$
\[
    \nabla f(x)=-\frac1n\sum_{i=1}^n y_i a_i \sigma(-s_i),
    \qquad
    \nabla^2 f(x)=\frac1n\sum_{i=1}^n w_i a_i a_i^\top,\quad
    w_i:=\sigma(s_i)\sigma(-s_i), \qquad s_i=y_i\inp{a_i}{x}.
\]

Using the quadratic form of $\nabla^2 f(x)$
\begin{align*}
    \frac{\inp{\nabla f(x)}{\nabla^2 f(x) \nabla f(x)}}{\norm{\nabla f(x)}^2}
    &=\frac{\inp{\nabla f(x)}{\left(\frac1n\sum_{i=1}^n w_i a_i a_i^\top\right)\nabla f(x)}}{\norm{\nabla f(x)}^2} 
    \\
    &=\frac{1}{n}\sum_{i=1}^n w_i \frac{\big(\inp{a_i}{\nabla f(x)}\big)^2}{\norm{\nabla f(x)}^2}.
\end{align*}

For each $i$, $(\inp{a_i}{\nabla f(x)})^2\le \norm{a_i}^2 \norm{\nabla f(x)}^2$ by Cauchy--Schwarz implying
\[
    \frac{\inp{\nabla f(x)}{\nabla^2 f(x) \nabla f(x)}}{\norm{\nabla f(x)}^2}
     \le \frac{1}{n}\sum_{i=1}^n w_i \norm{a_i}^2.
\]

By Lemma~\ref{lem:scalar-tight-enough},
\(
w_i=\sigma(s_i)\sigma(-s_i)\le \frac{2}{3\sqrt{3}}\sqrt{\log\bigl(1+e^{-s_i}\bigr)}.
\)
Therefore,
\begin{align*}
    \frac{1}{n}\sum_{i=1}^n w_i \norm{a_i}^2
    &\le \frac{2}{3\sqrt{3}}\cdot \frac{1}{n}\sum_{i=1}^n \norm{a_i}^2 \sqrt{\log\bigl(1+e^{-s_i}\bigr)} 
    \\
    &\le \frac{2}{3\sqrt{3}} 
    \left(\frac{1}{n}\sum_{i=1}^n \norm{a_i}^4\right)^{1/2}
    \left(\frac{1}{n}\sum_{i=1}^n \log\bigl(1+e^{-s_i}\bigr)\right)^{1/2}
    \qquad\text{(Cauchy--Schwarz in $\R^n$)}
    \\
    &= \frac{2}{3\sqrt{3}} 
    \left(\frac{1}{n}\sum_{i=1}^n \norm{a_i}^4\right)^{1/2} \sqrt{f(x)}.
\end{align*}

Under linear separability, $f_\star=0$ (see Lemma~\ref{lem:sep}). Thus $\sqrt{f(x)}=\sqrt{f(x)-f_\star}$, and combining steps above yields
\[
    \frac{\inp{\nabla f(x)}{\nabla^2 f(x) \nabla f(x)}}{\norm{\nabla f(x)}^2}
     \le \frac{2}{3\sqrt{3}} \sqrt{\frac{1}{n}\sum_{i=1}^n \norm{a_i}^4} \sqrt{f(x)-f_\star}.
\]
\end{proof}
\begin{corollary}
    Let~\cref{as:convexity,as:bounded-level-set,as:Lsmooth,as:relaxed} hold. Then the total number of iterations needed to reach $f(x^K)-f_\star\le\varepsilon$ satisfies
    \begin{equation*}
        K = \cO\left(\frac{L_0D^2 + (f_0 - f_*)}{\varepsilon} + \frac{\rbr*{L_1 D^2 +\sqrt{HD^3} + \sqrt{f_0 - f_*}}^2}{\sqrt{\varepsilon}} \right).
    \end{equation*} 
\end{corollary}

\begin{proof}
    We begin with bounding the following term and making it non-positive.
    \begin{equation*}
        \inp{\nabla f(x^k)}{\nabla^2 f(x^k) \nabla f(x^k)} -  \frac{2\sqrt{M_k}}3  \norm{\nabla f(x^k)}^{5/2} \leq \norm{\nabla f(x^k)}^{2} \rbr*{L_0 + L_1 \rbr{f(x) - f_\star}^{1/2} - \frac23 \sqrt{M_k \norm{\nabla f(x^k)}}}.
    \end{equation*}

    \begin{multline*}
        \text{\textit{(i)}   }  L_0 \leq L_1(f(x^k) - f_\star)^{1/2}, \text{ then } \quad L_0 + L_1 \rbr{f(x) - f_\star}^{1/2} - \frac23 \sqrt{M_k \norm{\nabla f(x^k)}}
        \\
        \leq 2L_1 \rbr{f(x) - f_\star}^{1/2} - \frac23 \sqrt{M_k \norm{\nabla f(x^k)}} \leq 2L_1\sqrt{D\norm{\nabla f(x^k)}} - \frac23 \sqrt{M_k \norm{\nabla f(x^k)}}
    \end{multline*}
    
    by inequality \eqref{eq:per-condition} with $\alpha = 0.5$ and $M_k = M = \max{\rbr*{H, 9L_1^2D}}$ yields
     \begin{equation*}
        f\rbr*{x^k - \frac{\nabla f(x^k)}{\sqrt{M\norm{\nabla f(x^k)}}}} 
        \leq 
        f(x^k) -  \frac{1}{3\sqrt{M}} \norm{\nabla f(x^k)}^{3/2}.
    \end{equation*}
    % The other case 
    
    $L_0 \geq L_1(f(x^k) - f_\star)^{1/2}$ reduces to the standard smoothness case with $L = 2L_0$, so that, when
    
    % $ \text{  \textit{(ii)} } L_0 \geq L_1(f(x^k) - f_\star)^{1/2}$ and $(f(x^k) - f_\star)^{1/2} \geq \frac32 \frac{2L_0D^2}{\sqrt{HD^3}}$
    \begin{equation*}
        \text{\textit{(ii)}   } L_0 \ge L_1(f(x^k) - f_\star)^{1/2} \text{ and } (f(x^k) - f_\star)^{1/2} \ge \frac32 \frac{2L_0D^2}{\sqrt{HD^3}}, \text{ then }
        f(x^{k+1}) \leq f(x^k) -  \frac{1}{3\sqrt{HD^3}} \rbr*{f(x^k) - f_\star}^{3/2}
    \end{equation*}
    
    % $L_0 \geq L_1(f(x^k) - f_\star)^{1/2}$ and $(f(x^k) - f_\star)^{1/2} \leq \frac32 \frac{2L_0D^2}{\sqrt{HD^3}}$
    \begin{equation*}
        \text{\textit{(iii)}   } L_0 \ge L_1(f(x^k) - f_\star)^{1/2} \text{ and } (f(x^k) - f_\star)^{1/2} \le \frac32 \frac{2L_0D^2}{\sqrt{HD^3}}, \text{ then }
        f(x^{k+1})
        \leq f(x^k) -  \frac{1}{9L_0D^2} \rbr*{f(x^k) - f_\star}^{2} .
    \end{equation*}

    Then first two cases reduce to 
    \begin{equation*}
        f(x^{k+1}) \leq f(x^k) -  \frac{1}{3\sqrt{\max{\rbr*{H, 9L_1^2D}}D^3}} \rbr*{f(x^k) - f_\star}^{3/2},
    \end{equation*}
    which allows to follow exactly the proof of \cref{thm lsmoothglobal ap} with $H$ replaced with $\max{\rbr*{H, 9L_1^2D}}$ and $L$ with $2L_0$.
\end{proof}

\newpage
\section{Missing proofs for CaCuSGD }

\begin{theorem} 
    Let \cref{as:stoch} holds. Assume that $\widehat L \ge L$ and $\widehat H \ge H$. Then the method
    % \begin{equation*}
    %     x^{k+1} = x^k - \frac{g^k}{\sqrt{M_k\norm{g^k}}}
    % \end{equation*}
    \[
        x^{k+1} = x^k - \frac{g^k}{\sqrt{M_k\norm{g^k}}}, \quad M_k \coloneqq
        \begin{cases}
            \widehat L^2 {\norm{g^k}}^{-1}, & 
            \text{if}\quad {4 \inp{g^k}{H^k g^k}^{2}}{\norm{g^k}^{-5}} \ge \widehat H,\\
            % \dfrac{\widehat L}{\norm*{g^k}}, & \text{otherwise.}
            \widehat H, & 
            \text{otherwise,}
            % \widehat H, & \text{if}\quad \tfrac{4\inp*{g^k}{H^k g^k}^{2}}{\norm*{g^k}^{5}} \le H,\\
        \end{cases} \quad \text{satisfies}
    \]
    % where $M_k = \widehat H$ if $\frac{4\inp*{g^k}{H^k g^k}^2 }{\norm{g^k}^5} \le H$ and $M_k = \frac{\widehat L}{\norm{g^k}}$
    % satisfy
    \begin{align*}
        &\text{ (i) } \E f(x^{k+1}) \le \E f(x^k)
        -\frac{1}{4\widehat L}\E\|\nabla f(x^k)\|^2
        +\Big(\frac{2L}{\widehat L^2}+\frac{2L^2}{3\widehat L^3}\Big) \sigma_g^2
        +\frac{\sigma_H^3}{6 \widehat H^2},  
        &\text{if}\quad {4 \inp{g^k}{H^k g^k}^{2}}{\norm{g^k}^{-5}} \ge \widehat H,
        \\
        &\text{(ii) } \E f(x^{k+1})
        \le
        \E f(x^k)
        -\frac{1}{3\sqrt{\widehat H}} \E\sbr*{\|\nabla f(x^k)\|^{3/2}}
        +\frac{7}{4\sqrt{\widehat H}}\sigma_g^{3/2}
        +\frac{32}{3 \widehat H^2} \sigma_H^{3}, 
        &\text{otherwise}.
    \end{align*}
\end{theorem}
\begin{proof}
    For the sake of simplicity we write $g^k = g(x^k, \xi_k)$  and $H^k = H(x^k, \xi_k)$
    
    There exist $M_k$ is s.t.
     \begin{multline*}
        f_k\rbr*{x^k - \frac{g^k}{\sqrt{M_k\norm{g^k}}}} 
        \leq
        f_k(x^k)
        + \frac{1}{2} \sbr*{ \frac{\inp{g^k}{H^k g^k}}{M_k \norm{g^k}} -  \frac{1}{\sqrt{M_k}} \norm{g^k}^{3/2}}-  \frac{1}{6\sqrt{M_k}} \norm{g^k}^{3/2}
        % \\
        \leq
        -  \frac{5}{12\sqrt{M_k}} \norm{g^k}^{3/2},
    \end{multline*}
    i.e.  $M_k \geq \max{\rbr*{H, \frac{4\inp*{g^k}{H^k g^k}^2 }{\norm{g^k}^5}}}$. Next we consider the cases.

    \textbf{Case} $ \frac{4\inp*{g^k}{H^k g^k}^2 }{\norm{g^k}^5} \le \widehat H.$ Then $M_k = H$ and 
    \begin{equation*}
        - \langle g^k, s^k\rangle + \frac12 \langle s^k, H^k s^k\rangle + \frac{M_k}{3}\|s^k\|^3
        \le
        - \frac{5}{12\sqrt{M_k}} \|g^k\|^{3/2},
    \end{equation*}
    where $s^k = \frac{g^k}{\sqrt{M_k \norm{g^k}}}$. Then
    
    \begin{multline*}
        f(x^{k+1})
        \le
        f(x^k)
        - \inp{\nabla f(x^k)}{s^k}
        +\frac{1}{2} \inp{s^k}{\nabla^2 f(x^k) s^k}
        + \frac{H}{3}\norm{s^k}^3
        \\
        =
        f(x^k)
        - \inp{\nabla f(x^k) \pm g^k}{s^k}
        +\frac{1}{2} \inp*{s^k}{\nabla^2 f(x^k) s^k \pm H^k s^k}
        + \frac{H}{3}\norm{s^k}^3
        \\
        = 
        f(x^k)
        - \inp{g^k}{s^k} + \inp{s^k}{H^k s^k} + \frac{H}{3}\norm{s^k}^3
        \\
        - \inp{\nabla f(x^k) - g^k}{s^k}
        +\frac{1}{2} \frac{\inp*{g^k}{\rbr*{\nabla^2 f(x^k) - H^k} g^k}}{H\norm{g^k}}.
    \end{multline*}

    Next, one applies Young's inequality ($p=3, q=3/2$):
    \begin{equation*}
        u v
       \le \frac13 \alpha^3 u^3 + \frac23 \alpha^{-3/2}v^{3/2}.
    \end{equation*}
    with $u = \|s^k\|$, $v = \|\nabla f(x^k) - g^k\|$, $ \alpha^3 = H/8$, and
    
    \begin{equation*}
        - \langle \nabla f(x^k) - g^k, s^k\rangle
        \le \frac{1}{24\sqrt{H}} \|g^k\|^{3/2}
          + \frac{4\sqrt{2}}{3}\frac{1}{\sqrt{H}}\|\nabla f(x^k) - g^k\|^{3/2}.
    \end{equation*}
    and with 
    $u = \norm{\nabla^2 f(x^k) - H^k}$, $v = \frac{\norm{g^k}}{2 H}$, $\alpha = 2^{5/3} H^{-2/3}$
    for
    \[
        \frac{1}{2} \frac{\inp*{g^k}{(\nabla^2 f(x^k) - H^k) g^k}}{H \norm{g^k}}
        \le
        \frac{1}{2 H} \norm{\nabla^2 f(x^k) - H^k} \norm{g^k}.
    \]
    The latter gives
    \[
        \frac{1}{2} \frac{\inp*{g^k}{(\nabla^2 f(x^k) - H^k) g^k}}{H \norm{g^k}}
        \le
        \frac{1}{24 \sqrt{H}} \norm{g^k}^{3/2}
        +
        \frac{32}{3 H^{2}} \norm{\nabla^2 f(x^k) - H^k}^{3}.
    \]

    Also
    \begin{equation*}
        \|\nabla f(x^k)\|^{3/2} \le \sqrt{2}(\|g^k\|^{3/2} + \|g^k-\nabla f(x^k)\|^{3/2})
    \end{equation*}
    implying that
    \begin{equation*}
          -\|g^k\|^{3/2} \le -\frac{1}{\sqrt{2}} \|\nabla f(x^k)\|^{3/2} + \|g^k-\nabla f(x^k)\|^{3/2}
    \end{equation*}

    Putting the pieces together we obtain that
    \begin{equation*}
        f(x^{k+1})
        \le
        f(x^k)
        -\frac{1}{3\sqrt{H}} \|\nabla f(x^k)\|^{3/2}
        +\frac{1+4\sqrt{2}}{3\sqrt{2H}}\|g^k-\nabla f(x^k)\|^{3/2}
        +\frac{32}{3H^2} \norm{\nabla^2 f(x^k) - H^k}^{3}.
    \end{equation*}
    
    Taking expectation we have    
    \begin{equation*}
        \E f(x^{k+1})
        \le
        \E f(x^k)
        -\frac{1}{3\sqrt{H}} \E\sbr*{\|\nabla f(x^k)\|^{3/2}}
        +\frac{1+4\sqrt{2}}{3\sqrt{2H}}\sigma_g^{3/2}
        +\frac{32}{3H^2} \sigma_H^{3}.
    \end{equation*}

    \textbf{Case} $\frac{4\inp*{g^k}{H^k g^k}^2 }{\norm{g^k}^5} > \widehat H$. Then $M_k = \frac{\widehat{L}^2 }{\norm{g^k}}$.
    We assume that there exists a constatnt $L$ s.t.
    \begin{equation*}
        L_k \eqdef \frac{\inp*{g^k}{H^k g^k}}{\norm{g^k}^4} \le L,
    \end{equation*}
    then
    \begin{equation*}
        s^k = \frac{g^k}{\sqrt{M_k \norm{g^k}}}
        = \frac{g^k}{2 \widehat{L}}.
    \end{equation*}
    
    \begin{multline*}
        f(x^{k+1})
        \le
        f(x^k)
        - \frac1{2\widehat{L}}\inp{\nabla f(x^k)}{g^k}
        + \frac{1}{8\widehat{L}^2} \inp{g^k}{\nabla^2 f(x^k) g^k}
        + \frac{H}{24\widehat{L}^3}\norm{g^k}^3
        \\
        \leq
        f(x^k)
        - \frac1{2\widehat{L}}\inp{\nabla f(x^k)}{g^k}
        + \frac{1}{8\widehat{L}^2} \inp{g^k}{(\nabla^2 f(x^k) \pm H^k) g^k}
        + \frac{H}{24\widehat{L}^3}\norm{g^k}^3
        \\
        =
        f(x^k)
        - \frac1{8\widehat{L}}\inp{\nabla f(x^k)}{g^k}
        + \frac{L_k}{8\widehat{L}^2}\norm{g^k}^2
        + \frac{H}{24\widehat{L}^3}\norm{g^k}^3
        \\
        + \frac{1}{8\widehat{L}^2} \inp{g^k}{(\nabla^2 f(x^k) - H^k) g^k}
    \end{multline*}

    Next we use that
    \[
        \frac{1}{8\widehat{L}^2}\langle g^k,(\nabla^2 f(x^k)-H^k)g^k\rangle
        \le
        \frac{1}{8\widehat{L}^2}\|g^k\|^2\|\nabla^2 f(x^k)-H^k\|,
    \]
    apply Young's ($p=3, q=3/2$):
    \[
        ab \le \frac{\alpha^3}{3}a^3 + \frac{2}{3\alpha^{3/2}}b^{3/2}.
    \]
    with
    $a = \|\nabla^2 f(x^k)-H^k\|$, $b = \frac{\|g^k\|^2}{2\widehat{L}^2}$, $\alpha = 2^{-1/3} H^{-2/3}$ and obtain that
    \[
        \frac{1}{8\widehat{L}^2}\|g^k\|^2\|\nabla^2 f(x^k)-H^k\|
        \le
        \frac{1}{6H^2}\|\nabla^2 f(x^k)-H^k\|^3
        +\frac{H}{24\widehat{L}^3}\|g^k\|^3.
    \]
    
    Thus we have
    \begin{multline*}
        f(x^{k+1})
        \le
        f(x^k)
        - \frac1{2 \widehat{L}}\inp{\nabla f(x^k)}{g^k}
        + \frac{L_k}{\widehat{L}^2}\norm{g^k}^2
        + \frac{2H}{24\widehat{L}^3}\norm{g^k}^3
        \\
        +\frac{1}{6H^2}\|\nabla^2 f(x^k)-H^k\|^3
        \\
        \le
        f(x^k)
        - \frac1{2 \widehat{L}}\inp{\nabla f(x^k)}{g^k}
        + \frac{L_k}{8\widehat{L}^2}\norm{g^k}^2
        + \frac{8L_k^2}{24\widehat{L}^3}\norm{g^k}^2
        \\
        +\frac{1}{6H^2}\|\nabla^2 f(x^k)-H^k\|^3
        \\
        \le
        f(x^k)
        - \frac1{2\widehat{L}}\inp{\nabla f(x^k)}{g^k}
        + \frac{L_k}{4\widehat{L}^2}\norm{\nabla f(x^k)}^2
        + \frac{2L_k^2}{3\widehat{L}^3}\norm{\nabla f(x^k)}^2
        \\
        + \frac{L_k}{4\widehat{L}^2}\norm{\nabla f(x^k) - g^k}^2
        + \frac{2L_k^2}{3\widehat{L}^3}\norm{\nabla f(x^k) - g^k}^2+\frac{1}{6H^2}\|\nabla^2 f(x^k)-H^k\|^3.
    \end{multline*}

    Take conditional expectation
    \[
        \E[f(x^{k+1})\mid x^k]
        \le f(x^k)
        -\Big(\frac{1}{2\widehat L}-\frac{L}{4\widehat L^2}-\frac{2L^2}{3\widehat L^3}\Big) \|\nabla f(x^k)\|^2
        +\Big(\frac{L}{4\widehat L^2}+\frac{2L^2}{3\widehat L^3}\Big) \sigma_g^2
        +\frac{\sigma_H^3}{6H^2}.
    \]

    Setting $\widehat{L} > 3L$
    \[
        \E[f(x^{k+1})\mid x^k]
        \le f(x^k)
        -\frac{1}{4\widehat L}\|\nabla f(x^k)\|^2
        +\Big(\frac{L}{4\widehat L^2}+\frac{2L^2}{3\widehat L^3}\Big) \sigma_g^2
        +\frac{\sigma_H^3}{6H^2}.
    \]
    
    And finally, the full expectation
    \[
        \E f(x^{k+1})
        \le \E f(x^k)
        -\frac{1}{4\widehat L}\E\|\nabla f(x^k)\|^2
        +\Big(\frac{2L}{\widehat L^2}+\frac{2L^2}{3\widehat L^3}\Big) \sigma_g^2
        +\frac{\sigma_H^3}{6H^2}
    \]
    concludes the proof.
\end{proof}

\end{document}